\documentclass{amsart}
\usepackage[a4paper,hmargin=2cm,vmargin=3cm]{geometry}
\usepackage[utf8]{inputenc}
\usepackage{xcolor}
\usepackage{amsmath,bm,bbm}
\usepackage{amsthm}
\usepackage{amssymb}
\usepackage{graphicx}
\usepackage[shortlabels]{enumitem}
\usepackage{hyperref}
\hypersetup{
	colorlinks,
	linkcolor={red!50!black},
	citecolor={green!50!black},
	urlcolor={blue!80!black}
}
\usepackage{amsaddr}
\usepackage{appendix}
\usepackage{dsfont}

\numberwithin{equation}{section}
\newtheorem{theo}{Theorem}[section]%

\newtheorem{defi}[theo]{Definition}

\newtheorem{assum}[theo]{Assumption}%
\newtheorem{lemma}[theo]{Lemma}%
\newtheorem{rem}[theo]{Remark}%
\newtheorem{prop}[theo]{Proposition}%
\newtheorem{Ex}[theo]{Example}%

\def\tr{{\rm Tr}}
\newcommand{\diag}{\mathrm{diag}}

\newcommand{\E}{\mathbb E}

\newcommand{\Pp}{\mathbb P}
\newcommand{\C}{\mathbb C}
\newcommand{\R}{\mathbb R}
\newcommand{\N}{\mathbb N}

\newcommand{\Ss}{\mathbb S}

\newcommand{\bs}{\bm{\sigma}}

\DeclareMathOperator{\1}{\mathbbm{1}}

\begin{document}
	
	\title{Spherical Integrals of Sublinear Rank}
\author{Jonathan Husson* and Justin Ko*}

\email{jhusson@umich.edu}
\email{justin.ko@ens-lyon.fr}
	\thanks{* Supported in part by ERC Project LDRAM:  ERC-2019-ADG	Project 884584.}

\begin{abstract}
We consider the asymptotics of $k$-dimensional spherical integrals when $k = o(N)$. We prove that the $o(N)$-dimensional spherical integrals are approximately the products of $1$-dimensional spherical integrals. Our formulas extend the results for $k$-dimensional spherical integrals proved by Guionnet and Maïda in \cite{GuMa05} and  Husson and Guionnet in \cite{GuHu21} which are only valid for $k$ finite and independent of $N$. These approximations will be used to prove a large deviation principle for the joint $2k(N)$ extreme eigenvalues for sharp sub-Gaussian Wigner matrices and for additive deformations of GOE/GUE matrices. Furthermore, our results will be used to compute the free energies of spherical SK vector spin glasses and the mutual information for matrix estimation problems when the dimensions of the spins or signals have sublinear growth.  
\end{abstract}
	
	\maketitle
	\section{Introduction}
	
	The Harish--Chandra--Itzykson--Zuber integral was first introduced by Harish--Chandra as the following integral on the orthogonal group  or unitary group: 
	\begin{equation}\label{eq:HCIZ}
		HCIZ(A,B) = \int_{\mathcal{U}_N} \exp(N \tr(AUBU^*)) dU 
	\end{equation}
	where $A,B$ are two self-adjoint $N \times N$ matrices, $\mathcal{U}_N$ is either the unitary group or the orthogonal group and $dU$ is the Haar measure on it. This integral can be thought of as a way to generalize the Laplace transform on the orthogonal and unitary group \cite{Harish}. An explicit formula was given in the unitary case by Itzykson and Zuber \cite{Itzyksonzuber} and Harish-Chandra: 
	\[ HCIZ(A,B) = \frac{ \det( ( e^{N \lambda_i \mu_j})_{ 1 \leq i, j \leq N})}{\Delta(A) \Delta(B)} \]
	where $\lambda_1, \dots, \lambda_N$ are the eigenvalues of $A$ and $\mu_1,\dots,\mu_N$ are the eigenvalue of $B$ and $\Delta(A) = \prod_{i>j} |\lambda_i - \lambda_j|$. It is a powerful and well studied object in a variety of fields from algebraic geometry to physics. In random matrix theory, results by Coquereaux, McSwiggen and Zuber \cite{CoMcSZu} and Zuber \cite{Zu} use these spherical integrals to express the density of the eigenvalues for matrix models of the form $A +UBU^*$ where $A$ and $B$ are deterministic self-adjoint matrices and $U$ is an Haar distributed random matrix in the orthogonal or the unitary group. For questions of large deviations, knowing an equivalent of $\ln HCIZ(A_N,B_N)$ depending on the behavior of the spectra of $A_N$ and $B_N$ can help prove large deviation principles for the largest eigenvalue. For instance, one can refer to \cite{HuGu,Hu,AuGuHu,McKenna,GuHu21} for large deviation principles for the largest eigenvalue of matrices with entries that satisfy a sharp sub-Gaussian bound (see Definition ~\ref{sharpsubG}), \cite{Mai} for the largest eigenvalue of an additive deformation of a GOE/GUE matrix, \cite{GuMa20} for the largest eigenvalue of the sum of two random matrices, and \cite{BGM} for a large deviation principle of the empirical measure of diagonal entries of a unitary invariant matrix. All these results hinge on the asymptotic behavior of the logarithm of the spherical integral either for $B_N$ with finite rank $k$ when we are interested in the $k$ largest eigenvalue or $B_N$ of full rank when we are interested in the empirical mesure.  When the eigenvalue distribution of $A_N$ and $B_N$ converge, Guionnet and Zeitouni investigated the limit of $N^{-2} \ln HCIZ(A_N,B_N)$ \cite{GZ3,GZei1add}. In the case where the rank of $B_N$ is one, more precisely when $B_N = \theta e e^*$ where $e$ is some unitary vector and $\theta$ is a real number (that does not depend on $N$), the limit $N^{-1} \ln HCIZ(A_N, B_N)$  was determined by Maïda and Guionnet \cite{GuMa05} (see also \cite{GoPa}). If $\theta > 0$, this limit depends on the limit of the largest eigenvalue of $A_N$ and the limit of the eigenvalue distribution of $A_N$. More precisely, assuming that both those quantities converge toward respectively toward $\lambda$ and $\mu$, we have that: 
	\[ \lim_{N \to \infty}\frac{1}{N} \ln HCIZ\bigg(A_N,\frac{\beta}{2} \theta e e^* \bigg) = \frac{\beta}{2}J(\theta, \lambda, \mu) \]
	where $J$ is defined by using $G_{\mu}$,  the Stieltjes transform of $\mu$ and $G_{\mu}^{-1}$, its inverse function as follows:
	\begin{equation}\label{eq:defJ}
		J(\theta, \lambda, \mu) = \theta \lambda + (v - \lambda) G_{\mu}(v) - \ln| \theta| - \int \ln | v - x | d \mu(x) - 1,
	\end{equation}
	where 
	\[ v = \begin{cases} \lambda &\text { when } 0 \leq G_{\mu}(\lambda) \leq \theta \text{ or } \theta \leq G_{\mu}(\lambda) \leq 0 \\
		G_{\mu}^{-1}(\theta) &\text{ otherwise. } 
	\end{cases}
	\]
	
	This result was generalized to $B_N$ of finite rank by Guionnet and one of the authors in \cite{GuHu21}. If $\theta^+ _1 \geq \theta^+_2 \geq \dots \geq \theta_l^+ \geq 0$ and $\theta^- _1 \leq \theta^-_2 \leq \dots \leq \theta_m^- \leq 0$ and $A_N$ is a sequence of (deterministic) matrices such that for $i \leq l$, the $i$-th largest eigenvalue converges toward $\lambda^+_i$, that for $j \leq m$, the $j$-th smallest eigenvalue converges toward $\lambda^-_j$ and that the eigenvalue distribution of $A_N$ converges toward $\mu$ then, if $B_N =\frac{\beta}{2} \Big[ \sum_{i=1}^l \theta^+_i e_i e_i^* +\sum_{i=1}^m \theta^-_i f_i f_i^* \Big]$ where $ \{ e_i \}_{1 \leq l} \cup \{ f_i \}_{1 \leq m}$ is a family of orthonormal vectors, we have that:
	
	\[ \lim_{N \to + \infty} \frac{1}{N} \ln HCIZ(A_N,B_N) = \frac{\beta}{2} \Big[ \sum_{i=1}^l J( \theta_i^+, \lambda_i^+, \mu) + \sum_{i=1}^m J( \theta_i^-, \lambda_i^-, \mu)  \Big]\]
	where $J$ is given by \eqref{eq:defJ}.
	At the limit, there is an additivity phenomenon where we pair each parameter $\theta$ to a corresponding eigenvalue of $A_N$. Up to this pairing the asymptotical behavior of the integral is similar to the sum of the behavior of rank one integral.

	From this result, one can make a conjecture regarding the behavior of the same integral where the rank of the matrix $B_N$ is negligible relative to $N$. More precisely, if $B_N =\sum_{i=1}^{l(N)} \theta^+_i e_i e_i^* +\sum_{i=1}^{m(N)} \theta^-_i f_i f_i^*$ is a sequence of matrices bounded in operator norm with $\theta^+_1 \geq \dots \geq \theta^+_{l(N)} \geq 0$ and $\theta^-_1 \leq \dots \leq \theta^-_{m(N)} \leq 0$ and $\{ e_i \}_{1 \leq l(N)} \cup \{ f_i \}_{1 \leq m(N)} $ a family of orthonormal vectors, and $(A_N)_{N \in \N}$ is a sequence of matrices bounded in operator norm then:
	\begin{equation}\label{eq:conjecture}
	\frac{1}{N} \ln HCIZ(A_N,B_N) =  \frac{\beta}{2} \Big[\sum_{i=1}^{l(N)} J( \theta_i^+, \lambda_i^+, \mu) + \sum_{i=1}^{m(N)} J( \theta_i^-, \lambda_i^-, \mu) \Big] + o(l(N) + m(N))
	\end{equation}
	where the $\lambda_i^+$ are the $i$-th largest eigenvalue of $A_N$ and $\lambda_i^-$ the $i$-th lowest one. 
	
	In this mesoscopic case,  Guionnet and Maïda investigated the case when $k(N) = l(N) + m(N) = o( N^{-1/2 -\epsilon})$ and  $\theta_i$ below the transition threshold \cite{GuMa05} and Collins and Sniady investigated the case where the extremal eigenvalues stick to the edges of the limit measures \cite{CoSn}. Huang also provides in \cite{Huang} and expansion of such integrals again for small values of $\theta_i$. Note than none of those cases dealing with a non-constant $k$ actually exhibits the pairing phenomenon of $\lambda_i^{\pm}$ with $\theta_{i}^{\pm}$ since in these cases $J(\theta, \lambda, \mu)$ does not actually depend on the value of $\lambda$.
	In this paper we will prove the conjecture stated on equation \eqref{eq:conjecture} for any sequence $k(N)$ such that $k(N)=o(N)$. We will also generalize the large deviation results of \cite{HuGu} for the largest eigenvalue of sharp sub-Gaussian random matrices and \cite{Mai} for the largest eigenvalue of an additive deformation of a GOE/GUE matrix to the joint large deviations of the $k(N)$ largest eigenvalues. 
	
	These growing rank spherical integrals also have applications in spin glasses. The spherical $2$-spin models have deep connections with random matrix theory because the Hamiltonians can be expressed as quadratic forms of a GOE matrix. This spherical model was introduced in \cite{sphericalskoriginal} as a variant of the Ising spin Sherrington--Kirkpatrick model introduced in \cite{SK}. A generalized form of this model called the mixed $p$-spin model and the analogue of the Parisi formula \cite{parisi1979infinite,parisi1980sequence} for the free energy of this model was discovered by Crisanti and Sommers in \cite{crisanti1992sphericalp} and was proven rigorously in \cite{TSPHERE,CASS}. Because of the simple structure of the Hamiltonian in the spherical case, the computation of the free energy is closely tied to the behavior of the eigenvalues of a GOE matrix, which has been the studied extensively in random matrices.  Random matrix techniques have been applied to study the fluctuations of the free energy and corresponding phase transtions in \cite{BaikFluct, BaikFerro, BaikExternal}, the connection the large deviations of the top eigenvalue in \cite{mergny2022right}, and the marginals of spherical spin glasses with correlated disorder matrices in \cite{barbier_correlated}.
	
	In this paper, we provide another application of random matrix tools to tackle a high dimensional analogue of the spherical SK model called the vector spin model. The analogue of the Cristanti--Sommers and Parisi formula for the limit of the free energy of this model was proved \cite{PTSPHERE, kosphere, kocs}. The derivation of this formula for the vector spin free energy  for the spherical SK model used  standard techniques in spin glasses such as interpolation \cite{guerra2003broken}, the cavity method \cite{AS2, CASS}, ultrametricity \cite{physicsultrametricity1, physicsultrametricity2, PUltra}, or sychronization \cite{PPotts, PVS}. The spherical integrals can be used as a direct large deviations proof of the limit of the free energy. Our formula will allow us to compute the limit of vector spin free energies for $2$ spin models when the dimension of the vector spins grow sublinearly with $N$.

	Lastly, we state an application of the spherical integrals in statistical inference. There has been a lot of interest and rigorous results in a class of statistical inference problems called the matrix factorization problems \cite{krzakala2016mutual,dia2016mutual,miolanefundamentallimits,el2018estimation,barbier2019adaptive,el2020fundamental}. Fundamental limits of the finite rank matrix factorization problems, which involves estimating a low rank signal in the presence of a Gaussian noise matrix, were proved by adapting techniques in spin glasses in \cite{miolanefundamentallimits}. Techniques to study the extensive rank problems, when the rank of the signal is on the same order as the dimension of the noise matrix, was studied  recently by physicists in \cite{maillard2021perturbative,troiani2022optimal,barbier2021statistical}. An interest lies in seeing what happens at the interface of these two regimes. In \cite{ReevesMatrixTensor}, the limit of the mutual information was computed under some additional assumptions on the sublinear growth rate of the dimension. In this work, we will use the spherical integrals to derive explicit formulas in the case when the rank of the matrix factorization problems go to $\infty$ at a sublinear rate, without any additional restrictions on the sublinear growth rate. See also \cite{MourratTensor1, MourratTensor2,Camilli2022AnIP} for other rigorous works related to the mutual information of matrix factorization problems.
	
	\section{Notation and main results}
	In the following, we will denote $\mathcal{H}_N^{\beta}$ the set of $N \times N$  symmetric matrices if $\beta=1$ and $N \times N$ Hermitian matrices if $\beta=2$. We will denote $\mathcal{U}_N^{\beta}$ the orthogonal group of $\R^N$ if $\beta=1$ and the unitary group of $\C^N$ if $\beta =2$. In the rest of the paper, $\beta$ will be fixed.  For a real or complex matrix $M $ we will denote 
		\[ |||M||| := \sup_{ u \neq 0} \frac{||Mu||_2}{||u||_2}  \]
	its operator norm and 
	\[ ||M||_2 = \sqrt{ \sum_{i,j} |M(i,j)|^2} \]
	its Euclidian norm. We will also denote by $B(M,r)$ the ball with center $M$ and radius $r$ under the norm $|||\cdot|||$ and $B_2(M,r)$ for the ball of center $M$ and radius $r$ for the norm $||\cdot||_2$

	 We will also denote for any integer $k$, $I_k$ the $k \times k$ identity matrix. Finally, if $A,B \in \mathcal{H}_N^{\beta}$, we will denote $A \leq B$ to mean that $B-A$ is a positive matrix.

	Let $(k(N))_{N \in \N}$ be a sequence of integers such that $\lim_{N \to + \infty} k(N)/N =0$, $(l(N))_{N \in \N}$ a sequence of integers such that $l(N) \leq k(N)$.	We will also manipulate sequence of matrices $(A_N)_{N \in \N} $ and $(D_N)_{N \in \N}$ as parameters for our spherical integral. The following assumption will be needed to state our main result:
	\begin{assum} \label{assum:A1}
		We assume that $(A_N)_{N \in \N}$ and $(D_N)_{N \in \N}$ are two sequences of matrices such that: 
		\begin{enumerate}
			\item $\forall N \in \N, A_N, D_N \in \mathcal{H}_N^{\beta}$. 
			\item There exists $K > 0$ such that for all $N \in \N$, $|||A_N|||, |||D_N||| \leq K$.
			\item For all $N \in \N$, the signature $(p,n, N - p - n)$ of $D_N$ is such that $p + n \leq k(N)$, $n \leq l(N)$. In other words, $D_N$ has  a rank at most $k(N)$ and has at most $ l(N)$ negative eigenvalues. 
		\end{enumerate}
		If these conditions are met, for every $N$, we will denote $\lambda_1 \geq \dots \geq \lambda_N$ the eigenvalues (with multiplicity) of $A_N$ and $\theta_1 \geq \dots \geq \theta_{k(N)}$ the reals (which may be $0$) such that the spectrum of $D_N$ (with multiplicity) is the $\theta_i$ together with $0$ with multiplicity $N - k(N)$. Lastly we will assume that $\theta_{l(N)} \geq 0 \geq \theta_{l(N) +1}$. So, for a fixed $i$, $\lambda_i$ and $\theta_i$ depend implicitly on $N$ but for the sake of concision we omit this dependency in the notation.

	 Lastly, we assume that the eigenvalue distribution of $A_N$, $\mu_{A_N} = \frac{1}{N} \sum_{i=1}^N \delta_{\lambda_i}$ converges weakly toward a compactly supported measure $\mu$.
		
	\end{assum}

	We now define our spherical integral:
	\begin{defi}\label{def:sphericalint}
		For $N \in \N$, and two matrices $A$ and $B$ in $\mathcal{H}_N^{\beta}$, we will denote by $I_N(A,B)$ the following integral on the orthogonal/unitary group $\mathcal{U}_N^{\beta}$: 
		\[I_N(A,B) = \int \exp\Big( \frac{\beta N}{2}  \tr(AUBU^* ) \Big) dU \]
		where $dU$ is the Haar measure.

		Furthermore, we extend this definition for matrices whose dimensions are smaller than $N$. If $A$ is in $\mathcal{H}_L^{\beta}$ and $B \in\mathcal{H}_M^{\beta} $ with $L,M \leq N$, we denote $A'$ and $B'$ the $N \times N$ matrices whose topleft submatrix is respectively $A$ and $B$ and whose remaining entries are $0$. Then we define \[I_N(A,B):= I_N(A',B').\]
	\end{defi}
	In the rest of the paper, we are often going to consider matrices $A_N$ of (potentially) full rank and matrices $D_N$ whose rank is going to be small relative to $N$. The eigenvalues of $D_N$ will sometimes be called \emph{inverse temperatures} by analogy to physics in the case where $D_N$ is of rank $1$. 
	We also define the function $J$ which is going to govern the asymptotics of $I_N$: 
	\begin{defi}\label{def:J}
		Let $\mu$ be a compactly supported measure  on $\R$, $\theta \geq 0$ and $\lambda \in \R$. We denote $supp(\mu)$ the support of $\mu$, and $r(\mu), l(\mu)$ denotes respectively the rightmost and leftmost points of $supp(\mu)$. We let
		$G_{\mu}$ denote the Stieltjes transform of $\mu$ defined for $z \in \C \setminus supp(\mu)$ by: 
		\[ G_{\mu}(z) = \int \frac{1}{z - x} d \mu(x). \] 
		For $z = r(\mu), l(\mu)$, we define 
		\[ G_{\mu}(r(\mu)) = \lim_{ z \to r(\mu)^+ \atop z \in \R} G_{\mu}(z) \quad\text{and}\quad G_{\mu}(l(\mu)) = \lim_{ z \to \ell(\mu)^- \atop z \in \R} G_{\mu}(z)\]
		so that $G_{\mu}(r(\mu)), G_{\mu}(l(\mu))$ may be infinite. We define on $[ G_{\mu}( l(\mu)), G_{\mu}( r(\mu))] \setminus \{ 0 \}$  the inverse function $G_{\mu}^{-1}$ of $G_{\mu}$.
	
		We define $J(\theta, \lambda, \mu)$ by  
		\begin{equation}
			J(\theta, \lambda, \mu) = \theta \lambda ' + (v - \lambda') G_{\mu}(v) - \ln| \theta| - \int \ln | v - x | d \mu(x) - 1 
		\end{equation}
		where $\lambda' = \max(\lambda,r(\mu))$
		\[ v:= v(\lambda,\theta) = \begin{cases} \lambda' &\text { when } 0 \leq G_{\mu}(\lambda') \leq \theta \text{ or } \theta \leq G_{\mu}(\lambda') \leq 0 \\
			G_{\mu}^{-1}(\theta) &\text{ otherwise } .
		\end{cases}
		\]
	\end{defi}

	The aysmptotics for the spherical integrals when $k(N) = 1$ was proven in \cite{GuMa05} and  was extended to finite dimensional $k(N) = k$ in \cite{HuGu} . We remind the readers of the main results here.     
	\begin{theo}\label{theo:rank1}
		Suppose $(A_N)_{N \in \N}$ and $(D_N)_{N \in \N}$ are two sequences of matrices that satisfy Assumption~\ref{assum:A1}. We have
		\begin{enumerate}
			\item If $k(N)=1$ for all $N$: 
			\[ \lim_{N \to \infty} \Big| \frac{2}{\beta N} \ln I_N(D_N,A_N) -  \mathds{1}_{\theta_1 \geq 0} J(\theta_1,\lambda_1, \mu) - \mathds{1}_{\theta_1 < 0} J(\theta_1,\lambda_N, \mu)  \Big| = 0. \]
			\item If $k(N)=k$ and $l(N) = l$ for all $N$: 
			\[ \lim_{N \to \infty} \left| \frac{2}{\beta N} \ln I_N(D_N,A_N) - \left[ \sum_{i=1}^{l} J(\theta_i, \lambda_i,\mu ) +\sum_{i=1}^{k- l} J(\theta_{l +i}, \lambda_{ N + i -k + l},\mu ) \right]  \right| = 0. \]
		\end{enumerate} 
	\end{theo}
	In the first scenario of Theorem~\ref{theo:rank1}, using the boundedness on $\theta_1$ and $\lambda_1$ and the continuity of $J(\theta, \lambda,\mu)$ in $(\theta, \lambda)$, one can see that this is equivalent to showing that if $\lambda_1$ converges to $\lambda$ and $\theta_1$ to $\theta > 0$, or if $\lambda_N$ converges to $\lambda$ and $\theta_1$ to $\theta <0$, then 
	\[ \lim_{N \to \infty} \frac{2}{\beta N} \ln I_N(D_N,A_N) = J(\theta,\lambda, \mu). \]
	This result was first proved in \cite{GuMa05} under the stronger hypothesis that $d(\frac{1}{N}\sum_{i=1}^N \delta_{\lambda_i}, \mu) \leq N^{-\kappa}$. This hypothesis is relaxed in \cite{GuHu21}.

	The main result of this paper is the extension of Theorem~\ref{theo:rank1} to the setting where $k(N) = o(N)$. 
	
	\begin{theo}\label{maintheo}
		If $(A_N)_{N \in \N}$ and $(D_N)_{N \in \N}$ are two sequences of self-adjoint matrices that satisfy Assumption~\ref{assum:A1}, then: 
		\[\lim_{N \to \infty} \left| \frac{2}{\beta k(N) N} \ln I_N(D_N,A_N) - \left[ \frac{1}{k(N)} \sum_{i=1}^{l(N)} J(\theta_i, \lambda_i,\mu ) + \frac{1}{k(N)} \sum_{i=1}^{k(N)- l(N)} J(\theta_{l(N) +i}, \lambda_{ N + i -k(N) + l(N)},\mu ) \right]\right| = 0. \]
	\end{theo}
	
	\subsection{Applications of the Main Result}
	
	The asymptotics of the growing rank spherical integrals proved in Theorem~\ref{maintheo} has several applications in random matrix theory and statistical physics which we now summarize. 
	
	\subsubsection{Large Deviations of Extremal Empirical Measures}
	
	In Section~\ref{sec:largesteigenvalues}  we prove a large deviations principle for the joint $k(N)$ largest and lowest eigenvalues of a Wigner matrix with sharp sub-Gaussian coefficients in Theorem~\ref{LDPTheo}. We remind here that a random variable $X$ with values in $\R^k$ is said to be sharp sub-Gaussian if for every $t \in \R^k$:
	\[ \E[ \exp( \langle t, X \rangle )] \leq \exp ( \langle t, Cov(X) t \rangle /2). \]
	
	For a random matrix $X_N$, we will capture the behavior of its $k(N)$ largest and smallest eigenvalues through the following ``extremal empirical measure"
	\begin{equation}\label{eq:extremalempirical}
	 \hat{\nu}_{k(N)}(X_N) = \frac{1}{2 k(N)} \Big( \sum_{i=1}^{k(N)} \delta_{\lambda_i(X_N)} + \delta_{\lambda_{N-i+1}(X_N)} \Big)
	\end{equation}
	where $\lambda_1(X_N) \leq \lambda_2(X_N) \leq \dots \leq \lambda_N(X_N)$  are the eigenvalues of $X_N$. 
	
	To be stated, the result will also need two other assumptions, one stating that the empirical measure of $X_N$ concentrates at speed greater than $k(N)N$ (Assumption~\ref{assum2}) and one stating that the entries of $X_N$ are sharp sub-Gaussian and the variance (or in the complex case, identifying $\C$ with $\R^2$, the covariance matrix) of each entries must be the same as for a matrix for the GOE/GUE (Assumption~\ref{assum1}). We leave the details of these two assumptions in Section~\ref{sec:largesteigenvalues}.
		\begin{theo} 
		Let $(X_N)_{N \in \N}$ be a sequence of Wigner matrices satisfying Assumptions \ref{assum2} and \ref{assum1}. 
		Let $k(N)\in \N$ such that  $k(N) = o(N/ \ln N)$.

		Then $\hat{\nu}_N = \hat{\nu}_{k(N)}(X_N)$ statisfies a large deviation principle for the weak topology with speed $2N k(N)$ and rate function $\frac{\beta}{2} \mathcal{I}$ where $\mathcal{I}$ is defined by:
		\[ \mathcal{I}(\nu) = \begin{cases}
			\int_{\R} I(x) d \nu(x) &\text{ if } \nu(] - \infty, -2]) = \nu([2, + \infty[) = \frac{1}{2} \\ + \infty  &\text{otherwise,}\end{cases} \]
		and $I$ is the function defined by 
		\[ I(x) = \begin{cases} \int_{2}^x \sqrt{t^2 - 4} dt &\text{ when } x \geq 2 \\
			\int_{x}^{-2} \sqrt{t^2 - 4} dt &\text{ when } x \leq -2 \\
			0 &\text{ when } - 2 < x < 2. \end{cases} \]
		In particular, $\mathcal{I}$ is a good rate function. 
	\end{theo}
	This result is an extension of the results of \cite{HuGu,GuHu21} to a growing number of eigenvalues. In particular, the rate function in \cite{Mai,GuHu21} that was expressed as a sum of the rate functions for one eigenvalue, is now an integral over $\hat{\nu}$.

	\subsubsection{Large Deviations of Extreme Eigenvalues of a Perturbation of a Gaussian Matrix}
	
	Next, in Section~\ref{sec:extremeeig} we prove a large deviations principle for the extreme eigenvalues of a rank $k(N)$ deformation of a Gaussian matrix in Theorem~\ref{theo:extremeeigGaussian}. Namely, we take a random matrix $X_N$ from either the GOE (for $\beta =1$) or the GUE (for $\beta =2$) and $D_N$ a deterministic random variable 
	of rank $2 k(N)$:
	\[ D_N = \diag( \theta^N_{- k(N)}, \dots \theta^N_{-1}, \theta^N_1 ,\dots, \theta^N_{k(N)},0,\dots,0) \] where $k(N)= o(N)$ and $\theta^N_{-k(N)} \leq \dots \leq \theta^N_{-1} \leq 0 \leq \theta^N_1\leq \dots \leq \theta_{k(N)}^N$.
	For a probability measure $\mu \in \mathcal{P}(\R)$, we let $Q_{\mu}$ be the inverse of the cumulative distribution function of $\mu$. In other words, $Q_{\mu}$ is defined for $p \in ]0,1[$ by:
	\begin{equation}\label{eq:quantileintro}
		Q_{\mu}(p) = \inf \{ x \in \R: p \leq \mu( ] - \infty,x ]) \} .
	\end{equation}
	We have the following large deviations principle for the extremal empirical measure of $X_N + D_N$.
			\begin{theo}
		Let us assume that there is some probability measure $\xi$ such that: 
		\[ \lim_{N \to \infty} \frac{1}{2 k(N)} \sum_{i= -k(N),\dots,k(N) \atop i \neq 0} \delta_{\theta^N_i} = \xi. \]
		Then, $\hat{\nu}_N = \hat{\nu}_{k(N)}(X_N+D_N) $ satisfies a large deviation principle in speed $2 N k(N)$  with good rate function $\beta \mathcal{I}_{\xi} /2$ defined by:
		\[ \mathcal{I}_{\xi}(\nu) = \begin{cases}
			\int_0^1 I_{ Q_{\xi}(t)}( Q_{\nu}(t)) d t & \text{ if } \nu(] - \infty, -2]) = \nu([2, + \infty[) = \frac{1}{2} \\ + \infty  &\text{otherwise.}\end{cases} \]
		where for any $\theta \geq 0$ and $x \geq 0$, $I_{\theta}(x)$ is defined by:
		\[ I_{\theta}(x) = \begin{cases} + \infty &\text{ if }x < 2\\
		I(x) - J(\theta,x,\sigma) - \inf_{y \geq 2} (I(y) - J(\theta,y,\sigma)) &\text{ if  }x > 2
	\end{cases}
	\]
	and for $\theta \leq 0$, $x \leq 0$, $I_{\theta}(x)$ is defined by $I_{\theta}(x) = I_{ - \theta}( -x)$. 
	\end{theo}
Again, this is an extension of results from \cite{ GuHu21} and the rate function is the integration of the rank one rate function over $\hat{\nu}$.
	
	\subsubsection{The Free Energy of Spherical Vector Spin Glasses}
	
	The applications of spherical integrals to compute the free energies of spherical spin glass models are discussed in Section~\ref{sec:spinglass}. Let $G_N$ be a GOE matrix and let $\sigma \in \R^N$ be a unit vector. Consider the function,
	\[
	H_N(\sigma) = \frac{N}{2} \sigma^\top G_N \sigma
	\]
	which is called the pure 2-spin Hamiltonian associated with the spherical Sherrington--Kirkpatrick (SK) model. 
		
	In this paper, we are interested in a high dimensional variant of this model called the vector spin model. In contrast to the standard vector spin models, the main novelty is that Theorem~\ref{maintheo} also allows us to consider the case when the dimensions of the vector spins are dependent on $N$.  Consider a matrix $\Sigma = \Sigma_{k(N)} = (\sigma_1, \dots, \sigma_{k(N)}) \in \R^{k(N) \times N}$ of $k(N)$ replica and a sequence of constraint matrices $Q = Q_{k(N)} \in \R^{k(N) \times k(N)}$ with $1$ along the diagonal. The free energy is defined by
	\[
	\tilde F^\epsilon_N(Q) = \frac{1}{N k(N)} \ln \int \1( |||  \Sigma \Sigma^\top - Q ||| \leq \epsilon ) e^{\sum_{\ell = 1}^{k(N)} \theta_\ell H_N(\sigma^\ell)} \, d \sigma^1 \cdots d \sigma^{k(N)},
	\]
	where $d \sigma$ is uniform on the unit sphere in $\R^N$. If the sequence of $Q_{k(N)}$ have smallest eigenvalue uniformly bounded away from $0$, and the eigenvalue distributions of the constraint matrices $Q_{k(N)}$ and a temperature transformed constraint matrix $\tilde Q_{k(N)} = (\sqrt{\theta_i \theta_j}Q_{ij})_{ij \leq k(N)}$ converges weakly to compactly supported measures $\mu$ and $\tilde \mu$ (see Assumption~\ref{assum:AQ}), then the limit of the free energy can be computed precisely.
	\begin{prop}  Let $k(N) = o(N)$. Suppose that the matrices $D_{k(N)} = \diag (\theta_1, \dots, \theta_{k(N)})$ and $Q_{k(N)}$ satisfy Assumption~\ref{assum:AQ}, then
		\[
		\lim_{\epsilon \to 0} \lim_{N\to \infty}  \E \tilde F_N^{\epsilon} (Q_{k(N)}) = \int p(x) d \tilde \mu(x) + \int \ln(x) d \mu(x) 
		\]
		where $p(x)$ is the one dimensional limit of the spherical SK free energy given precisely by
		\[
		p(x) =  \begin{cases}
			\frac{x^2}{4} & x < 1
			\\x - \frac{\ln x}{2} - \frac{3}{4}  &x \geq 1.
		\end{cases}
		\]
	\end{prop}
	This proposition when $k(N) = k$ is independent of $N$ was already proven in \cite[Theorem~2]{PTSPHERE}. In that paper, it was remarked that a large deviations principle can be used to yield a simpler proof of the limit of the free energy in the $2$ spin models. This large deviations principle is precisely the asymptotics of the spherical integral we prove in this paper.  The limit of the free energy and its reduction to a growing rank spherical integral is detailed in Proposition~\ref{prop:vectorspinspherical}.

	\subsubsection{The Mutual Information of Spiked Matrix Factorization}

	In Section~\ref{sec:matestimation}, we explore the application of spherical integralswhen studying the mutual information of spiked matrices with rotationally invariant prior.
	
	Consider the following estimation problem
	\[
	Y_N = G_N + \sqrt{\frac{\gamma}{N}} X_N
	\]
	where $X_N \in \R^{N \times N}$ is a random rank $k(N)$ rotationally invariant symmetric matrix and $\gamma \in \R^+$ is the signal to noise ratio. In the case when $k(N) = k$ is constant, this model is the classical finite rank matrix estimation problem was studied in works such as \cite{miolanefundamentallimits, barbier2020information, lesieur2015mmse, dia2016mutual, krzakala2016mutual}. We are interested in the sub extensive rank case when the rank $k$ signal goes to $\infty$ but slow enough such that $\frac{k}{N} \to 0$. This is in a different regime than the challenging extensive rank case when $\frac{k}{N} \to \alpha > 0$. These extensive rank models have been a topic of a lot of recent works \cite{maillard2021perturbative,troiani2022optimal,barbier2021statistical}.
	
	Suppose that the eigenvalues of $\theta_1, \dots, \theta_{k(N)}$ have joint distribution $P_D$.  Our goal is to study the denoising of such matrices. In particular, the mutual information (see Section~\ref{sec:matestimation}) is given by
	\[
	\frac{1}{Nk(N)} I_N(\gamma):=  \frac{\gamma}{4} \frac{1}{k(N)} \E \tr(X^2) - \frac{1}{Nk(N)} \E_Y \ln \int  e^{-\frac{N\gamma}{4} \sum_{i = 1}^{k(N)} \theta^2_i} \bigg( \int \exp \frac{ \sqrt{\gamma} N }{2} \tr\bigg(   U^\top Y U D  \bigg) dU \bigg) \, d P_D(\theta)  \label{eq:mutualinformation}.
	\]
	A formula for this limit when $k(N)$ is constant was proven in \cite{miolanefundamentallimits}. The phase transition and several applications of these models can be found in \cite{KZLowRank}.
	We combine facts about the phase transitions of the top eigenvalues \cite{BBAP05} and the spectrums of spiked matrices \cite{BGNSpiked,JYmeso} with the limits in Theorem~\ref{maintheo} to compute this limit explicitly. The spherical integrals that appear in these models are explicit, and can be seen as a limit of the finite rank problems. This is in sharp contrast to the complicated formulas that appear in the extensive rank matrix factorization problems.
	
	To state the limit (see Assumption~\ref{assum:eigdist}), we assume that the empirical distribution  $\frac{1}{k(N)}\sum_{i=1}^{k(N)}\delta_{\theta_{i}}$  converges under $P_D$ in probability towards a probability measure $\eta$ in a metric that metrizes weak convergence, and that its law 
	satisfies a large deviations principle with good rate function $\Gamma$ and speed $k(N) N$. We moreover assume that $P_D$ is compactly supported in $[-M,M]^{k(N)}$ for some finite $M$.

	\begin{prop}\label{prop:mutualinfo}
			If Assumption~\ref{assum:eigdist} holds, then, 
			\[
			\lim_{N \to \infty}\frac{1}{Nk(N)} I_N(\gamma)  = \frac{\gamma}{4} \int_0^1 x^2 d\eta(x) - \sup_{\nu} \bigg( - \frac{\gamma}{4} \int_0^1 x^{2}d\nu(x) + \frac{1}{2} \int_0^1 J( \sqrt{\gamma} Q_\nu( x ), f( \sqrt{\gamma} Q_\eta(x ) ) , \sigma ) \, d x - \Gamma(\nu) \bigg)
			\] where $Q_{\nu}$ denotes the quantile function \eqref{eq:quantileintro} and 
			\[
			f(x) = \begin{cases}
				2 & x < 1\\
				x + \frac{1}{x} & x > 1
			\end{cases}
			\]
			is the BBP transition map.
	\end{prop}
	
	\begin{rem}
		If $k$ is independent of $N$ and the signal $X = U^\top D_k U$ has deterministic eigenvalues $D_k = \diag(\theta_1, \dots, \theta_k)$ then the result from Proposition~\ref{prop:mutualinfo} simplifies to
		\[
		\lim_{N\to \infty} \frac{1}{Nk}I_N(\gamma) = \frac{\gamma}{2} \sum_{i = 1}^k \theta^2_i - \frac{1}{2k} \sum_{i = 1}^k F(\gamma, \theta_i)
		\]
		where
		\[
		F(\gamma, \theta_i) = \begin{cases}
			\frac{\gamma \theta_i^2}{2} & \gamma \leq \frac{1}{\theta_i^2}
			\\\gamma \theta_i^2 - \ln(\gamma \theta_i^2) - \frac{1}{2 \gamma \theta_i^2}  & \gamma > \frac{1}{\theta_i^2}.
		\end{cases}
		\]
	\end{rem}
	
	
	\subsection{Outline of the Paper}A first elementary remark on the HCIZ integral is that it is invariant in each of its arguments by conjugation by a unitary matrix. There one can assume that both are real diagonal matrices. 
	
	One can then notice that when $B$ is diagonal of rank $k$ such that the $k$ non-zero eigenvalues are first on the diagonal, $\tr(AUBU^*)$ only depends on the first $k$ columns of $U$. Therefore in the case where $k$ remains finite, one can proceed by successive conditionning on the columns of $U$ (see \cite{GuHu21}). However this method becomes a lot less tractable when the rank goes to infinity with $N$ because the number of conditionings is no longer constant. Instead, we introduce a new approach and break up our problem in two parts. One part will involve the $k(N) \times k(N)$ topmost leftmost submatrix of $U$ and the matrix containing the extremal eigenvalues of $A_N$ and the second one involves the $N -k(N) \times k(N)$ bottommost leftmost submatrix of $U$. 
=
		We will need the following definitions: 
		\begin{defi}\label{def:Dk} Let $k \geq 1$ and reals $\theta_1, \dots, \theta_k$ such that $\theta_1 \geq \dots \geq \theta_k$. Given a matrix $H \in \mathcal{H}_k^{\beta}$, let $\mathrm{spec}(H) = (\phi_1, \dots, \phi_k)$ denote its spectrum arranged in decreasing order. Given $\overline{\theta} = (\theta_1, \dots, \theta_k)$, we define
			\begin{align*} \mathcal{D}^{\overline{\theta}}_{k} :=\{ ( \overline{\phi}, \overline{\psi}) \in (\R^{ k})^2 &: \text{there exists $H_1, H_2 \in \mathcal{H}_k^{\beta}$ such that } \overline{\phi} = \mathrm{spec}(H_1), \overline{\psi} = \mathrm{spec}(H_2),  \overline{\theta} = \mathrm{spec}(H_1 + H_2)	
				\\&\qquad \text{and  $\overline{\phi}, \overline{\psi}$ have non-negative entries} 
			 \}. \end{align*}
\end{defi}

\begin{defi}\label{defi:optimization}
	Let $k\geq 1$, $N \geq 2k$ and $\mu \in \mathcal{P}( \R)$ a compactly supported measure. Let $\overline{\lambda}=( \lambda_1, \dots, \lambda_k)$ such that  $\lambda_1 \geq \dots \lambda_k \geq r(\mu) $ and $\overline{\theta}= (\theta_1, \dots ,\theta_k)$ such that $\theta_1 \geq \dots \geq \theta_k \geq 0$. For $(\overline{\phi}, \overline{\psi}) \in \mathcal{D}_k^{\overline{\theta}}$ we define:
	\[\mathcal{F}(\mu,\overline{\lambda},\overline{\theta},\overline{\phi}, \overline{\psi}) :=  \sum_{i=1}^{k} \left[ \lambda_i \psi_i + J(\phi_i, \lambda_k, \mu) + (\ln(\phi_i) - \ln(\theta_i)) \right]. \]
	We also define: 
	\[ \mathcal{M}(\mu,\overline{\lambda},\overline{\theta}) := \sup_{(\overline{\phi}, \overline{\psi}) \in \mathcal{D}_k^{\overline{\theta}}} \mathcal{F}(\mu,\overline{\lambda},\overline{\theta},\overline{\phi}, \overline{\psi}).	\]
	\end{defi}
	We now describe the structure of the paper. In Section~\ref{sec:upbd}, we begin by proving an intermediate upper bound for positive matrices $D_N$.
	\begin{theo}\label{theo:positivetheo}
		If $(A_N)_{N \in \N}$ and $(D_N)_{N \in \N}$ are two sequences of self-adjoint matrices that satisfy Assumption~\ref{assum:A1} and if $D_N$ is positive then: 
		\[\limsup_{N \to \infty}  \Big[ \frac{2}{\beta k(N) N} \ln I_N(D_N,A_N) -  \frac{1}{k(N)} \sum_{i=1}^{k(N)} J(\theta_i, \lambda_i,\mu ) \Big]\leq 0. \]
	\end{theo}
 To this end we will first prove:
	\begin{theo}\label{theo:positivetheobis}
	If $(A_N)_{N \in \N}$ and $(D_N)_{N \in \N}$ are two sequences of self-adjoint matrices that satisfy Assumption~\ref{assum:A1} and if $D_N$ is positive then, with $\overline{\theta}^k $ the $k(N)$ largest eigenvalues of $D_N$ and $\overline{\lambda}^k$ the $k$ largest eigenvalues of $A_N$, 
	\[\limsup_{N \to \infty}  \Big[ \frac{2}{\beta k(N) N} \ln I_N(D_N,A_N) - \mathcal{M}(\mu,\overline{\lambda}^k,\overline{\theta}^k) \Big]\leq 0. \]
\end{theo}
We will study the variational problem that defines $\mathcal{M}$ and using the finite rank case result, we will prove that: 
\begin{theo}\label{thm:variational} If $k,N$ are integers and $\overline{\lambda}, \overline{\theta} \in \R^k$ and $\mu \in \mathcal{P}(\R)$ which satisfy the conditions of Definition~\ref{defi:optimization}:
	\[ \mathcal{M}(\mu, \overline{\lambda}, \overline{\theta}) \leq \sum_{i=1}^k J( \theta_i, \lambda_i, \mu). \]
\end{theo}
\noindent Those two results clearly implies Theorem~\ref{theo:positivetheo}. 

The restriction to positive $D_N$ is useful since then $I_N(D_N,\cdot)$ becomes an increasing function. Next, we explain how to deal with negative temperatures in  Subsection~\ref{sec:negtemp}, to extend Theorem~\ref{theo:positivetheo} to the setting with negative temperatures.
	
	\begin{prop}\label{prop:UB}
		If $(A_N)_{N \in \N}$ and $(D_N)_{N \in \N}$ are two sequences of self-adjoint matrices that satisfy Assumption~\ref{assum:A1}, then: 
		\[\limsup_{N \to \infty}  \frac{2}{\beta k(N) N} \ln I_N(D_N,A_N) - \left[ \frac{1}{k(N)} \sum_{i=1}^{l(N)} J(\theta_i, \lambda_i,\mu ) + \frac{1}{k(N)} \sum_{i=1}^{k(N)- l(N)} J(\theta_{l(N) +i}, \lambda_{ N + i -k(N) + l(N)},\mu ) \right] \leq  0 .\]
	\end{prop}

	Lastly, we prove the matching lower bound in Section~\ref{sec:lwbd}.
	
	\begin{prop}\label{prop:lwbd}
		If $(A_N)_{N \in \N}$ and $(D_N)_{N \in \N}$ are two sequences of self-adjoint matrices that satisfy Assumption~\ref{assum:A1}, then: 
		\[\liminf_{N \to \infty}  \frac{2}{\beta k(N) N} \ln I_N(D_N,A_N) - \left[ \frac{1}{k(N)} \sum_{i=1}^{l(N)} J(\theta_i, \lambda_i,\mu ) + \frac{1}{k(N)} \sum_{i=1}^{k(N)- l(N)} J(\theta_{l(N) +i}, \lambda_{ N + i -k(N) + l(N)},\mu ) \right] \geq  0. \]
	\end{prop}
	
	The applications to the large deviations of the extremal empirical measures will be discussed in Section~\ref{sec:largesteigenvalues}. The large deviations of the extreme eigenvalues of deformations of a Gaussian matrix will be explained in Section~\ref{sec:extremeeig}. Lastly, we explain how to compute the free energies of spherical spin glasses in Section~\ref{sec:spinglass} and the mutual information of spiked matrices in Section~\ref{sec:matestimation}.

	\section{Upper bound by temperature conditioning}\label{sec:upbd}
	
	The main goal of this section is to prove Theorem~\ref{theo:positivetheo}. Throughout the paper, since the function $I_N$ is invariant by conjugation by elements from $\mathcal{U}_N^{(\beta)}$, we will assume that both $A_N$ and $D_N$ are diagonal matrices. The general idea of this upper bound is to break up the trace in the integral into two terms: a first term that depends on the $k(N)$ largest eigenvalues of $A_N$ and the $k(N)$ first coordinates of the matrix $U$ (whose matrix will be denoted $U_1$) and a second term involving the remainder of the eigenvalues of $A_N$ as well as the remainder of the coordinates of $U$. We will then condition on $M = U_1^* U_1$ whose law is explicit. After this conditioning, we are left to deal with a product of two spherical integrals but each with two different sets of temperatures $\Theta_N'$ and $\Theta''_N$ depending on $M$ (hence the name ``temperature conditioning"). The first integral is going to be bounded by a simple maximal bound, and for the second integral, we use the positivity of $D_N$ to use a monotonicity argument. Then putting the two bounds together along with the density of the law of $M$ gives an upper bound in the form of an optimization problem on $(\R^+)^{k(N)} \times (\R^+)^{k(N)}$ that we solve using  the finite rank formulas in Subsection~\ref{sec:varprob}. The extension of the upper bound to the setting with negative temperatures will be done in Subsection~\ref{sec:negtemp}.

	 First, we will prove a decomposition lemma that will be very useful for the rest of the paper: 
	\begin{lemma} \label{cond}
		Let $N \in \N$, $P = \diag( p_1,\dots,p_N)$ and $Q = \diag(q_1,\dots,q_N)$ where $p_1 \geq \cdots \geq p_N$. For $k \in [ 1,N]$, we denote $ P^{(k +)}:= \diag( p_1,\dots, p_{N - k})$ and $P^{(k-)}:= \diag(p_{k+1},\dots,p_N)$. We also denote $Q_1^{(k)} = \diag(q_1,\dots,q_k)$ and $Q_2^{(k)} = \diag(q_{k+1},\dots,q_N)$. If $k$ is such that the entries of $Q_2^{(k)}$ are non-negative then, we have 
		\[ I_N(Q_1^{(k)},P) I_{N - k} \Big(Q_2^{(k)}, \frac{N}{N -k} P^{(k-)} \Big) \leq I_N(Q,P) \leq I_N( Q_1^{(k)},P) I_{N - k} \Big(Q_2^{(k)}, \frac{N}{N -k} P^{(k+)} \Big). \]
		If all the entries of $Q_2^{(k)}$ are non-positive, we have the same inequalities with $\leq$ replaced by $\geq$. 
	\end{lemma}
	\begin{proof}
		For $U$ our Haar-distributed unitary matrix, we will denote $U_1^{(k)}$ the rectangular matrix formed by taking the first $k$ columns of $U$ and $U_2^{(k)}$ the matrix formed by taking the remaining last $N - k$ columns. Then we can condition the spherical integral by $U_1^{(k)}$,
		\begin{eqnarray*}
			I_N(Q,P) &=& \E[\exp( N  \tr( U^* P UQ))] \\
			&=& \E[\exp( N  \tr( (U_1^k )^* P U_1^{k}Q_1^{(k)})+ N \tr((U_2^k)^* P U_2^kQ_2^{(k)}) )] \\ & = &
			\E[\exp( N  \tr( (U_1^k )^* P  U_1^{k}Q_1^{(k)})) \E[ \exp( N \tr((U_2^k)^* P U_2^kQ_2^{(k)})) | U_1^k  ]]. 
		\end{eqnarray*}
	 	Conditionally on $U_1^{(k)}$, the law of $U_2^{(k)}$ is that of $WV$ where $W$ is a deterministic and arbitrary $N \times N- k$ matrix whose columns are perpendicular to the columns of $U_1^{(k)}$ and $V$ is a Haar- distributed matrix on $\mathcal{U}_{N-k}^{\beta}$. 
		Therefore, if we write $R = (U_2^k)^* P U_2^k$ conditionally on $U_1^{(k)}$, $R$ has the same law as $V^* M V $ where $M = W^* P W$ so one can write: 
		\[ I_N(Q,P) =\E[\exp( N  \tr( (U_1^k )^* P  U_1^{k}Q_1^{(k)})) \E[ \exp( N \tr(V^* M V Q_2^{(k)})) | U_1^{(k)}  ]] .\]
		Furthermore, using Weyl's  formulas which gives the $i$-th largest eigenvalue of an Hermitian matrix $A$ as:
		\[ \lambda_i(A) = \max_{ V \text{subspace of } \R^N \atop dim(V) = i} \min_{ u \in V \setminus \{ 0 \}} \frac{ \langle A u, u \rangle}{\langle u, u \rangle}  \text{ and } \lambda_{N - i}(A)= \min_{ V \text{subspace of } \R^N \atop dim(V) = i} \max_{ u \in V \setminus \{ 0 \}} \frac{ \langle A u, u \rangle}{\langle u, u \rangle}\]
		one can see that if one denotes $\lambda_i(M)$ the $i$-th largest eigenvalue of $M$ that $ p_i \geq \lambda_i( M) \geq p_{i + k}$ for $i = 1,\dots, N- k$. And so, we have that exists $U'\in \mathcal{U}^{(\beta)}_{N-k}$  such that $P^{(k-)} \leq U' MU'^* \leq P^{(k+)}$. Since $Q_2^{(k)} \geq 0$, $ M \mapsto \tr(VMV^* Q_2^{(k)})$ is an increasing function and therefore 
		\[ \E[ \exp( N \tr(V^* P^{(k-)} V Q_2^{(k)})) ] \leq \E[ \exp( N \tr(V^* M V Q_2^{(k)})) | U_1^k  ] \leq \E [ \exp( N \tr(V^* P^{(k+)} V Q_2^{(k)})) ] .\]
		The result follows easily. 	
	\end{proof}
	
	Given our hypothesis in Assumption~\ref{assum:A1}, we can in fact assume that:
	\[ D_N = \diag(\theta_1, \dots , \theta_{k(N)},0\dots,0) \text{ and } A_N = \diag(\lambda_1,\dots ,\lambda_N) .\]
	Then, given the definition of $I_N(A,B)$ we gave for matrices $A$ and $B$ that are not $N \times N$ in Definition~\ref{def:sphericalint},  we can choose to consider instead $D_N = \diag(\theta_1, \dots , \theta_{k(N)})$. 
	
	If we consider a unitary matrix $U$ that is Haar distributed in $\mathcal{U}_N^{\beta}$, we can write $U_1$ and $U_2$ the respective $k(N) \times k(N)$ and $N- k(N) \times k(N)$ matrices such that: 
	\[ U = \begin{pmatrix} U_1 & *  \\ U_2 & *  \end{pmatrix} . \]
	Let $M = U_1^{*} U_1$. $M$ follows the following law on the set of Hermitian matrices $\{ H \in \mathcal{H}_N^{\beta}: 0 \leq H \leq I_{k(N)}\}$ 
	 \begin{equation} \label{eq:density}  \frac{1}{Z}  \det( I -M)^{\frac{\beta}{2} ( N - k(N) +1) -1 } \det (M)^{ \frac{\beta}{2} -1} dM \end{equation} 
	(see for instance the proof of this result in \cite[Section 2]{Col05}). Since we assumed $A_N$ diagonal, up to permutation of row and columns, we can assume that the eigenvalues are ordered decreasingly on the diagonal.
	Let us denote $A'_N$ the matrix extracted from $A_N$ by taking its first $k(N)$ rows and columns and $A''_N$ the matrix extracted from $A_N$ by taking its last $N - k(N)$ rows and columns so that: 
	\[ A_N = \begin{pmatrix} A'_N & 0 \\
		0 & A''_N \end{pmatrix} .\]
	
	We can write:
	\[ \exp\Big( N \frac{\beta}{2} \tr( U^* A_N U D_N) \Big)= \exp \Big( N \frac{\beta}{2} \tr(U^*_1 A'_N U_1 D_N) \Big) \exp\Big( N  \frac{\beta}{2} \tr(U^*_2 A''_N U_2 D_N) \Big). \]
	We can write $U_1 = V_1 \sqrt{M}$ and $U_2 = V_2 \sqrt{ I_{k(N)} - M}$ where $V_1$ and $V_2$ are both independent, $V_1$ is Haar distributed on $\mathcal{U}^{\beta}_{k(N)}$ and $V_2$ has the same distribution as the $k(N)$ first columns of a Haar distributed matrix on $\mathcal{U}^{\beta}_{N -k(N)}$. We have the following decomposition,
	
	\begin{equation} \label{eq1} I_N(D_N,A_N) = \frac{1}{Z}\int_{M \in \mathcal{H}_N \atop 0 \leq M \leq I_{k(N)} } \det( I -M)^{\frac{\beta}{2} ( N - k(N) +1) -1 } \det (M)^{ \frac{\beta}{2} -1} I^{(1)}( \Theta'_N,A'_N) I^{(2)}(\Theta''_N, A''_N) dM
	\end{equation}
	where $\Theta'_N = \sqrt{M}D_N \sqrt{M}$ and $\Theta''_N = \sqrt{I_{k(N)} - M}D_N \sqrt{I_{k(N)} - M}$, and 
	
	\[ I^{(1)}(\Theta'_N,A'_N) = \E\Big[\exp\Big(N \frac{\beta}{2}\tr(U^*A'_N U\Theta'_N) \Big)\Big] \]
	with $U$ being Haar-distributed on $\mathcal{U}^{\beta}_{k(N)}$ and
	\[ I^{(2)}(\Theta''_N,A''_N) = \E\Big[\exp\Big(N \frac{\beta}{2} \tr(U^*A''_N U\Theta''_N) \Big) \Big] \]
	with $U$ being distributed as the $ k(N)$ first columns of a unitary Haar-matrix in $\mathcal{U}^{\beta}_{N- k(N)}$. Our main goal is to prove the following upper bound corresponding to the first bound in Theorem~\ref{theo:positivetheobis}. 
	\begin{lemma} \label{lem1}
		We have the following upper bound:
		\begin{equation}\label{eq:lem1}
		 \frac{2}{\beta k(N) N} \ln I_N(D_N,A_N) \leq \mathcal{M}(\mu,\overline{\lambda},\overline{\theta})  + o_N(1)
		\end{equation}
	where $\overline{\theta}$, $\overline{\lambda}$ are the corresponding $k(N)$ largest eigenvalues of $D_N, A_N$
	\end{lemma}

	We are going to prove this upper bound by letting $\bar{\psi}$ and $\bar{\phi}$ be the respective spectra of $\Theta'_N$ and $\Theta''_N$ ordered decreasingly. Then, in equation \eqref{eq1} we will bound the determinants and the quantities $I^{(1)}$ and $I^{(2)}$. First, we notice that $\Theta'_N$ and $\Theta''_N$ are respectively similar to $\sqrt{D_N} M \sqrt{D_N}$ and $\sqrt{D_N} (I_{k(N)} -M) \sqrt{D_N}$ and their spectra are also respectively the squares of the singular values of $\sqrt{D_N} \sqrt{M}$ and $\sqrt{D_N} \sqrt{ I_{k(N)} -M}$. Since those two Hermitian matrices sum up to $D_N$ we have that indeed $\bar{\psi},\bar{\phi} \in \mathcal{D}_{k(N)}$. Then we will need the following Lemma:
	\begin{lemma}\label{lemmaUB}
		We have:
		\begin{equation}\label{eq:UBeq1}
			\frac{1}{N k(N)} \left[ \ln  \int_{M \in \mathcal{H}_N \atop 0 \leq M \leq I_{k(N)} } \det (M)^{ \frac{\beta}{2} -1} dM - \ln Z \right] = o_N(1) 
		\end{equation}
	where $Z$ is the renormalizing constant appearing in the equation \eqref{eq:density}.
		Uniformly for $M \in \mathcal{H}_N^{\beta}$ such that $0 \leq M \leq I_{k(N)}$, we have: 
		\begin{equation}\label{eq:firstrate}
			\frac{2}{\beta N k(N)} \ln I^{(1)}(\Theta'_N,A'_N) \leq \frac{1}{k(N)}\sum_{i=1}^{k(N)} \lambda_i \psi_i
		\end{equation}
		and 
		\begin{equation}\label{eq:UBeq3}
			\frac{2}{\beta N k(N)} \ln I^{(2)}(\Theta''_N,A''_N) \leq \frac{1}{k(N)} \sum_{i=1}^{k(N)} J(\phi_i, \lambda_{k(N)}, \mu) + o_N(1) .
		\end{equation}
	\end{lemma}
	\begin{proof}
		To prove the first point \eqref{eq:UBeq1}, recalling the normalization term in \eqref{eq:density}, 
		one can first notice that 
		\begin{eqnarray*} \frac{\int_D \det (M)^{ \frac{\beta}{2} -1} dM}{Z} &=& \frac{ \int_D \det (M)^{ \frac{\beta}{2} -1} dM}{\int_D \det(I - M)^{\alpha(N)-1} \det (M)^{ \frac{\beta}{2} -1} dM} \\ & = & \frac{ \int_D \det (M)^{ \frac{\beta}{2} -1} dM}{\int_D dM} \frac{\int_D dM}{\int_D \det(I - M)^{\alpha(N)-1} \det (M)^{ \frac{\beta}{2} -1}dM}
		\end{eqnarray*}
		where $\alpha(N) = \frac{\beta}{2}( N - k(N) +1) -1$ and the domain $D$ is $\{ M \in \mathcal{H}^{\beta}_{k(N)}: 0 \leq M \leq I\}$. We can use Selberg formula to compute the first term \cite{Selberg,AndersonSelberg}. Indeed, we have: 
		\[ \frac{ \int_D \det (M)^{ \frac{\beta}{2} -1} dM}{\int_D dM} = \frac{ \int_C \Delta(x)^{\beta} \prod_{i=1}^{k(N)} x_i dx }{\int_C \Delta(x)^{\beta} dx} = \frac{ S_{k(N)}(1,\frac{\beta}{2},\frac{\beta}{2})}{S_{k(N)}(1,1,\frac{\beta}{2})} \]
		where $C$ is the hypercube $[0,1]^{k(N)}$ and $S$ is given by: 
		\[ S_n(a,b,c) = \prod_{j=1}^n \frac{\Gamma(a +jc) \Gamma(b + jc) \Gamma(1 + (j+1)c)}{\Gamma(a + b + (n+j-1)c) \Gamma(1 + c)}. \]
		For $\beta =2$, this quotient is trivially $1$. With $\beta=1$, the quotient simplifies to: 
		
		\[ \frac{ \int_D \det (M)^{ \frac{\beta}{2} -1} dM}{\int_D dM} = \frac{\Gamma(\frac{1}{2})  \Gamma(\frac{k(N)+3}{2})}{\Gamma(\frac{k(N)+1}{2}) \Gamma(\frac{3}{2})}.\]
		Using Stirling's equivalent, we have that the $\ln$ of this ratio divided by $N k(N)$ is $o(1)$. For the second term, let us first denote $B_n(A,r)$ the ball of center $A$ and radius $r$ in $\mathcal{H}^{\beta}_n$ for the operator norm and let $V(r)$ denote its volume for the measure $dM$. We have $D \subset B_{k(N)}(0,1)$, therefore the numerator on the second term is less than $V(1)$. Let $\epsilon >0$. For $M \in B_{k(N)}( \epsilon I, \epsilon)$, we have $\det(I-M) \geq (1 - 2 \epsilon)^{k(N)}$, therefore by localizing the integral on $B(\epsilon I, \epsilon)$ we have:
		\[ \int_D \det(I - M)^{\alpha(N)-1} \det (M)^{ \frac{\beta}{2} -1}dM \geq V(\epsilon) (1 -2 \epsilon)^{\alpha(N)k(N)},\]
		and so: 
		\[
		\frac{\int_D dM}{\int_D \det(I - M)^{\alpha(N)-1} \det (M)^{ \frac{\beta}{2} -1}dM}
		\leq \frac{V(1)}{V(\epsilon)} (1 -2 \epsilon)^{\alpha(N) k(N)}.\]
		Therefore, noticing that  since $\mathcal{H}_{k(N)}^{\beta}$ has dimension $ \frac{\beta k(N)( k(N) -  1)}{2} + k(N)$
		\begin{equation}\label{eq:volumeball}
			V(\epsilon) = \epsilon^{\frac{\beta k(N)( k(N) -  1)}{2} + k(N)} V(1)
		\end{equation}
		we have that 
		\[ \limsup_{N \to \infty} \frac{1}{N k(N)} \ln  \frac{\int_D dM}{\int_D \det(I - M)^{\alpha(N)-1} \det (M)^{ \frac{\beta}{2} -1}dM} \leq 0.\]
		
		To prove the second point \eqref{eq:firstrate}, recall that if $A,B$ are both Hermitian matrices with respective spectrum $\lambda_1 \geq \dots \geq \lambda_n$ and $\psi_1 \geq \dots \geq \psi_n$ , then 
		\[ \tr(AB) \leq \sum_{i=1}^n \lambda_i \psi_i .\]
		The inequality \eqref{eq:firstrate} is thus trivial.

		To prove the third point \eqref{eq:UBeq3}, we use Lemma \ref{cond} and its notation. With an immediate recursion, we have 
		\[ I^{(2)}( \Theta''_N , A''_N) \leq \prod_{i=0}^{k(N)-1} I_{N - k - i} \Big( \phi_{i + 1}, \frac{N}{N -k -i} A''^{(i  +)} \Big) .\]
		
		Then to conclude, we notice that the one dimensional result in Theorem \ref{theo:rank1} implies that
		\begin{equation}\label{eq:uniformbound}
			\lim_{N \to + \infty} \sup_{ i \in [0,k(N) - 1] }\left| \frac{2}{ \beta N} \ln I_{N - k - i} \Big( \phi_{i + 1}, \frac{N}{N -k -i} A''^{(i  +)} \Big) - J(\phi_{i + 1}, \lambda_{k(N)}, \mu) \right| = 0.
		\end{equation}
	\end{proof}
	
	The proof of Lemma~\ref{lem1} is now immediate.
	
	\begin{proof}[Proof of Lemma~3.2]
	Notice that since $\Theta''_N = \sqrt{I_{k(N)} - M} D_N \sqrt{I_{k(N)} - M}$ , then 
	\[\ln \det( I_{k(N)} - M) = \ln \det \Theta''_N - \ln \det D_N = \sum_{i=1}^{k(N)} (\ln \phi_i - \ln \theta_i) \]
	Therefore, using equations \eqref{eq:firstrate} and \eqref{eq:UBeq3}, we have that uniformly in $M$ such that $0 \leq M \leq I$, 
	
	\begin{multline}
		\frac{2}{ \beta N k(N)} \ln \det( I - M)^{\alpha(N)} I^{(1)}_N( \Theta'_N, A'_N) I^{(2)}( \Theta''_N,A''_N) \leq \\
		\frac{1}{k(N)} \sum_{i=1}^{k(N)}\left[ \lambda_i \psi_i + J(\phi_i, \lambda_{k(N)}, \mu) + (\ln(\phi_i) - \ln(\theta_i)) \right] + o_N(1).
	\end{multline}
	Using this bound in conjunction with equation ~\eqref{eq1} and equation ~\eqref{eq:UBeq1} from Lemma~\ref{lem1}, one gets the result of Lemma~\ref{eq:lem1}. 
	\end{proof}

	\subsection{Solving the Variational Problem}\label{sec:varprob}
	
	To finish the proof of Theorem~\ref{theo:positivetheobis}, we have to show that the supremum appearing in the upper bound Lemma~\ref{lem1} is indeed the expected limit. We start with an upper bound of the supremum.  In order to do this, we first notice that the dependence in $N$ has been completely removed. The variational problem that defines $\mathcal{M}$ only depends on a finite of parameters, $\mu, k, \overline{\theta}$ and $\lambda_1,\dots, \lambda_k$. We will use this to our advantage by using the results that have already been established in the case of $k$ finite. 
		First let us take $k$ fixed, $\mu \in \mathcal{P}( \R)$ $\overline{\lambda} = ( \lambda_1, \dots, \lambda_k)$ such that $\lambda_k \geq r(\mu)$, and $\overline{\theta} = (\theta_1, \dots , \theta_k)$. Let  $\overline{\xi}^N =( \xi_1^N, \dots,\xi_{N -2k}^N)$ a sequence such that $r(\mu) \geq  \xi_1^N \geq \dots \geq \xi_{N -2k}^N \geq l( \mu)$ for all $N \geq 2k$ and such that 
		\[ \lim_{N \to \infty} \frac{1}{N -2k} \sum_{i=1}^{N -2k} \delta_{\xi_i^N} = \mu.\]
		We then define for all $N \geq 2k$: 
		\[ A_N = \text{diag}( \lambda_1, \dots, \lambda_k, \underbrace{\lambda_k, \dots, \lambda_k}_{ k \text{ times }}, \xi_1^N, \dots, \xi_{N -2k}^N ) \]
		and 
		\[ D= \text{diag}( \theta_1, \dots, \theta_k).\]
		Let us prove the following lemma:
	\begin{lemma}\label{lem:lowboundM}
		\[ \liminf_{N \to \infty} \frac{2}{\beta N} \ln I_N(D,A_N) \geq \mathcal{M}(\mu, \overline{\lambda}, \overline{\theta}). \]
		\end{lemma}
	\begin{proof}
		To prove this lemma, we need only to prove that for any $(\overline{\phi},\overline{\psi}) \in\mathcal{D}_k^{\overline{\theta}}$ defined in Definition~\ref{def:Dk}, we have 
			\[ \liminf_{N \to \infty} \frac{2}{\beta N} \ln I_N(D,A_N) \geq \mathcal{F}(\mu, \overline{\lambda}, \overline{\theta},\overline{\phi},\overline{\psi} ). \]
				For this, given such $(\overline{\phi},\overline{\psi})$, we have by definition two positive matrices $H_1$ whose spectrum is $\overline{\phi}$ and $H_2$ whose spectrum is $\overline{\psi}$ such that $H_1 + H_2 = D$. We can then find a positive matrix $L$ such that $0 \leq L \leq I_k$ and $H_1 = \sqrt{D}( I_k -L) \sqrt{D}$ and $H_1 = \sqrt{D}L \sqrt{D}$. We can also assume that $I_k - L$ is invertible, otherwise we have that at least one $\phi_i$ is zero, which implies $\mathcal{F}( u, \overline{\lambda}, \overline{\theta}, \overline{\phi}, \overline{\psi}) = - \infty$ in which case our result is trivial. Furthermore since equation \ref{eq1} still holds with here 
				\[ A'_N = A' = \text{diag}( \lambda_1, \dots, \lambda_k) \]
				and 
				\[ A''_N = \text{diag}( \underbrace{\lambda_k, \dots, \lambda_k}_{ k \text{ times }}, \xi_1^N, \dots, \xi_{N -2k}^N ) \]
				 and since $k$ here remain finite, the set of matrices $M$ we integrate over is not $N$-dependent anymore. For $M=L$, we have through a classical Laplace method: 
				\begin{equation}\label{lim1} \lim_{N\to \infty} \frac{1}{ \beta N} \ln I_N^{(1)}( \sqrt{I_k -L} D \sqrt{I_k -L}, A') = \max_{U \in \mathcal{U}_k^{\beta} } \text{ Tr}( U^* \sqrt{I_k -L} D \sqrt{I_k -L} U A') = \sum_{i=1}^k \lambda_i \psi_i 
				\end{equation}
				since $\overline{\psi}$ is the spectrum of $\sqrt{I_k -L} D \sqrt{I_k -L}$. 
				and using Proposition 1 of \cite{GuHu21}, we have since $\overline{\phi}$ is the spectrum of $\sqrt{M}D \sqrt{M}$ that 
				
				\begin{equation}\label{lim2} \lim_{N\to \infty} \frac{2}{ \beta N} \ln I_N^{(2)}( \sqrt{I_k -L} D \sqrt{I_k -L},  A''_N) = \sum_{i=1}^k J(\phi_i, \lambda_k, \mu).
				\end{equation} 
				
			At last, we have: 
			\begin{equation}
				\det( I_k - L) = \sum_{i=1}^k ( \ln \phi_i - \ln \theta_i). 
				\end{equation}

				Therefore can use the equicontinuity of the functions 
				\[
				D\mapsto \frac{1}{N}\ln I^{(1)}( \sqrt{L} D \sqrt{L},A) \text{ and } D \mapsto \frac{1}{N}\ln I^{(2)}(  \sqrt{I_k - L} D \sqrt{I_k - L},A)
				\]
				 and then localize our integral $I_N( D, A_N)$ in a neighborhood of $L$. We prove such an equicontinuity in Lemma \ref{equicont} in the following section. We delay the proof until then since here thanks to $k$ remaining finite, we are in a simpler case. In other terms, for every $\epsilon >0$, one can find $\eta$ small enough such that for any $N >0$ any $L'$ such that $||| L- L'||| \leq \eta$,
				
				\begin{equation}\label{approxim1} \Big|  \frac{1}{N} \ln I^{(1)}(\sqrt{L} D \sqrt{L}, A') - \frac{1}{N} \ln I^{(1)}(\sqrt{L'} D \sqrt{L'}, A') \Big| \leq \epsilon 
				\end{equation}
				
				\begin{equation}\label{approxim2} \Big|  \frac{1}{N} \ln I^{(1)}(\sqrt{I_k - L} D \sqrt{I_k - L}, A''_N) - \frac{1}{N} \ln I^{(1)}(\sqrt{I_k - L'} D \sqrt{I_k - L'}, A''_N) \Big| \leq \epsilon 
				\end{equation}
				and 
				\begin{equation}\label{approxim3} | \ln \det( I_k -L') - \ln \det(I_k -L) | \leq \epsilon.
				\end{equation}
				 Last, to prevent against the fact that $L$ may not be in the interior of $\{M \in \mathcal{H}_N^{\beta}: 0 \leq M \leq I_k \}$, we let $L'' = (1 - \eta/2) L + \eta/2 I_k$ and we consider $V = B( L'', \eta/2) \subset B(L, \eta)$. 
				 Therefore, localizing our integral on $V$, if one denotes $\Phi = \sqrt{I_k -L} D \sqrt{I_k -L}$ and $\Psi = \sqrt{L} D \sqrt{L}$ and then using equations \eqref{approxim1}, \eqref{approxim2}, \eqref{approxim3}: 
				 \begin{eqnarray*}
				 	\frac{2}{ \beta N} \ln I_N(D,A_N) &=& \frac{2}{ \beta N} \ln \int_{ 0 \leq M \leq I_k} \det(M)^{\beta/2 -1} \det(I_k -M)^{ \frac{\beta}{2} ( N -k +1) -1} I^{(1)}( \Theta'_N, A') I^{(2)}( \Theta''_N, A''_N) dM - \ln Z \\
				 	&\geq& \frac{2}{\beta N} \ln \int_{ M \in V} \det(M)^{\beta/2 -1} dM + \frac{2}{\beta N} \ln I^{(1)}(\Psi, A') + \frac{2}{\beta N} \ln I^{(2)}( \Phi, A''_N) - \frac{2}{\beta N} \ln Z \\
				 	&& + \frac{ N - k + 1 - \frac{2}{\beta}}{N}\ln \det( I_k - L)  -\epsilon.
				 \end{eqnarray*}
			 Taking the $\liminf_{N \to \infty}$ and then $\epsilon$ to $0$, on has: 
			 
			 \[ \liminf_N \frac{2}{ \beta N} \ln I_N(D,A_N) \geq \mathcal{F}(\mu, \overline{\lambda}, \overline{\theta},\overline{\phi},\overline{\psi} ). \]				
								
				Optimizing in $(\overline{\phi}, \overline{\psi})$ then gives the result. 
\end{proof}				 
			
		Then, to conclude with Theorem, one has to just notice that we can apply Proposition 1 of \cite{GuHu21} to $I_N(D,A_N)$ which gives that
	\[ \lim_{N \to \infty} \frac{1}{\beta N} \ln I_N( D_N, A_N) = \sum_{i=1}^k J( \theta_i, \lambda_i, \mu) \]
which proves the bound. It is actually easy to see then that the inequality in Theorem~\ref{thm:variational} is an equality. One need only to apply Lemma~\ref{lem1} with the same $A_N$ and $D_N$ we used in the preceding proof with $k(N)=k$ fixed and then use again the results of \cite{GuHu21} to get the reverse inequality.  

\begin{rem}
	In a previous version of this paper, the authors tried to solved the variational problem that defines $\mathcal{M}(\mu, \overline{\lambda}, \overline{\theta})$ directly. That attempt used the Ky Fan inequalities that $(\overline{\phi}, \overline{\psi})$ must satisfy if $(\overline{\phi}, \overline{\psi}) \in \mathcal{D}_k^{\overline{\theta}}$. Those inequalities stipulate that if the $\theta_i$ are ordered decreasingly, and reminding that we assume that the $\phi_i$ and $\psi_i$ are ordered decreasingly, we have for any $l \in [1,k]$ 
	\[ \psi_1 + \dots + \psi_l + \phi_{k} + \dots + \phi_{k -l+1} \leq \theta_1 + \dots + \theta_l. \]
	Using these inequalities, one can prove that: 
	
	\begin{eqnarray*} \mathcal{F}(\mu, \overline{\lambda}, \overline{\theta}, \overline{\phi}, \overline{\psi})  &\leq& \sum_{i=1}^k \Big[ \lambda_{i} \theta_i -  \lambda_i \phi_{k -i +1}  + J(\phi_{k-i+1}, \lambda_k, \mu) + ( \ln \phi_i - \ln \theta_i) \Big] \\
		&\leq& \sum_{i=1}^k F_i(\phi_i) 
		\end{eqnarray*}
	with
	\[ F_l( \phi) = \lambda_{ k-l+1} ( \theta_{k-l+1} - \phi) + J( \phi, \lambda_k, \mu) + (\ln( \phi) - \ln(\theta_{k-l+1})) .\]
	One can then differentiate $F_i$,
	\[ F'_l( \phi) = \begin{cases}  \lambda_k - \lambda_{k -l +1} \text{ when } \phi \geq G_{\mu}(\lambda_{k - l +1} ) \\
		\lambda_k - G_{\mu}^{-1}(\theta_{k-l+1}) \text{ when } \phi < G_{\mu}(\lambda_{k - l +1} )  \end{cases} \] 
	where we take the convention $G_{\mu}^{ -1}( \theta) = r(\mu)$ if $\theta > G_{\mu}( r(\mu))$. 
	Therefore we can compute the maximum of $F_i$ on $\R^+$ and if $G_{\mu}(\lambda_{k-l +1} )\leq \theta_{k-l+1}$ for all $l$, we have:
	\[ \max_{\R^+} F_{l}( \phi) = F_{i}( G_{\mu}^{-1}(\theta_{k-l+1})) = J( \theta_{ k -l+1}, \lambda_{k+l -1} , \mu) \]
	and the upper bound is satisfied. In the previous version of this paper, we incompletely argued that in the general case (if at least one inequality $G_{\mu}(\lambda_{k-l +1} )\leq \theta_{k-l+1}$ fails) we could enforce the additional constraints $\phi_l \in [0, \theta_{k -l+1}]$ for all $l$, which would yield the upper bound we seek. Since we were not able to solve thius variationnal problem directly this way in the general case, we instead chose to argue using the results already known in the case of finite rank from \cite{GuHu21}. However it is our belief that using constraints on $(\overline{\phi}, \overline{\psi}) \in \mathcal{D}_k^{\overline{\theta}}$ other than the Ky Fan inequalities, it should be possible to directly solve this problem directly without relying on the results of \cite{GuHu21}. Descriptions of the necessary and sufficient constraints on $(\overline{\phi}, \overline{\psi})$ for it to be in $\mathcal{D}_k^{\overline{\theta}}$, also known as Horn's problem, were first given by Helmke and Rosenthal in \cite{HelRos95}, who proved their necessity. Klyashko then proved their sufficiency in \cite{Kly98}. In \cite{KTW04}, Knutson, Tao and Woodward also provide a beautiful description of those constraints in terms of honeycomb networks. 
	\end{rem}

	\subsection{Inclusion of negative temperatures}\label{sec:negtemp}
	We now consider the general case by adding negative temperatures $\theta_i^-$ to the matrix $D_N$ in order to prove Proposition~\ref{prop:UB}. We take two sequences of matrices $(A_N)_{N \in \N}$ and $(D_N)_{N \in \N}$ that satisfy Assumption~\ref{assum:A1}. To simplify the notations, we let $\lambda^{+}_i = \lambda_i, \theta^{+}_i = \theta_i$ for $1 \leq i \leq l(N)$ and $ \lambda^{-}_i = \lambda_{N-i+1}, \theta^{-}_i = \theta_{k(N) - i+1}$ for $1 \leq i \leq k(N) - l(N)$. We need to prove that under Assumption~\ref{assum:A1} on $A_N$ and $D_N$, that
		\[\limsup_{N \to \infty}  \frac{2}{\beta k(N) N} \ln I_N(D_N,A_N) - \left[ \frac{1}{k(N)} \sum_{i=1}^{l(N)} J(\theta_i, \lambda_i,\mu ) + \frac{1}{k(N)} \sum_{i=1}^{k(N)- l(N)} J(\theta_{l(N) +i}, \lambda_{ N + i -k(N) + l(N)},\mu ) \right] \leq  0 .\]
	\begin{proof}[Proof of Proposition~\ref{prop:UB}]
	We write: 
	\[ D_N = \begin{pmatrix} D_N^- & 0 \\
		0 &   D_N^+  \end{pmatrix} \]
	with 
	\[D_N^+ = \diag(\theta_1^+,\dots.\theta_{l(N)}^+ ) \text { and } D_N^- = \diag(\theta_1^-,\dots.\theta_{k(N) - l(N)}^- ).\]
	First, we can write using Lemma \ref{cond}: 
	\[ I_N(D_N,A_N) \leq I_N(D_N^{-},A_N) I_{N - k(n) + l(N)} \bigg( D_N^{+}, \frac{N}{N - k(N) + l(N)} A_N^{((k(N) -l(N) )+)} \bigg) \]
	since $ I_N(- D_N^{-}, - A_N) = I_N( D_N, A_N)$ and $- D_N^{-}$ is nonnegative. We can now use Theorem \ref{theo:positivetheo} to state that:
	\[  \limsup_{N \to \infty} \bigg( \frac{1}{N (k(N) - l(N))} \ln I_N( D_N^{-}, A_N) - \frac{\beta}{2 k(N) - l(N)} \sum_{i=1}^{k(N) - l(N)} J( \theta_i^{-}, \lambda^{-}_{i}, \mu) \bigg) \leq 0
	\]
	and since $k(N) \geq l(N)$,
	\[  \limsup_{N \to \infty}  \bigg( \frac{1}{N k(N)} \ln I_N( D_N^{-} , A_N ) - \frac{\beta}{2 N k(N)} \sum_{i=1}^{k(N) - l(N)} J( \theta_i^{-}, \lambda^{-}_{i}, \mu) \bigg) \leq 0.
	\]
	For the second term of the product, using that $N / (N -k(N) +l(N))$ tends to $1$ and that (for $N$ large enough) the $l(N)$ largest eigenvalue of $A_N^{((k(N) - l(N))+)}$ are the same as the $l(N)$ largest eigenvalues of $A_N$ and that its empirical measure converges to $\mu$ since $k(N) - l(N) = o(N)$, we have using Theorem \ref{theo:positivetheo}:
	
	\begin{multline} \limsup_{N \to \infty} \frac{2}{\beta (N - k(N) + l(N))  l(N)} \Big\{ \ln I_{N - k(N) + l(N)} \bigg( D_N^{+}, \frac{N}{N - k(N) + l(N)} A_N^{((k(N) -l(N) )+)}\bigg) - \\ \frac{\beta}{2 l(N)} \sum_{i=1}^{l(N)} J( \theta_i^{-+}, \lambda^{+}_{i}, \mu) \Big\} \leq 0 \end{multline}
	which implies easily since $k(N) = o(1)$ that 
	\begin{multline} \limsup_{N \to \infty} \frac{2}{\beta N k(N)} \Big\{ \ln I_{N - k(N) + l(N)} \bigg( D_N^{+}, \frac{N}{N - k(N) + l(N)} A_N^{((k(N) -l(N) )+)} \bigg) - \\ \frac{\beta}{2 l(N)} \sum_{i=1}^{l(N)} J( \theta_i^{+}, \lambda^{+}_{i}, \mu) \Big\} \leq 0. \end{multline}	
	Putting these two bounds together, finishes the proof of Proposition~\ref{prop:UB}. 
	\end{proof}

	\section{Lower Bound} \label{sec:lwbd}

	In this section, we prove the lower bound in Proposition~\ref{prop:lwbd}. We first adapt our notations to include negative temperature. Given the spectra of $A_N$ and $D_N$, we can assume that
	\[ A_N = \begin{pmatrix} A'_N & 0 \\
		0 & A''_N \end{pmatrix} \]
	where 
	$A'_N = \diag(\lambda_1,\dots,\lambda_{l(N)}, \lambda_{ N - k(N) + l(N)+1},\dots,\lambda_N)$
	and 
	$A''_N = \diag( \lambda_{l(N)+ 1 },\dots,\lambda_{N -k(N) + l(N)})$. 
	We can also assume
	$D_N = \diag( \theta_1,\dots,\theta_{k(N)})$. Then for every $N\in \N$, we define the following diagonal matrix 
	\[
	L_N = \diag( m_1^{+},\dots,m_{l(N)}^+, m^{-}_{k(N) - l(N)},\dots,m_1^- )
	\]
	where for $i = 1, \dots, l(N)$
	\[m_i ^+ = \begin{cases} 1 - \frac{G_{\mu}( \lambda^+ _i)}{\theta_i^{+}} &\text{ if } \theta_i^{+} \geq G_{\mu}(\lambda_i^{+})\\
		0 &\text{otherwise} \end{cases}\]
	and  for $i =1,\dots,k(N) - l(N)$
	\[m_i ^- = \begin{cases} 1 - \frac{G_{\mu}( \lambda^- _i)}{\theta_i^{-}} &\text{ if } \theta_i^{-} \leq G_{\mu}(\lambda_i^{-})\\
		0 &\text{otherwise} \end{cases}.\]
	
	This definition implies
	\begin{equation}\label{eq:psidef}
	\Psi_N:= \sqrt{L_N} D_N \sqrt{L_N} = \diag(\psi_1^+,\dots,\psi_{l(N)}^+, \psi_{k(N) - l(N)}^{-} \dots, \psi^-_1) 
	\end{equation}
	and 
	\begin{equation}\label{eq:phidef}
	\Phi_N:= \sqrt{ I_N - L_N} D_N \sqrt{ I_N - L_N} = \diag(\phi_1^+,\dots,\phi_{l(N)}^+, \phi_{k(N) - l(N)}^{-} \dots, \phi^-_1)
	\end{equation}
	where for $i = 1, \dots, l(N)$: 
	\[\psi_i ^+ = \begin{cases}\theta_i^{+}  - G_{\mu}( \lambda^+ _i) &\text{ if } \theta_i^{+} \geq G_{\mu}(\lambda_i^{+})\\
		0 &\text{otherwise} \end{cases}\]
	and  for $i =1,\dots,k(N) - l(N)$
	\[\psi_i ^- = \begin{cases} \theta_i^{-} - G_{\mu}( \lambda^- _i) &\text{ if } \theta_i^{-} \leq G_{\mu}(\lambda_i^{-})\\
		0 &\text{otherwise} \end{cases}\]
	and where $\phi_i^{\pm} = \theta_i^{\pm} -  \psi_i^{\pm}.$
	
	We now localize the integral in equation \eqref{eq1} on an $\epsilon$-neighborhood of $L_N$ to prove the lower bound, that is
	\begin{equation}\label{eq:localized}
		I_N(A_N,D_N) \geq \frac{1}{Z}\int_{M \in B(L_N, \epsilon)  \atop 0 \leq M \leq I_{k(N)} } \det( I -M)^{\frac{\beta}{2} ( N - k(N) +1) -1 } \det (M)^{ \frac{\beta}{2} -1} I^{(1)}( \Theta'_N,A'_N) I^{(2)}(\Theta''_N, A''_N) dM .
	\end{equation}
	We recall that $B(L_N,\epsilon)$ denotes the ball of radius $\epsilon$ and center $L_N$ in the space of Hermitian matrices $\mathcal{H}_N^{\beta}$.  Our goal is to estimate the localized integral in the lower bound. 
	
	We first prove a modification of the third point of Lemma~\ref{lemmaUB} to its corresponding lower bound.
	
	\begin{lemma} We have
		\begin{eqnarray}\label{eq:lwbdineq4}
			\frac{2}{\beta N k(N)} \ln I^{(1)}(\Psi_N,A'_N) \geq \frac{1}{k(N)} \sum_{i=1}^{l(N)} \psi_i^{+} \lambda_i + \frac{1}{k(N)} \sum_{i=1}^{k(N) -l(N)} \psi_i^{-} \lambda_{N -i +1} + o_N(1) 
		\end{eqnarray}
	and
		\begin{eqnarray}\label{eq:lwbdineq0}
			\frac{2}{\beta N k(N)} \ln I^{(2)}(\Phi_N,A''_N) &\geq& \frac{1}{k(N)} \sum_{i=1}^{l(N)} J(\phi^+_i, \lambda_{l(N) + i + 1}, \mu) \nonumber \\
			& &\quad + \frac{1}{k(N)} \sum_{i=1}^{k(N) -l(N)} J(\phi^-_i, \lambda_{N - k(N) - i - 1}, \mu) + o_N(1).
		\end{eqnarray}

	\end{lemma}
	
	\begin{proof}
		For the first inequality we remind that 
		\[ \tr( \Psi_N A_N') = \sum_{i=1}^{l(N)} \psi_i^{+} \lambda_i^{+}  + \sum_{i=1}^{k(N) -l(N)} \psi_i^{-} \lambda_{N - i +1} \] 
		
		If $V$ is another unitary matrix, one has: 
		\begin{eqnarray*} |\tr(N U^*\Psi_N U A'_N)  - \tr(N V^*\Psi_N V A'_N) | &\leq& K^2 N  \sum_{i,j} | u_{i,j} - v_{i,j}| | u_{i,j} + v_{i,j}|  \\ &\leq& K^2 N  ||U - V||_2 ||U + V ||_2 \\
			&\leq& 2K^2 \sqrt{k(N)}||U - V||_2
		\end{eqnarray*}
		since $\max |\lambda_i \psi_j| \leq K^2$. 
		We then localize the expectation on an $\sqrt{k(N) } \epsilon$-neighborhood of $I_{k(N)}$ for $\|\cdot \|_2$ 
		
		\[ I^{(1)}(\Psi_N,A'_N) \geq \Pp[ ||I_{k(N)} - U ||_2 \leq \sqrt{k(N)}\epsilon ] \exp \Big( \frac{\beta}{2} N  \Big( \sum_{i=1}^{l(N)} \psi_i^{+} \lambda_i^{+}  + \sum_{i=1}^{k(N) -l(N)} \psi_i^{-} \lambda_{N - i +1} - 2k(N)M \epsilon \Big) \Big)\]
		So one only needs to show that
		\[ \lim_{N \to \infty} \frac{1}{N k(N)} \ln \Pp[ ||I_{k(N)} - U ||_2 \leq \sqrt{k(N)} \epsilon ] = 0. \]
		This is done by noticing that the ball $B_2(0,1)$ in $\mathcal{M}_{k(N)} (\C)$ can be covered by $(C\epsilon \vee 1)^{ - 2k(N)^2} $ balls of radius $\epsilon$ where $C$ is some constant and therefore $B_2(0,\sqrt{k(N)})$ 
		can be covered by $(C\epsilon \vee 1)^{ 2k(N)^2} $ balls of radius $\sqrt{k(N)}\epsilon$. Since $\mathcal{U}^{\beta}_{k(N)} \subset B_2(0,\sqrt{k(N)})$,then, it can also be covered by $(2C\epsilon^{-1} \vee 1 )^{ 2k(N)^2} $ balls of radius $\sqrt{k(N)}\epsilon$.  Using the invariance of $\|\cdot\|_2$ by left and right multiplication by elements of $\mathcal{U}^{\beta}_{k(N)}$, we have: 
		\[ (2C\epsilon^{-1} \vee 1)^{ 2k(N)^2} \Pp[ ||U - I_{k(n)} ||_2 \leq \sqrt{k(N)} \epsilon]  \geq 1, \]
		so
		\[  \Pp[ ||U - I_{k(n)} ||_2 \leq \sqrt{k(N)} \epsilon]  \geq ( 2C \epsilon^{-1}  \vee 1)^{-2 k(N)^2} , \]
		which proves the limit above since $\lim_{N \to +\infty} N^{-1} k(N) = 0$.

		For the second equation we can write: 
		\[ \Phi_N = \begin{pmatrix} \Phi_N^+ & 0 \\ 0 & \Phi_N^{-} \end{pmatrix} \]
		with
		\[ \Phi_N^+ = \diag(\phi_1^+,\dots,\phi_{l(n)}^+) \] 
		and 
		\[ \Phi^-_N:=  \diag( \phi_{k(N) - l(N)}^{-} \dots, \phi^-_1) \]
		
		Using Lemma~\ref{cond} and the fact that $\Phi_N^{-}$ is negative,
		\begin{eqnarray*}
			I^{(2)}(\Phi_N, A''_N) & = & I_{N - k(N)}\bigg( \frac{N}{N - k(N)}\Phi_N, A''_N\bigg) \\
			& \geq& I_{N - k(N)}\bigg( \frac{N}{N - k(N)}\Phi^- _N, A''_N\bigg) I_{ N -k(N) -l(N)}\bigg( \frac{N}{N - k(N) -l(N)}\Phi^+ _N, A''^{(l(N) +)}_N \bigg).
		\end{eqnarray*}

		This is proved the same way as the converse upper bound except we use the converse bound in Lemma~\ref{cond}. Applying the lower bound in Lemma~\ref{cond} recursively implies
		\begin{align*}
			\frac{2}{\beta N k(N)} \ln I_{N - k(N)}\bigg(\frac{N}{N -k(N)}\Phi^+_N, A_N''\bigg) &\geq \frac{2}{\beta N k(N)} \ln  \prod_{i = 0}^{l(N)-1} I_{N - k(N) - i} \Big(\phi^+_i, \frac{N}{N - k(N) - i} A_N''^{(i-)} \Big).
		\end{align*}
		The uniform bound in \eqref{eq:uniformbound} applied to the $l(N)$ eigenvalues eigenvalues of $A''_N$ implies that 
		\[
		\lim_{N \to + \infty} \sup_{ i \in [0,l(N)-1] }\left| \frac{2}{ \beta N} \ln I_{N - k - i} \Big( \phi^+_i, \frac{N}{N -k -i} A_N''^{(i  -)} \Big) - J(\phi_i, \lambda_{l(N) + i +1 }, \mu) \right| = 0,
		\]
		so
		\[
		\frac{2}{\beta N k(N)} \ln I^{(2)}\bigg(\frac{N}{N -k(N)}\Phi_N^{+}, A_N''\bigg)  \geq  \sum_{i=1}^{l(N)} J(\phi^+_i, \lambda_{k(N) + i +1 }, \mu) + o_N(1).
		\]
		
		We following is proved in the same way
		\[ \frac{2}{\beta N k(N)} \ln I_{ N -k(N) -l(N)}\left( \frac{N}{N - k(N) -l(N)}\Phi^+ _N, A''^{(l(N) +)}_N \right) = \frac{1}{k(N)} \sum_{i=1}^{k(N) - l(N)}J(\phi^-_i, \lambda_{N -k(N) -1 -i}, \mu) + o_N(1).\]
	\end{proof}
	
	We now control the volume of the integral in \eqref{eq:localized}. Let us remind that $D = \{ M \in \mathcal{H}_N^{\beta}: 0_{k(N)} \leq M \leq I_{k(N)}\}$. Since for $M \in D$, we do not have necessarily $B(M, \epsilon) \subset D$ we will first need to prove that: 
	\begin{equation}\label{eq:lwbdineq1}
		\lim_{N \to \infty} \frac{1}{N k(N)}  \ln \frac{\int_{ B(L_N,\epsilon)\cap D} dM}{Z} = 0,
	\end{equation}
	where $Z$ is the normalization factor in \eqref{eq:localized}.  We remind that we proved in Lemma~\ref{lemmaUB} that 
	\[ \lim_{N \to \infty} \frac{1}{N k(N)} \Big| \ln Z - \ln \int_{D} \det(M)^{ \beta/2 -1} dM \Big| =0 .\]
	Furthermore Using Selberg's formula in the proof of this same Lemma we also proved that:
	\[ \lim_{N \to \infty} \frac{1}{N k(N)} \Big| \ln \int_{D} dM  - \ln \int_{D} \det(M)^{ \beta/2 -1} dM \Big| =0 \]
	Using the fact that $B(I_{k(N)}/2,1/2) \subset D \subset B(0,1)$, we have that 
	
	\[  V(1/2) \leq \int_{D} dM \leq V(1) \]
	Since 
	\[ V(1/2) = V(1) 2^{ - O(k(N)^2)} \]
	we have 
	\[ \lim_{N \to \infty} \frac{1}{N k(N)} \Big| \ln \int_{D} dM  - \ln V(1) \Big| =0. \]
That leads to 
	\begin{equation}\label{eq:ZV}
	\lim_{N \to \infty} \frac{1}{N k(N)} \Big| \ln Z  - \ln V(1) \Big| =0. 
	\end{equation}
	
	 Furthermore, since with $L'_N := (1- \epsilon/2)L_N + \epsilon/2 I_N$,   $B(L'_N/2,\epsilon/2) \subset D \cap B(L_N,\epsilon) \subset B(L_N,\epsilon)$, 
				\[ \frac{V(\epsilon/2)}{Z} \leq \frac{\int_{ B(L_N,\epsilon)\cap D} dM}{Z} \leq \frac{V(\epsilon)}{Z}, \]
		 Since $V(\epsilon) = \epsilon^{ O( k(N)^2)} V(1)$ by \eqref{eq:volumeball} and \eqref{eq:ZV} \eqref{eq:lwbdineq1} then promptly follows.
	
	Next we need equicontinuity in $M$ of $\ln I^{(1)}(\Theta'_N,A'_N)$ and $\ln I^{(2)}(\Theta''_N,A''_N)$. Recall that the dependence of these functionals on $M$ is through $\Theta'_N$ and $\Theta''_N$ given below \eqref{eq1}. For this, we will first use the following lemma: 
	\begin{lemma}\label{equicont}
		The functions
		\[ f_N^{(1)}: M \mapsto \frac{1}{N k(N)} \ln I^{(1)}(\Theta'_N,A'_N) \]
		and 
		\[ f_N^{(2)}: M \mapsto \frac{1}{N k(N)} \ln I^{(2)}(\Theta''_N,A''_N) \]
		satisfy for every $M,M' \in \mathcal{H}_N^{\beta}$ such that $0 \leq M,M' \leq I_N$ and $M \neq M'$; 
		\[ | f_N^{(j)}(M) - f_N^{(j)}(M') | \leq 2 K |||M - M'|||^{1/2}. \]
	\end{lemma}
	\begin{proof}
		First, we remind that for any positive matrices $M,M'$
		\[ |||\sqrt{M} - \sqrt{M'}|||\leq |||M - M'|||^{1/2} .\]
		Therefore, since $|||\sqrt{M}|||, |||\sqrt{M'}||| \leq 1$, and $|||D_N||| \leq K$, 
		
		\begin{eqnarray*}
			|||\sqrt{M} D_N \sqrt{M} - \sqrt{M'} D_N \sqrt{M'}||| & \leq& |||(\sqrt{M} - \sqrt{M'}) D_N \sqrt{M} - \sqrt{M'} D_N (\sqrt{M'} - \sqrt{M})||| \\
			&\leq& |||\sqrt{M} - \sqrt{M'} ||| ( |||D_N |||( |||\sqrt{M}||| + |||\sqrt{M'}|||)) \\
			& \leq & 2 K |||\sqrt{M} - \sqrt{M'} |||.
		\end{eqnarray*}
		Then, for any matrix $V \in \mathcal{U}_{k(N)}^{\beta}$ ,since $||| A'_N||| \leq K$, we have that: 
		\begin{eqnarray*}
			|\tr(N V^* A_N' V \sqrt{M} D_N \sqrt{M}) - \tr(N V^* A_N' V \sqrt{M'} D_N \sqrt{M'})|  &= &|\tr(N  V^* A_N' V (\sqrt{M} D_N \sqrt{M} - \sqrt{M'} D_N \sqrt{M'}))| \\ 
			&\leq& N k(N) |||\sqrt{M} D_N \sqrt{M} - \sqrt{M'} D_N \sqrt{M'} |||. |||A'_N||| \\
			&\leq& 2N k(N) K^2 |||M - M'|||^{1/2} .
		\end{eqnarray*}
		From this we easily deduce the result for $f_N^{(1)}$. The only modification for $f_N^{(2)}$ is that we replace $M$ by $I_N - M$ which doesn't change anything since $0 \leq I_N - M \leq I_N$, and one has to notice that if $V$ is a $( N- k(N)) \times k(N)$  matrix whose columns are orthonormal, $V^* A_N'' V$ is a $k(N) \times k(N)$ matrix whose operator norm is less than $|||A_N''|||\leq K$.  
	\end{proof}
	
	Lastly, we simplify the determinants appearing in \eqref{eq:localized} when $M$ is localized near $L_N$. Using the definition of $L_N$, one can see that $|||(I - L_N)^{-1}||| \leq K'$ where 
	\[ K':= \frac{K}{\min(G_{\mu}(K),\min(G_{\mu}(- K)))}. \] 
	
	Denoting for a non-negative matrix $M$, $\sigma(M)$ its smallest eigenvalue, we have, $\sigma( I_N- L_N) \geq  K'^{-1}$. Therefore, if $\epsilon \leq K'^{-1}/2$, then $\sigma(I_N - M) \geq 1/(2K')$ for $M \in B(L_N, \epsilon)$. Then, using that the $\ln$ is $2 K'$-Lipschitz on $[1/ (2 K'),1]$, we can write for such $M$ that:  
	\[   \Big| \ln \det( I_N - L_N) -  \ln \det( I_N - M) \Big| \leq k(N) 2 K' \epsilon.  \]
	Looking at the definition of $L_N$, one can see that 
	\begin{equation}\label{eq:lwbdineq3}
		\frac{1}{N} \ln \det(I_N - L_N ) = \sum_{i=1}^{l(N)} (\ln \phi^+_i - \ln \theta^+_i) + \sum_{i=1}^{k(N)-l(N)} (\ln (-\phi^-_i) - \ln (-\theta^-_i))
	\end{equation}
	
	We now have the tools to estimate the integral in \eqref{eq:localized} and finish the proof of the lower bound.
	
	\begin{proof}[Proof of Proposition~\ref{prop:lwbd}]
	
	We now estimate the lower bound, from \eqref{eq:localized} for $\epsilon \leq 1/2 K'$. Recalling the definition of the matrices $\Psi_N$ and $\Phi_N$ in \eqref{eq:psidef} and \eqref{eq:phidef}, we take the $\ln$, divide by $N k(N)$ and use the preceding estimates \eqref{eq:lwbdineq4}, \eqref{eq:lwbdineq0}, \eqref{eq:lwbdineq1} and \eqref{eq:lwbdineq3} to conclude that
	\begin{align*}
		\frac{1}{N k(N)} \ln 	I_N(A_N,D_N) & \geq \frac{\ln(V(\epsilon/2) ) - \ln(Z)}{N k(N)} + \frac{\ln I^{(1)}(\Psi_N,A_N') + \ln I^{(2)}(\Phi_N,A_N'')}{N k(N)} 
		\\&\quad  + O(K^{2} \epsilon^{1/2}) + \frac{\beta \ln \det( I_N -L_N)}{2 N k(N)} + O(2K' \epsilon) +o_N(1) \\
		& \geq \frac{\beta}{2 k(N)} \sum_{i=1}^{l(N)} \Big( \lambda_i \psi^+_i + J( \phi^+_i, \lambda_{l(N)+i}, \mu) + \ln(\phi^+_i) - \ln \theta^+_i \Big) \\
		&\quad + \frac{\beta}{2 k(N)}\sum_{i=1}^{k(n) -l(N)} \Big( \lambda_{N -i +1} \psi^-_i + J( \phi^-_i, \lambda_{N - k(N)-i}, \mu) + \ln(-\phi^-_i ) - \ln(- \theta^-_i) \Big) \\
		&\quad + o_N(1) + O(2K' \epsilon) + O( K^{2} \epsilon^{1/2})
	\end{align*}
	
	For $i =1,\dots, l(N)$, we notice using the expression of $J$ that
	\begin{itemize}
		\item If $\theta_i^+ \leq G_{\mu}(\lambda_i)$, then $\psi_i^+ = 0$, and $\phi_i^+ = \theta_i^+$ and then since $\lambda_{l(N)+i} \leq \lambda_i$  
		\[ \psi_i^+ \lambda_i + J(\phi_i^+, \lambda_{l(N) + i},\mu) + \ln(\phi_i^+) - \ln(\theta_i^+) =  	J(\theta_i^+, \lambda_{l(N) + i},\mu) = J(\theta_i^+, \lambda_{ i},\mu)   \]
		\item If $\theta_i^+ \geq G_{\mu}(\lambda_i)$, then $\psi_i^+ = \theta_i^+ - G_{\mu}(\lambda_i^+)$ and $\phi_i^+  = G_{\mu}(\lambda_i^+)$ then:
		\begin{eqnarray*}
			\psi_i^+ \lambda_i + J(\phi_i^+, \lambda_{l(N) + i},\mu) &= & \lambda_i (\theta_i^+  - G_{\mu}(\lambda_i)) + J( G_{\mu}(\lambda_i), \lambda_{l(N) +i},\mu) + \ln(G_{\mu}(\lambda_i)) - \ln(\theta_i^+) \\
			& =& \lambda_i (\theta_i^+  - G_{\mu}(\lambda_i)) + J( G_{\mu}(\lambda_i), \lambda_{i},\mu) + \ln(G_{\mu}(\lambda_i)) - \ln(\theta_i^+) \\
			&= &J(\theta_i^+, \lambda_i, \mu). 
		\end{eqnarray*}

	\end{itemize} The same way, for $i=1, \dots k(N) - l(N)$, we notice that 
	\[ \Big( \lambda_{N -i +1} \psi^-_i + J( \phi^-_i, \lambda_{N - k(N)-i}, \mu) + \ln(-\phi^-_i ) - \ln(- \theta^-_i) \Big) =J( \theta, \lambda_{N_i+1} ,\mu) \] which concludes the proof.
	\end{proof}
	
	\section{Application to large deviations of the largest eigenvalues of random matrices}\label{sec:largesteigenvalues}
	In this section, we will use our main result to extend the large deviation principle for the largest eigenvalue of Wigner matrices with sharp sub-Gaussian entries proved in \cite{HuGu} to a growing number of eigenvalues. First, we recall the definition of a sharp sub-Gaussian random variable. 
	\begin{defi}\label{sharpsubG}
		If $X$ is a random variable on $\R^n$ with $n \in \N$, we say that $X$ is sharp sub-Gaussian if for every $t \in \R^n$ we have that: 
		
		\[ \E[ \exp( \langle t, X \rangle) ] \leq \exp \Big( \frac{ \langle t, \text{Cov}(X) t \rangle}{2} \Big) \]
		where $\langle \cdot , \cdot \rangle $ is the canonical scalar product on $\R^n$ and $Cov(X)$ is the covariance of $X$. If $X$ is a random variable with values in $\C$, we extend this definition by identifying $\C$ to $\R^2$ with the canonical base being $(1,i)$. 
	\end{defi}
	Next, we state the definition of the Wigner matrix we will use and the asssumptions we will need on its entries. The following definition is taking into account both the real and the complex case
	
	\begin{defi}\label{Wigner}
		Let $N \in \N$,  $\{ a_{i,j} \}_{1 \leq i < j \leq N }$ be a family of  real (resp. complex) independent centered variables such that $\E[ |a_{i,j}|^2] =1$ for all $i,j$ and $\{ d_i \}_{1 \leq i \leq N}$ be a family of real independent centered variables with finite variance. We call real (resp. complex) Wigner matrix the random matrix defined by
		\[ X_N(i,j) = \begin{cases} \frac{a_{i,j}}{\sqrt{N}} \text{ when } i < j \\
			\frac{d_i}{\sqrt{N}} \text{ when } i = j	\\
			\frac{a_{j,i}}{\sqrt{N}} \text{ when } i > j \end{cases}.\]
	\end{defi}
	
	It is well known that the empirical measure of a sequence of Wigner matrices whose diagonal entries have a bounded variance converges weakly in probability toward the semi-circular measure $\sigma =(2 \pi)^{-1} \mathds{1}_{[-2,2]}(x) \sqrt{x^2 -4} dx$. In order to approximate the empirical measure by $\sigma$ for the regime of large deviation we will be considering, we will need the following assumption.
	
	\begin{assum}\label{assum2}
		Let us assume that there exists a sequence of positive real numbers $\epsilon(N)$ converging to $0$ and a distance that metrizes the weak convergence in $\mathcal{P}( \R)$ such that the sequence of empirical measures $\hat{\mu}_N$  of $X_N$ satisfies: 
		\[ \lim_{N \to \infty} \frac{1}{k(N)N} \ln \Pp[ d(\hat{\mu}_N, \sigma ) > \epsilon(N) ] = - \infty .\]
	\end{assum}
	This assumption is for instance satisfied when the law of the unrenormalized entries $a_{i,j}$ and $d_i$ satisfy log-Sobolev inequalities  or have their support in some compact set independent of $N$ as stated in the following Lemma. 
	
	\begin{lemma}\label{lem:convergenceempirical}
		If the laws of the $a_{i,j}$ and $d_i$ satisfy one of the the following assumptions:
		\begin{enumerate}
			\item There is a compact $K \subset \C$ independent of $N$, such that the laws of the $a_{i,j}$ and $d_i$are supported in $K$.
			\item If $\beta=1$, there is a constant $c >0$ independent of $N$ such that the laws of the $a_{i,j}$ and $d_i$ satisfy a log-Sobolev inequality with constant $c$. We remind that we say that $\mu$ satisfy a log-Sobolev inequality if for every smooth function $f$,
			\[ \int f^2 \ln \frac{f^2}{ \mu(f^2)} d \mu \leq c \mu( ||\nabla f||_2^2 ) .\]
			If $\beta=2$, there is for every $N \in \N$, $i,j \in [A,N]$ such that $0 \leq i < j \leq N$, $u \in \C$ such that $|u|=1$ such that the laws of the $\Re ( u a_{i,j})$ and $\Im (u a_{i,j})$ are independent and $c >0$ independent on $N$ such that all $\Re( u a_{i,j}),\Im( u a_{i,j})$ and $d_{ i}$ satisfy a log-Sobolev inequality with constant $c >0$. 
			\item If $\beta =1$, the laws of $a_{i,j}, d_i$ are uniformly sub-Gaussian (in the sense that there exists $A > 0$ such that for all $N\in \N$, $i,j \in [1,N]$ such that $0 \leq i < j \leq N$ $ \E[ \exp( t \sqrt{N} a_{i,j})] \leq e^{A t^2}$ and for all $N\in \N$, $i \in [1,N]$, $ \E[ \exp( t \sqrt{N} d_i)] \leq e^{A t^2}$  and all $t \in \R$) and their distribution are symmetric with log-concave tails, in the sense that the functions $x \mapsto \Pp[ | a_{i,j}| \geq x]$ and $x \mapsto \Pp[ | d_i| \geq x]$ are log-concave. If $\beta=2$, there is for every $N \in \N$, $i,j \in [0,N]$ such that $0 \leq i < j \leq N$, $u \in \C$ such that $|u|=1$ such that the laws of the $\Re ( u a_{i,j})$ and $\Im (u a_{i,j})$ are independent and such that all $\sqrt{N} \Re( u a_{i,j}), \sqrt{N}\Im( u a_{i,j})$ and $\sqrt{N}d_{ i}$ are symmetric, uniformly sub-Gaussian and with log-concave tails.  
			\end{enumerate}
		then Assumption \ref{assum2} is satisfied. 
		\end{lemma} 
	
Part (1) and (2) of this lemma is a consequence of \cite[Theorem 1.3, Theorem 1.5]{GZ}. Part (3) is a slight generalization of Proposition 8.1 in \cite{AuGuHu} (Note that the uniformly sub-Gaussian part of this assumption becomes automatic once one has the sharp sub-Gaussian assumption). We will also need the following sharp sub-Gaussian hypothesis. 
	
	\begin{assum}\label{assum1}
		Let us assume that the $a_{i,j}$ and $d_i$ of Definition~\ref{Wigner} are sharp sub-Gaussian. 
		In the real case, let us also assume that they are such that: 
		\[ \forall 1 \leq i< j \leq N, \E[ a_{i,j}^2] = 1, \forall 1 \leq i \leq N, E[ d_i^2] =2. \]
		In the complex case, let us assume that they are such that:
		\[ \forall 1 \leq i< j \leq N, ~\E[\Re (a_{i,j})^2] =E[\Im (a_{i,j})^2]=\frac{1}{2}, \forall 1 \leq i \leq N, E[ d_i^2] =1. \]
	\end{assum}	
	The result we are going to prove is the following: 
	\begin{theo} \label{LDPTheo}
		Let $(X_N)_{N \in \N}$ be a sequence of Wigner matrices satisfying Assumptions \ref{assum1} and \ref{assum2}. 
		Let $(k(N))_{N \in \N}$ such that  $k(N) = o(N/ \ln N)$. 
		Let $\lambda_1^N \leq \dots \leq  \lambda_N^N$ denote the eigenvalues of $X_N$ and let us define $\hat{\nu}_N$ the \emph{extremal empirical measure} as follows: 
		\[ \hat{\nu}_N =\frac{1}{2 k(N)} \bigg(
		\sum_{i=1}^{k(N)} \delta_{\lambda_i^N} + \sum_{i=1}^{k(N)} \delta_{\lambda_{N - i +1}^N} \bigg) \]
		Then $\hat{\nu}_N \in \mathcal{P}_N(\R)$ statisfy a large deviation principle for the weak topology with speed $2 N k(N)$ and good rate function $\frac{\beta}{2} \mathcal{I}$ where $\mathcal{I}$ is defined as:
		\begin{equation}\label{eq:calI}
		\mathcal{I}(\nu) = \begin{cases}
			\int_{\R} I(x) d \nu(x) &\text{if  $\nu(] - \infty, -2]) = \nu([2, + \infty[) = \frac{1}{2}$}\\
			+ \infty & \text{ otherwise }   \end{cases} 
		\end{equation}
		where $I$ is the function defined by
		\begin{equation}\label{eq:Ix}
		I(x) = \begin{cases} \int_{2}^x \sqrt{t^2 - 4} dt \text{ when } x \geq 2 \\
			\int_{x}^{-2} \sqrt{t^2 - 4} dt \text{ when } x \leq -2 \\
			0 \text{ when } - 2 < x < 2 \end{cases} .
		\end{equation}
	\end{theo}
	The ideas of the proof remain largely the same as in \cite{GuHu21}, although considering measures formed from extremal eigenvalues rather than simply a $k$-tuple introduces some topological difficulties.
	
	First, we will need the following extension of our result in the case where the matrix $(A_N)_{N \in\N}$ is not bounded in operator norm any more. For this we will need the following variation on Assumption~ \ref{assum:A1}.
	
	\begin{assum}\label{assum:A2}
	Suppose that $(A_N)$ and $(D_N)$ satisfies the assumptions in Assumption~\ref{assum:A1} with the following modification to point (2): instead of assuming that $|||A_N||| \leq K$ in point (2) of Assumption~\ref{assum:A1}, we weaken the assumption and instead require that there exists a $M$ finite such that for every $N \in \N$: 
	\begin{equation}\label{eq:empiricalsquare}
	\hat \nu_N(x^2) =\frac{1}{2k(N)} \Big[\sum_{i=1}^{k(N)} \lambda_i^2 + \sum_{i=1}^{k(N)} \lambda_{N -i+1}^2 \Big] \leq M.
	\end{equation}
	\end{assum}

	The result in Theorem~\ref{maintheo} also holds under the weaker Assumption~\ref{assum:A2}
	
	\begin{theo}\label{generalizedSIconv}
		If $(A_N)$ and $(D_N)$ are two sequences of self-adjoint matrices that satisfy Assumption \ref{assum:A2}, then: 
		\[\lim_{N \to \infty} \left| \frac{2}{\beta k(N) N} \ln I_N(D_N,A_N) - \left[ \frac{1}{k(N)} \sum_{i=1}^{k(N)} J(\theta_i, \lambda_i,\mu ) + \frac{1}{k(N)} \sum_{i=1}^{k(N)} J(\theta_{l(N) +i}, \lambda_{ N + i -k(N)},\mu ) \right]\right| = 0. \]	
	\end{theo}
	
	\begin{proof}
		Let us take $L >0$ and let us call $g^{(L)}$ the function defined for a real $x$ by $g^{(L)}(x)= (x \vee -L) \wedge L$ consider for every $N \in \N$, $A^{(L)}_N= g^{(L)}(A_N)$ 
		
		For $L > \sqrt{M}$ with $\Delta_N^{(L)}:= A_N - A_N^{(L)}$  using the Von-Neumann inequality  we have
		\begin{eqnarray*}
			 |\tr(U A_N U^* D_N) -  \tr(U A^{(L)}_N U^* D_N)| &\leq&  |\tr(U\Delta_N^{(L)}U^* D_N) | \\
			 &\leq& K \Big( \sum_{i=1}^{k(N)} |\lambda_i| \mathds{1}_{|\lambda_i| \geq L}+ \sum_{i=1}^{k(N)} |\lambda_{N-i+1}| \mathds{1}_{|\lambda_{N-i+1}| \geq L}\Big) \\ &\leq& k(N) \frac{K M}{L } .
		 \end{eqnarray*}
		
		Therefore, one has that 
		\[ \Big| \frac{1}{N k(N)} \ln I_N(A_N,D_N) - \frac{1}{N k(N)} \ln I_N(A^{(L)}_N,D_N)\Big| \leq \frac{K M}{L } . \]
		Using the fact that $|J(\theta,., \mu)|$ is a $\theta$-Lipschitz function, we have that 
		\[ \Big|\frac{1}{k(N)} \sum_{i=1}^{k(N)} J(\theta_i, \lambda^{(L)}_i,\mu ) - \frac{1}{k(N)} \sum_{i=1}^{k(N)} J(\theta_i, \lambda_i,\mu ) \Big| \leq \frac{1}{k(N)}\sum_{i=1}^{k(N)} |\theta_i| \lambda_i \mathds{1}_{ \lambda_i^{(L)} < \lambda_i} \leq K \frac{M}{L} .\]
		We have the same inequality for the indices $i = l(N)+1,\dots,k(N)$. Using the Theorem \ref{maintheo} for $A_N^{(L)}$ and then making $L$ tend to $\infty$ proves the result. 
	\end{proof}

	Here, if $\lambda_1^N \leq \cdots \leq \lambda_N^N$ are the eigenvalues of a Wigner matrix $X_N$, to capture the rare events that involves $k(N)$ extremal eigenvalues, we will use the ``extremal" empirical measure: 
	\begin{equation}\label{eq:empiricalextremalmeasure}
	\hat{\nu}_N = \frac{1}{2 k(N)}\Big[ \sum_{i=1}^{k(N)} \delta_{\lambda_i^N} + \sum_{i=1}^{k(N)} \delta_{\lambda_{N - i +1}^N} \Big].
	\end{equation}
	
	The proof of the Theorem \ref{LDPTheo} follows closely the proof of the large deviation principle in \cite{GuHu21}. In that paper the main steps of the proof are as follows: 
	\begin{enumerate}
		\item Exponential  tightness (Lemma 1.8 and Section 2). 
		\item Asymptotics of the annealed spherical integral (Theorem 1.17 and Section 3). 
		\item Large deviation upper bound using those asymptotics (Theorem 1.9 and Corollary 1.16 and identification of the rate in Section 4).
		\item Large deviation lower bound using a tilt(Theorem 1.10 and Section 5). 
	\end{enumerate}
	Here the main difference is that we do not have only one parameter $\theta$ but a number $k(N)$ of parameters $\theta_{i}$ which varies in $N$. Here are the adaptations we will make to the original proof: 
	
	\begin{enumerate}
		\item For the exponential tightness, we will prove the following proposition: 
		\begin{prop}\label{exptight}
			If $(X_N)_{N \in \N}$ are Wigner matrices satisfying Assumptions \ref{assum1} and \ref{assum2} then for every $L > 0$, there is $M > 0$ such that: 
			\[ \limsup_{N \to \infty} \frac{1}{ k(N)N} \ln \Pp[ \hat{\nu}_N(x^2) \geq M ] \leq - L.\]
		\end{prop}
		\item Regarding the asymptotics of the annealed spherical integral, we will prove the following result: 
		\begin{prop}\label{AnnealedSI}
			If $(X_N)_{N \in \N}$ are Wigner matrices satisfying Assumption \ref{assum2}, $k(N) = O( N/ \ln N)$ for some $\epsilon>0$ and $D_N$ is a sequence of deterministic diagonal matrices of the form: 
			\[ D_N = \diag(\theta_1 ,\dots,  \theta_{2 k(N)},0,\dots,0) \]
			with $\theta_1 \geq \dots \geq \theta_{k(N)} \geq 0 \geq \theta_{2 k(N)} \geq \dots \geq \theta_{k(N) +1}$,
			and such that $||D_N|| \leq K
			$ for some $K > 0$.
			Then
			\[ \limsup_{N \to \infty} \frac{1}{k(N)} \Big| \frac{2}{\beta N} \ln \E[ I_N(X_N,D_N)] - \sum_{i=1}^{2 k(N)} \frac{\theta_i^2}{2}  \Big| = 0 .\]
		\end{prop}
		
		\item For the upper bound, since we are using a number of parameters $\theta_i$ that grows with $N$, we will first, with $N$ fixed, optimize on the $\theta_i$ and than we will look at the limit of this optimum as $N$ tends to infinity. 
		\item For the lower bound, we will first restrict ourselves to neighborhood of ``nice" $\nu$, that is measures $\nu$, supported on a compact subset of $] - \infty, -2 [ \cup ] 2, + \infty[$ and whose partition function is continuous. We will then use tilts of the measure of the form $I_N(X,D_N)/ \E[I_N(X_N,D_N)]$ where $\theta_i+ \frac{1}{\theta_i}$ is the quantile of level $i/N$ of $\nu$. Under this tilt, we will apply a technique similar to the proof of the Gartner-Ellis theorem to show that $\hat{\nu}_N$ converges toward $\nu$. That proof makes use of large deviation upper bounds for the tilted measure. Finally we will explain how to approximate any measure by ``nice" measures. 
	\end{enumerate}
	
	\subsection{Exponential tightness: Proof of Proposition \ref{exptight}}\label{proofexptight}
	This proof will use a classical argument on the cardinality of an $\epsilon$-net. First, let us state the cardinality bound we will need: 
	
	\begin{lemma}\label{Net}
		Let $\beta = 1,2$, $N\in \N$ and $k \leq N$. Let us consider $\mathcal{R}^{(\beta)}_{N,k}$ the subset of $(\Ss^{\beta N -1})^k$ of families of orthonormal vectors (for the real scalar product if $\beta=1$ and the complex one if $\beta=2$). We consider on $\mathcal{R}^{(\beta)}_{N,k}$ the distance induced by the following norm on $(\R^{\beta N})^k$, $||u|| = \max (||u_1||_2, \dots,||u_k||_2)$ (where $||.||_2$ is the classical Euclidean norm on $\R^{\beta N }$). Then, there exists an $\epsilon$-net $\mathcal{N}^{(\beta)}_{N,k}(\epsilon)$ of $\mathcal{R}^{(\beta)}_{N,k}$ such that:
		\[ |\mathcal{N}^{(\beta)}_{N,k}(\epsilon)| \leq \Big( \frac{6}{\epsilon} \Big)^{ \beta N k}. \]
	\end{lemma}  
	
	\begin{proof} We can find an $\epsilon$-net $\mathcal{N}(\epsilon)$ of  $\Ss^{\beta N -1}$ of cardinality at most $(3/ \epsilon)^{\beta N }$, so $(\mathcal{N}(\epsilon/2))^k$ is an $\epsilon/2$-net on $(\Ss^{\beta N -1})^k$. We build $\mathcal{N}_{N,k}(\epsilon)$ by choosing for each $x \in (\mathcal{N}(\epsilon/2))^k$ such that $B(x, \epsilon/2) \cap \mathcal{R}^{(\beta)}_{N,k} \neq \emptyset$, some $x'$ in $B(x, \epsilon/2) \cap \mathcal{R}^{(\beta)}_{N,k}$ arbitrarily. Then, it is easy to see that the set of such $x'$ is an $\epsilon$-net of $ \mathcal{R}^{(\beta)}_{N,k}$ with the stated bound on its cardinality. 
	\end{proof}  
	Next we will use the fact that for any self-adjoint positive matrix $M$ whose eigenvalues are $\mu_1 \geq \dots \geq \mu_N$: 
	\begin{equation} \label{eqNet}
		\max_{(e_1,\dots,e_k) \in\mathcal{N}^{(\beta)}_{N,k}(1/\sqrt{2}) } \sum_{i=1}^k \langle e_i, M e_i \rangle  \leq \max_{(e_1,\dots,e_k) \in\mathcal{R}^{(\beta)}_{N,k} } \sum_{i=1}^k \langle e_i, M e_i \rangle \leq 16 \max_{(e_1,\dots,e_k) \in\mathcal{N}^{(\beta)}_{N,k}(1/\sqrt{2}) } \sum_{i=1}^k \langle e_i, M e_i \rangle \end{equation}
	
	First, let us recall that the maximum in the middle term is attained when $e=
	(e_1,\dots,e_k) = (u_1,\dots,u_k)$ are the unitary eigenvectors for the respective largest eigenvalues $\mu_1, \dots,\mu_k$ and is equal to $\sum_{i=1}^k \mu_i$.
	The first inequality is trivial. For the second one, let $e=(e_1,\dots,e_k) \in \mathcal{N}^{(\beta)}_{N,k}$ such that $||u_i - e_i||_2 \leq 1/ \sqrt{2}$ for all $i \in [1,k]$. Then $ \Re \langle u_i,e_i \rangle \geq \frac{1}{4}$. Therefore, using that $M$ is positive, we have that 
	$\langle e_i, M e_i \rangle \geq \frac{\mu_i}{16}$. Summing over $i$ gives the desired inequality. 
	
	Furthermore, if we denote $\mu^N_1 \geq \dots \geq \mu^N_N$ the eigenvalues of $X_N^2$, it is easy to see that the quantity defined in \eqref{eq:empiricalsquare} satisfies
	\begin{equation} \label{bound}
		\hat{\nu}_N(x^2) \leq \frac{1}{2 k(N)} \sum_{i=1}^{2k(N)} \mu_i .\end{equation}
	Therefore, it will be sufficient to prove the exponential tightness of 
	
	\[ \frac{1}{2k(N)}\max_{(e_1,\dots,e_k) \in\mathcal{R}^{(\beta)}_{N,2 k(N)} } \sum_{i=1}^{2 k(N)} \langle e_i, X_N^2 e_i \rangle = \sum_{i=1}^{2 k(N)} \frac{1}{2 k(N)} ||X_N e_i||^2 \]
	
	First, let us prove the following lemma 
	
	\begin{lemma}
		Let $N \in \N$, $k \leq N$, and  $(e_1,\dots,e_{k}) \in \mathcal{R}_{N,k}^{(\beta)} $ and let us assume that $(X_N)_{N \in \N}$ is a sequence of Wigner matrices that satisfy Assumption \ref{assum2}. For $a < \beta/4$, we have
		\[ \E\left[ \exp( a N \sum_{i=1}^{k} ||X_N e_i||_2^2 ) \right] \leq  \Big( \frac{1}{\sqrt{1 -(2/\beta)^2 a}} \Big)^{\beta k(k-1)/2} \Big( \frac{1}{\sqrt{1 -4a/ \beta}} \Big)^{k} \Big(\frac{1}{\sqrt{1- 2 a/\beta^2}} \Big)^{\beta(N-k)k}.\]
	\end{lemma}
	
	For this, we will use the following lemma on sharp sub-Gaussian variables:
	\begin{lemma}
		Let $X$ be a centered sharp sub-Gaussian random variable in $\R^d$ and $G$ be a centered Gaussian variable with the same covariance matrix as $X$. For any positive quadratic form $\phi$, we have: 
		\[ \E[ e^{\phi(X)}] \leq \E[ e^{\phi(G)}]. \]
	\end{lemma}
	\begin{proof}
		First, there exists a symmetric positive matrix $A$ such that $\phi(x) = ||Ax||^2$. Therefore, since $AX$ is still sharp sub-Gaussian with the same covariance matrix as $AG$, without loss of generality, we can assume $\phi(x) = ||x||^2/2$. Then, we remind that for any $x \in \R^k$
		
		\[ e^{\frac{||x||^2}{2}} = \Big(\frac{1}{2 \pi} \Big)^{\frac{d}{2}} \int_{ \R^{d}} e^{  \langle t,x \rangle } e^{- \frac{ ( ||t||^2 ) }{2}} dt.  \] 
		Substituting $X_N$ for $x$, taking the expectation and using the sharp sub-Gaussianity, we obtain
		
		\begin{eqnarray*}
			\E[ e^{\frac{||X||^2}{2}}] &=& \Big(\frac{1}{2 \pi} \Big)^{\frac{d}{2}} \int_{ \R^{d}} \E[e^{  \langle t,X \rangle }] e^{- \frac{ ( ||t||^2 ) }{2}} dt \\
			& \leq& \Big(\frac{1}{2 \pi} \Big)^{\frac{d}{2}} \int_{ \R^{d}} \E[e^{  \langle t,G \rangle }] e^{- \frac{ ( ||t||^2 ) }{2}} dt  \\
			&\leq &\E[ e^{\frac{||G||^2}{2}}].
		\end{eqnarray*}
		Therefore, for $\phi(X) = a N \sum_{i=1}^k ||X e_i||^2$, which is a positive quadratic form on the set of Hermitian matrices, we have that 
		\[ \E\left[ \exp\left( a N \sum_{i=1}^{k} ||X_N e_i||_2^2 \right) \right] \leq \exp\left( a N \sum_{i=1}^{k} ||Y_N e_i||_2^2 \right)  \]
		where $Y_N$ is a GOE matrix if $\beta =1$ or a GUE matrix $\beta =2$. Then, using the orthogonal/unitary invariance of the law of $Y_N$, we can also assume that in the right hand side, $e_i$ is the $i-th$ vector of the canonical basis. Let us compute then $ \E[ \exp( a N \sum_{i=1}^{k} ||Y_N e_i||_2^2)]$. Using the orthogonal/unitary invariance of the law of $Y_N$, we can assume without loss of generality that $e_i$ is the $i$-th vector of the canonical basis. Therefore, we have that 
		
		\begin{eqnarray*}
			\E\left[ \exp( a N \sum_{i=1}^{k} ||Y_N e_i||_2^2)\right] &=&  \exp \Big( a  \Big( \sum_{1 \leq i < j \leq k } 2 |a_{i,j}|^2 + \sum_{i=1}^k |d_i|^2  + \sum_{i = k+1}^N \sum_{j=1}^k |a_{i,j}|^2 \Big).
		\end{eqnarray*}
		In the case $\beta =1$, we have that the $a_{i,j}$ are of law $\mathcal{N}(0,1)$ and the $d_i$ are of law $\mathcal{N}(0,2)$ and therefore, computing the expectation above we have that provided $a< 1/4$: 
		\begin{eqnarray*}
			\E\left[ \exp( a N \sum_{i=1}^{k} ||Y_N e_i||_2^2)\right] &=& \Big(\frac{1}{\sqrt{1- 4a}}\Big)^{k(k+1)/2} \Big(\frac{1}{\sqrt{1- 2a}} \Big)^{(N-k)k}. 
		\end{eqnarray*}  
		In the case $\beta =2$, we have that the $a_{i,j}$ are such that $\Im a_{i,j}$ and $\Re a_{i,j}$ are independent of law $\mathcal{N}(0,1/2)$ and $d_i$ is of law $\mathcal{N}(0,1)$ so, provided $a < 1/2$, 
		\begin{eqnarray*}
			\E \Big[ \exp( a N \sum_{i=1}^{k} ||Y_N e_i||_2^2) \Big] &=& \Big(\frac{1}{\sqrt{1- a}}\Big)^{k(k-1)} \Big( \frac{1}{\sqrt{1 -2a}} \Big)^{k} \Big(\frac{1}{1- a/2} \Big)^{(N-k)k}.
		\end{eqnarray*}  
		
		Therefore for $a =1/8$, we can find some explicit constant $C$ such that
		\[ \E[ \exp( N \sum_{i=1}^{k} ||Y_N e_i||_2^2 /8) \leq \exp( kCN).\]
		Markov's inequality then gives for any $b \geq 0$
		\[ \Pp [ \sum_{i=1}^{k} ||Y_N e_i||_2^2 \geq k b ] \leq \exp( kN (C - b)). \]
		Now using equation \eqref{eqNet}, equation \eqref{bound} and Lemma \ref{Net}, there is some constant $C'$ such that: 
		\[ \Pp[ \hat{\nu}_N (x^2) \geq b ] \leq \exp(k N ( C' -b))  \]
		which prove the exponential tightness.

	\end{proof}

	\subsection{Asymptotics of the annealed spherical integral: Proof of Proposition \ref{AnnealedSI}}
	This computation is very similar to \cite{GuHu21}. For the upper bound, we show
	\[ \limsup_{N} \frac{1}{k(N)} \Big( \frac{2}{\beta N} \ln \E[ I_N(X_N,D_N)] - \sum_{i=1}^{2 k(N)} \frac{\theta_i^2}{2}  \Big) \leq  0 \]
	we refer the reader to \cite{GuHu21}  and in particular to the fact that: 
	\begin{eqnarray*}
		\E_X[I(X_N,D_N)] & =& \E_{X,U} \left[ \exp \Big( \sum_{1 \leq i < j \leq  N}  L_{i,j}( \beta \sqrt{N} (U^* D_N U)_{i,j}) + \sum_{1 \leq i \leq N} L_{i,i}(\frac{\beta}{2} \sqrt{N} (U^* D_N U)_{i,i}) \Big) \right] \end{eqnarray*}
	with $U$ a Haar -distributed matrix and $L_{i,j}$ being the Laplace transform of the unrenormalized entry $(i,j)$ of $X_N$. Because $L_{i,j}( z) \leq \frac{|z|^2}{2 \beta}$ for $i \neq j$ and $L_{i,i}(z) \leq \frac{\Re(z)^2}{\beta}$ we get the upper bound.  
	For the lower bound, we want to use a Taylor expansion of the $L_{i,j}$ near zero. For this, we want to prove that the quantities $\sqrt{N} (U D_N U)_{i,j}$ remain small for all off-diagonal entries. More precisely, we have the following lemma, whose proof will be deferred to Appendix~\ref{app:A}:
	\begin{lemma}\label{excessmass}
		For $k(N) = o(N/\ln N)$ and for $\epsilon > 0$ let  $A^{(\epsilon)}_N$ be the following random variable: 
		\[ A^{(\epsilon)}_N:= \frac{1}{k(N)}\sum_{i,j}  \mathds{1}_{ \beta\sqrt{N}| ( U^* D_N U)_{i,j}|/2 \geq \epsilon} |( U^* D_N U)_{i,j}|^2  \]
		Then  $A^{(\epsilon)}_N$ converges in probability toward $0$. 
	\end{lemma}
	
	We now prove the lower bound. First we remind that thanks to the sharp sub-Gaussian character of our entries, there is a function $\delta: \R^+ \to \R^+$ converging to $0$ in $0$ and such that: 
	\[ L_{i,j}(z) \geq \frac{|z|^2(1 - \delta(|z|))}{ 2^{\beta - \mathds{1}_{i = j}}}. \]
	 Indeed, this i just a Taylor expansion and for this we only to prove that the third derivative of $L_{i,j}$ is uniformly bounded in a neighborhood of zero. For convenience sake, let us just look at the case $i\neq j$ and $\beta=1$ as the other are very similar. We have 
	
\begin{eqnarray*}
	 |L'''_{i,j} (t) |
	&=&| \E[a_{i,j}^3e^{ta_{i,j}}] |\\
	& \leq& \E[|a_{i,j}^3|e^{ta_{i,j}}] \\
	& \leq & \E[(a_{i,j}^2 + a_{i,j}^4)e^{ta_{i,j}}] \\
	&\leq & 24 \E[ \cosh( a_{i,j}) e^{ta_{i,j}}] \\
	& \leq & 12 L_{i,j}( t- 1) + 12 L_{i,j}(t + 1) \\
	& \leq & 12( e^{(t-1)^2/2} +  e^{(t+1)^2/2})
	\end{eqnarray*}
	which leads to the desired bound on $L'''_{i,j}$ and then to the existence of $\delta$.
	For $\epsilon > 0$,
	\begin{align*}
		&\Big( \sum_{1 \leq i < j \leq  N}  L_{i,j}( \beta \sqrt{N} (U^* D_N U)_{i,j}) + \sum_{1 \leq i \leq N} L_{i,i}(\frac{\beta}{2} \sqrt{N} (U^* D_N U)_{i,i}) \Big) 
		\\&\geq \beta  N(1 - \delta(\epsilon) ) \Big( \sum_{1 \leq i < j \leq  N}  |(U^* D_N U)_{i,j}|^2 \mathds{1}_{|\sqrt{N}(U^* D_N U)_{i,j}| \leq \epsilon} + \frac{1}{2} \sum_{1 \leq i \leq N}  (U^* D_N U)_{i,i}^2 \mathds{1}_{|\sqrt{N}(U^* D_N U)_{i,i}| \leq \epsilon}\Big)
		\\&\geq \beta(1 - \delta(\epsilon)) N k(N) \Big( \frac{1}{2 k(N)} \sum_{1 \leq i \leq 2k(N)} \frac{(\theta^N_i)^2}{2} - A_N^{(\epsilon)} \Big).
	\end{align*}
	Therefore, one can write for every $\eta >0$: 
	\[ \E[ I_N(X_N,D_N)] \geq \Pp[ A_N^{(\epsilon)} \leq \eta] \exp\Big( \beta N k(N) (1 - \delta(\epsilon))\Big( \frac{1}{2 k(N)} \Big(\sum \frac{(\theta^N_i)^2}{2} - \eta \Big) \Big) \Big) \]
	and therefore, since $ \Pp[ A_N^{(\epsilon)} \leq \eta]$ converges to $1$, taking the $\ln$, dividing by $N k(N)$ and letting $N$ to $\infty$ and then $\epsilon$ to $0$ and then $\eta$ to $0$ gives the lower bound. 
	
	\subsection{Large deviation upper bound}
	In this subsection and the following subsection, we are going to use frequently the quantile function $Q_{\nu}$ of a probability measure $\nu$. We remind the definition and the classical properties of this function that we will use: 
	
	\begin{defi}
		Let $\mu \in \mathcal{P}(\R)$. We define $Q_{\mu}:]0,1[ \to \R \cup \{ \pm \infty\}$ the \emph{quantile function} of $\mu$ with the following expression for: 
		\begin{equation}\label{eq:quantile}
		Q_{\mu}(p) = \inf \{ x \in \R: p \leq \nu(]- \infty, x ]) \} .
		\end{equation}
	\end{defi}
	In particular, a well-known property of $Q_{\mu}$ is the following ``change of variable" formula:
	\begin{prop}
		If $\mu \in \mathbb{P}(\R)$ and $f \in \mathcal{B}(\R)$ then $f$ in $\mu$-integrable if and only if $f \circ Q_{\mu}$ is Lebesgue integrable on $[0,1]$ and then: 
		\[ \int_{\R} f(x) d \mu(x) = \int_0^1 f(Q_{\mu}(x)) dx.\]
	\end{prop}
	Lastly, it is well known that weak convergence is equivalent to the convergence of quantile function: 
	
	\begin{prop}
		Let $(\mu_N)_{N \in \N}$ be a sequence of elements of $\mathcal{P}(\R)$ and $\mu \in \mathcal{P}(\R)$. Then we have: 
		\[ \lim \mu_N = \mu \Leftrightarrow \lim Q_{\mu_N}=Q_{\mu} \text{ Lebesque a.e. on $[0,1]$}. \]
	\end{prop}
	
	For every 
	$\nu \in \mathcal{P}(\R)$ and every $a' < \mathcal{I}(\nu)$, we want to prove that there exists a neighborhood $\mathcal{V}$ of $\nu$ (for the topology of the convergence in law) such that 
	\[\limsup_{N \to \infty} \frac{1}{\beta  N k(N)} \ln \Pp[ \hat{\nu}_N \in \mathcal{V} ] \leq -a'. \]
	
	First, let us look at the case of $\nu$ that are such that either $\nu(] - \infty, -2]) \neq 1/2$ or $\nu([2, + \infty[) \neq 1/2$. In particular, this implies the existence of $a \in ] -2,2[$ such that either $\nu(] - \infty, a]) < 1/2$ or $\nu([ a, + \infty[) < 1/2$. Let us assume that we are in the first case. One can then choose $b < 1/2$ such that $\mathcal{V}:= \{ \nu' \in \mathcal{P}([0,1]): \nu'(] - \infty,a]) <  b \}$ is a neighborhood of $\nu$ for the weak topology. Furthermore, if $\hat{\nu}_N \in \mathcal{V}$, it implies the existence of $i \in [1, k(N)]$ such that, either $\lambda_{N -i+1}^N \geq a$. However, we have for $N$ large enough that 
	\[ \{ \exists i \in [1, k(N)]: \lambda_{N - i+1}^N > a   \} \subset \{ d( \hat{\mu}_N, \sigma) > \epsilon(N) \} \]
	where $\epsilon(N)$ is such that Assumption \ref{assum2} is satisfied. Therefore, using this assumption, the upper bound is satisfied.
	We denote $\mathcal{V}_{\nu,\epsilon}$ the following neighborhood of $\nu$: 
	\[ \mathcal{V}_{\nu,\epsilon}:= \{ \nu' \in \mathcal{P}(\R): d(\nu',\nu) < \epsilon \} \]
	where $d$ is a distance on $\mathcal{P}(\R)$ metrizing the topology of the convergence in law and $\mathcal{A}_{\nu,\epsilon,M}$ the following event:
	\begin{equation}\label{eq:event}
		\mathcal{A}_{\nu,\epsilon,M}:= \{ \hat{\nu}_N(x^2) \leq M, d(\hat{\mu}_N, \sigma) \leq \epsilon(N), \hat{\nu}_N \in \mathcal{V}_{\nu,\epsilon} \}
	\end{equation}
	Then since 
	\[ \Pp[ \hat{\nu}_N \in \mathcal{V}_{\nu,\epsilon} ] \leq  \Pp[ \mathcal{A}_{\nu,\epsilon,M} ] + \Pp[ d( \hat{\mu}_N, \sigma) \geq \epsilon(N)] + \Pp[ \hat{\nu}_N(x^2) \geq M] \]
	Since there exists $M >0$ such that: 
	\[ \limsup_{N \to \infty} \frac{1}{\beta k(N) N} \ln \Pp[ d( \hat{\mu}_N, \sigma) \geq \epsilon(N)] = - \infty \text{ and } \limsup_{N \to \infty} \frac{1}{\beta k(N) N} \ln \Pp[ \hat{\nu}_N(x^2) \geq M] < - a' \]
	it only remains to prove that for any $a'$ such that $0 < a' < \mathcal{I}(\nu)$:
	\[  \limsup_{N \to \infty} \frac{1}{\beta k(N) N} \ln \Pp[\mathcal{A}_{\nu,\epsilon,M}]  \leq - a'. \]
	
	Let us take for every $N\in \N$, $(\theta_i^N)_{-k(N) \leq i \leq k(N) \atop i \neq 0 }$ defined as follows for some parameter $K >0$
	\[\forall i = 1,\dots, k(N), ~\theta_{-i}^N := \max\left( - K, G_{\sigma}\left(Q_{\nu}\left( \frac{ i - 1/2}{2 k(N)}\right) \right)^{-1} \right) \]
	\[ \forall i = 1,\dots, k(N),~\theta_{i}^N := \min\left( G_{\sigma}\left( Q_{\nu}\left(1 - \frac{ i - 1/2}{2 k(N)}\right) \right)^{-1}, K\right).\]
	
	By taking $D_N = \diag( \theta^N_{- k(N)}, \dots \theta^N_{-1}, \theta^N_1 ,\dots, \theta^N_{k(N)},0,\dots,0)$, we have
	
	\begin{eqnarray*}
		\Pp[ \mathcal{A}_{\nu,\epsilon,a'}] &=& \E \Big[ \frac{ I_N( X_N, D_N)}{I_N(X_N,D_N)} \mathds{1}_{ \mathcal{A}_{\mu,\epsilon,a'}} \Big]
		\\
		& \leq  & \E[ \mathds{1}_{ \mathcal{A}_{\nu,\epsilon,M}} I_N(D_N,X_N)] e^{- \beta N k(N) (M_N(\nu,\epsilon,M)  + o(1))} \\
		\end{eqnarray*}
	where
	\[ M_N(\nu,\epsilon,M) = \inf_{ ( \lambda_i) \in \mathcal{E}^N_{\nu,\epsilon,M} } \frac{1}{2 k(N)}\sum_{i= - k(N) \atop i \neq 0}^{k(N)} J(\theta_i^N, \lambda_i, \sigma)  \]
	with
	\[ \mathcal{E}^N_{\nu,\epsilon,M}:= \bigg\{ (\lambda_i)_{- k(N) \leq i \leq k(N) \atop i \neq 0} \in (\R^{-,*})^{k(N)} \times (\R^{+,*})^{k(N)}: d\bigg(\frac{1}{2 k(N)}\sum \delta_{\lambda_i},\nu\bigg) \leq \epsilon, \sum \lambda_i^2 \leq 2 k(N) M \bigg\}. \]
	and where we used Theorem \ref{generalizedSIconv} to argue that 
	\[ \frac{ I_N( X_N, D_N)}{I_N(X_N,D_N)} \mathds{1}_{ \mathcal{A}_{\mu,\epsilon,a'}} \leq e^{- \beta N k(N) (M_N(\nu,\epsilon,M)  + o(1))} \]
	Then, we have  
		
			\begin{eqnarray*}
		\Pp[ \mathcal{A}_{\nu,\epsilon,a'}]	& \leq &\E[I_N(D_N,X_N)] e^{- \frac{\beta}{2}N k(N) (M_N(\nu,\epsilon,M)+o(1)) } \\
		&  \leq & e^{ \beta N k(N) (\frac{1}{2k(N)} \sum_{i=-k(N) \atop i \neq 0}^{k(N)} \frac{(\theta_i^N)^2}{2} - M_N(\nu,\epsilon,M) +o(1)))} \\
 &\leq& e^{  - \beta N k(N) ( \tilde{M}_N(\nu,\epsilon,M) +o(1))} \end{eqnarray*}
	where
	\[ \tilde{M}_N(\nu,\epsilon,M):= M_N(\nu,\epsilon,M) - \frac{1}{2 k(N)}\sum_{i= - k(N) \atop i \neq 0}^{k(N)} \frac{(\theta_i^N)^2}{2}  = \inf_{ ( \lambda_i) \in \mathcal{E}^N_{\nu,\epsilon,M} } \frac{1}{2 k(N)}\sum_{i= - k(N) \atop i \neq 0}^{k(N)} \Big( J(\theta_i^N, \lambda_i, \sigma) - \frac{(\theta_i^N)^2}{2} \Big) \]
	and where we used Proposition \label{AnnealedSI} to approximate $ \E[ I_N(D_N,X_N)]$.
	We only need to prove that for any $ a' < \mathcal{I}(\nu)$, there exists $K >0$ such that:
	\[ \lim_{\epsilon \to 0 } \liminf_{N \to \infty} \tilde{M}_N(\nu,\epsilon,M) > a' ,  \]
	If we denote for $(\lambda_i) \in \mathcal{E}^N_{\nu,\epsilon, M}$, $\hat{\nu}= \frac{1}{2k(N)} \sum \delta_{\lambda_i}$, then one can notice that 
	\[ \sum_{i= - k(N) \atop i \neq 0}^{k(N)} \Big( J(\theta_i^N, \lambda_i, \sigma) - \frac{(\theta_i^N)^2}{2} \Big) = \int_{[0,1]} J( \Theta^N(x), Q_{\hat{\nu}}(x), \sigma) - \frac{(\Theta^N(x))^2}{2} dx  \]
	where $\Theta^N = \sum_{i= 1}^{k(N)}\theta_{-i}^N \mathds{1}_{[ (i -1)/(2 k(N)), i/(2k(N))[} + \sum_{i= 1}^{k(N)}\theta_{i}^N \mathds{1}_{[ 1 - i/(2 k(N)), 1 - (i-1)/(2k(N))[}$
	and therefore
	\[ \lim_{\epsilon \to 0}\liminf_{N \to \infty} \tilde{M}_N(\nu,\epsilon,M) \geq   \lim_{\epsilon \to 0} \liminf_{N \to \infty} \inf_{\tilde{\nu} \in \mathcal{V}_{\nu,\epsilon}} \int_{[0,1]} J( \Theta^N(x), Q_{\tilde{\nu}}(x), \sigma) - \frac{(\Theta^N(x))^2}{2} dx.  \]
	Therefore, if $a'< \mathcal{I}(\nu)$, one only needs to find a parameter $K >0$ such that the right hand side of the preceding equation is greater than $a'$.
	
	This is equivalent to finding $K >0$ such that for any sequence of probability measure $(\nu_N)_{N \in \N}$ converging toward $\nu$: 
	\[ \liminf_{N \to \infty} \int_{[0,1]} J( \Theta^N(x), Q_{{\nu_N}}(x), \sigma) - \frac{(\Theta^N(x))^2}{2} dx  > a' \]
	Therefore, we need only the following lemma:

	\begin{lemma}\label{approx}
		Let $\nu \in \mathcal{P}(\R)$ such that  $\nu(] - \infty,-2]) = \nu([2, + \infty[) = 1/2$ and $(\nu_N)_{N \in \N}$ a sequence of probability measures that converges in law toward $\nu$ and such that  $\nu_N(] - \infty,-2]) = \nu_N([2, + \infty[) = 1/2$. Then, for every $a' < \mathcal{I}(\nu)$ there is some $K > 0$ such that: 
		\[\liminf_{N \to \infty} \int_{[0,1]} \Big( J( \Theta^N(x), Q_{\nu_N}(x), \sigma) - \frac{(\Theta^N(x))^2}{2} \Big)dx > a'. \]
	\end{lemma}
	\begin{proof}
		For $K > 1$, we let for $x \geq 0$
		\[ I^K(x) = \sup_{ 0 \leq \theta \leq K} \Big( J(\theta,x,\sigma) - \frac{\theta^2}{2}  \Big) \]
		and for $x \leq  0$ $I^K(x) = I^K (-x) = \sup_{ 0 \geq  \theta \geq - K} \Big( J(\theta,x,\sigma) - \frac{\theta^2}{2}  \Big).$
		It is clear that $I^K \geq 0$ and for every $x$, $K \mapsto I^K(x)$ is increasing on $\R^+$ and converges toward $I(x)$. Therefore, given that $\mathcal{I}(\nu) > a'$, there is some $K >$ such that 
		\[ \int_{\R} I^K(x)dx > a'. \]

		Furthermore, for $x \geq 0$ we have that: 
		\[ \mathrm{argmax}_{K \geq \theta\geq 0}\Big( J(\theta,x,\sigma) - \frac{\theta^2}{2}  \Big) = \begin{cases} 0 \text{ if } x \leq 2 \\
			\min(G_{\sigma}(x)^{-1},K) \text{ if } x \geq 2 \end{cases} \]
		and 
		\[ \mathrm{argmax}_{-K \leq \theta\leq 0}\Big( J(\theta,x,\sigma) - \frac{\theta^2}{2}  \Big) = \begin{cases} 0 \text{ if } x \leq 2 \\
			\max(G_{\sigma}(x)^{-1},-K) \text{ if } x \geq 2 \end{cases}. \]
		
		For every $x$ such that $Q_{\nu}$ is continuous in $x$, we have that for $x > 1/2$,  $\lim_{N \to \infty}\Theta^N(x) = \Theta(x):= \min (G_{\sigma}^{-1}(Q_{\nu}(x)),K)$ and for $x <1/2$, $\lim_{N \to \infty}\Theta^N(x) = \max (G_{\sigma}^{-1}(Q_{\nu}(x)),- K)$. So since $Q_{\nu}$ has an at most countable number of points of discontinuity, $x \mapsto ( J( \Theta^N(x), Q_{\nu_N}(x), \sigma) - \frac{(\Theta^N(x))^2}{2} )$ converges Lebesgue almost everywhere toward
		\[ x \mapsto \Big( J( \Theta(x), Q_{\nu}(x), \sigma) - \frac{(\Theta(x))^2}{2} \Big) = I^{K}(x). \]
		
		Using Fatou's lemma finishes the proof. 
	\end{proof}

	Therefore we conclude that 
	\[ \limsup_{N \to \infty} \frac{1}{\beta k(N)}  \ln \Pp[\mathcal{A}_{\nu,\epsilon,M}] \leq -a' \]
	and the large deviation upper bound is proved. 
	
	\subsection{Large deviation lower bound}
	In fact, we are going to prove the following large deviation lower bound 
	\begin{prop}\label{prop:lowerboundLDP}
		If $\nu \in \mathcal{P}(\R)$ such that there is $K >0$ so that $\nu( [ - K ; -2  ]) =  \nu( [ 2, K] ) = 1/2$. Then for the rate function $\mathcal I$ defined in \eqref{eq:calI},
		\[	\lim_{\epsilon \to 0} \liminf_N \frac{1}{\beta N k(N)} \ln \Pp[ d(\hat{\nu}_N, \nu) \leq \epsilon] \geq  - \mathcal{I}(\nu). \]
	\end{prop}
	
	Let us first prove that this proposition is sufficient for our large deviation lower bound. If we take $\nu$ such that $\mathcal{I}(\nu) < + \infty$. Then one can define for every $K >0$, $\nu^{(K)}$ the following probability measure 
	
	\[ \forall A \in \mathcal{B}(\R), \nu^{(K)}(A) = \frac{\nu(A \cap ( [-K, -2 ] \cup [ 2 , K]))}{\nu(  [-K, -2 ] \cup [ 2 , K])}. \]
	When $K$ tends to $\infty$, $\nu^{(K)}$ tends to $\nu$. Furthermore, since we have that 
	\[ \mathcal{I}( \nu^{(K)}) = \frac{1}{\nu(  [-K, -2 ] \cup [ 2 , K])} \int_{ -K}^{ K} I(x) d \nu(x) \]
		it is easy to see that $\mathcal{I}(\nu^{(K)})$ also tends to $\mathcal{I}(\nu)$.

	 Therefore, if $b > \mathcal{I}(\nu)$, for every $\epsilon >0$ there is $K$ such that $\mathcal{I}(\nu^{(K)})  < b $ and  $d( \nu^{(K)}, \nu) \leq \epsilon/2$
	Then $\Pp[ d( \hat{\nu}_N, \nu) \leq \epsilon] \geq \Pp[ d( \hat{\nu}_N, \nu^{(K)}) \leq \epsilon/2]$. Therefore, using the preceding proposition, we have
	\[ \liminf_{N \to \infty} \frac{1}{\beta N k(N)} \ln \Pp[ d(\hat{\nu}_N, \nu^{(K)}) \leq \epsilon/ 2] \geq  - \mathcal{I}(\nu^{(K)}) \geq -b \]
	and so:
	\[	\lim_{\epsilon \to 0} \liminf_{N \to \infty} \frac{1}{\beta N  k(N)} \ln \Pp[ d(\hat{\nu}_N, \nu) \leq \epsilon] \geq  - b. \]
	Since this is true for any $b > \mathcal{I}(\nu)$, the large deviation lower holds for any $\nu$.

	We will now prove Proposition~\ref{prop:lowerboundLDP}: 
	\begin{proof}[Proof of Proposition~\ref{prop:lowerboundLDP}]
		Let $\nu$ be as is in the assumption of the Proposition, then for every $N$ and $ i= 1,\dots, k(N)$, we define $\theta_i^N$ and $\theta_{-i}^N$ by
		\[ \theta_i^N = G_{\sigma} \left( Q_{\nu} \left( 1 - \frac{ i - 1/2}{2 k(N)} \right) \right)^{-1} \quad \text{and} \quad   \theta_{-i}^N = G_{\sigma}\left( Q_{\nu}\left( \frac{ i- 1/2}{2 k(N)} \right) \right)^{-1}. \]
		We define $\Pp^{\theta^N}$ as the following tilt on $\Pp$: 
		\[ d \Pp^{\theta^N}(X_N) = \frac{ I_X(D_N,X_N)}{\E_X[I(D_N,X_N)]} d \Pp(X_N) .\]
		
		Then, going back to the computation on the upper bound, we can write that 
		\begin{eqnarray*}
			\Pp[ \mathcal{A}_{\nu,\epsilon,M}] &=& \E \Big[ \frac{ I_N( X_N, D_N)}{I_N(X_N,D_N)} \mathds{1}_{ \mathcal{A}_{\nu,\epsilon,M}} \Big]
			\\
			& \geq  & \E[ \mathds{1}_{ \mathcal{A}_{\nu,\epsilon,M}} I_N(D_N,X_N)] e^{- \beta N k(N) (S_N(\nu,\epsilon,M)  + o(1))} \\
			& \geq &\frac{\Pp^{\theta^N}[\mathcal{A}_{\nu, \epsilon,M}]}{ \E[I_N(D_N,X_N)]} e^{- \frac{\beta}{2}N k(N) (S_N(\nu,\epsilon,M)+o(1)) } \\
			&  \geq &\Pp^{\theta^N}[\mathcal{A}_{\nu, \epsilon,M}] e^{ \beta N k(N) (\frac{1}{2k(N)} \sum_{i=-k(N) \atop i \neq 0}^{k(N)} \frac{(\theta_i^N)^2}{2} - S_N(\nu,\epsilon,M) +o(1)))} 
		\end{eqnarray*}
		where 
		\[ S_N(\nu,\epsilon,M) = \sup_{ ( \lambda_i) \in \mathcal{E}^N_{\nu,\epsilon,M} } \frac{1}{2 k(N)}\sum_{i= - k(N) \atop i \neq 0}^{k(N)} J(\theta_i^N, \lambda_i, \sigma).  \]
		If we let 
		\[\tilde{S}_N( \nu, \epsilon,M) = S_N( \nu, \epsilon,M) -\sum_{i=-k(N) \atop i \neq 0}^{k(N)} \frac{(\theta_i^N)^2}{2} \]
		we have:
		\[ \Pp[ \mathcal{A}_{\nu,\epsilon,M}] \geq  \Pp^{\theta^N}[ \mathcal{A}_{\nu,\epsilon,M}] e^{- \beta N k(N)( \tilde{S}_N(\nu,\epsilon,M) + o(1))}. \]

		First, we want to prove that: 
		\[ \lim_{\epsilon \to 0} \limsup_{N \to \infty}  \tilde{S}_N(\nu,\epsilon,M) \leq \mathcal{I}(\nu).  \]
		Again one can notice that: 
		\[  \tilde{S}_N(\nu,\epsilon,M) =  \sup_{(\lambda_i) \in \mathcal{E}_{\nu, \epsilon,M}} \int_{[0,1]} \Big( J( \Theta^N(x), Q_{\hat{\nu}}(x), \sigma) - \frac{(\Theta^N(x))^2}{2} \Big) dx .\]
		And so 
		\[ \lim_{\epsilon \to 0} \limsup_N  \tilde{S}_N(\nu,\epsilon,M) \leq \lim_{\epsilon \to 0} \limsup_{N \to \infty} \sup_{\tilde{\nu} \in \mathcal{V}_{\nu,\epsilon,M}} \bigg( \int_{[0,1]} J( \Theta^N(x), Q_{\tilde{\nu}}(x), \sigma) - \frac{(\Theta^N(x))^2}{2} dx \bigg) \]
		where:
		\[ \mathcal{V}_{\nu,\epsilon,M} = \{ \nu' \in \mathcal{P}(\R): d(\nu',\nu) \leq \epsilon, \nu'(x^2) \leq M \}.\]
		
		Therefore, one only needs to prove that the right hand side is lower than $\mathcal{I}( \nu)$. This is again equivalent to proving that for every sequence $(\nu_N)_{N \in}$ converging toward $\nu$ such that $\nu_N(x^2) \leq M$ for every $N$.
		\[ \limsup_{N \to \infty} \sup_{\tilde{\nu_N} \in \mathcal{V}_{\nu,\epsilon,M}} \int_{[0,1]} J( \Theta^N(x), Q_{\tilde{\nu}_N}(x), \sigma) - \frac{(\Theta^N(x))^2}{2} dx= \mathcal{I}(\nu). \]
		\begin{lemma}\label{approx2}
			Let $(\nu_N)_{N \in \N}$ be a sequence of probability measures that converges in law toward $\nu$ and such that $\nu_N(\R^{ -} ) = \nu_N(\R^{ +} ) = 1/2$ and $\nu_N(x^2) \leq M$. 
			\[\limsup_N \int_{[0,1]} \Big( J( \Theta^N(x), Q_{\nu_N}(x), \sigma) - \frac{(\Theta^N(x))^2}{2} \Big)dx = \mathcal{I}(\nu). \]
		\end{lemma}
		\begin{proof}
			Since $\Theta^N$ converges dx-a.e. toward $\Theta$,  $x \to J( \Theta^N(x), Q_{\nu_N}(x), \sigma) - \frac{(\Theta^N(x))^2}{2}$ converges $dx$-almost everywhere toward $x \mapsto J( \Theta(x), Q_{\hat{\nu}}(x), \sigma) - \frac{(\Theta(x))^2}{2}$, where we remind that $\Theta(x) = G_{\sigma}(Q_{\nu}(x))^{-1}$. Furthermore
			
			\[ \left| J( \Theta^N(x), Q_{\nu_N}(x), \sigma) - \frac{(\Theta^N(x))^2}{2} \right| \leq K |Q_{\nu_N}(x)| + \frac{K^2}{2}. \]
			Furthermore,
			\[ \int_0^1  (|Q_{\nu_N}(x)| + \frac{K^2}{2})^2 dx = \int (K|t| + K^2)^2 d\nu_N(t) \leq K^2 M + \sqrt{M} K^3 + K^4 \]
			since $\nu_N(x^2) \leq M$. Therefore, $x \to J( \Theta^N(x), Q_{\nu_N}(x), \sigma) - \frac{(\Theta^N(x))^2}{2}$ is bounded in $L^2( [0,1])$ and converges almost everywhere. Therefore, it converges in $L^1$, and so: 
			\[ \lim_N \int_0^1 \Big(J( \Theta^N(x), Q_{\nu_N}(x), \sigma) - \frac{(\Theta^N(x))^2}{2} \Big) dx = \int_0^1 \Big(J( \Theta(x), Q_{\nu}(x), \sigma) - \frac{(\Theta(x))^2}{2} \Big) dx = \int_{\R} I(x) d \nu(x) = \mathcal{I}(\nu). \]
		\end{proof}
		
		Therefore: 
		\[  \Pp[ \mathcal{A}_{\nu,\epsilon,M}] \geq \Pp^{\theta^N}[ \mathcal{A}_{\nu,\epsilon,M} ] e^{ -N k(N)( \mathcal{I}(\nu) + o(1))}. \]
		To conclude, we only need to prove that 
		\begin{lemma} For $M$ large enough, 
			\[ \lim_N  \Pp^{\theta^N}[ \mathcal{A}_{\nu,\epsilon,M}] = 1.\]
		\end{lemma}
		To prove this result, we will first the following exponential tightness lemma for the tilted measure: 
		\begin{lemma}\label{tiltexptight}
			There exist a positive sequence $\epsilon'(N)$ converging toward $0$ such that
			\[ \limsup_{N \to \infty} \frac{1}{N k(N)} \ln \Pp^{\theta^N} [ d( \hat{\mu}_N, \sigma) \geq \epsilon'(N)]  = - \infty .\]
			For every $M>0$, there exists $L' > 0$ such that 
			\[ \limsup_{N \to \infty} \frac{1}{N k(N)} \ln \Pp^{\theta^N} [ \hat{\nu}(x^2) \geq L ] \leq - M. \]
			In particular $\hat{\nu}_N$ is exponentially tight.
		\end{lemma}
		The proof of this lemma is postponed to the appendix. Then, using this lemma, we want to prove the following lemma: 
		\begin{lemma}\label{tiltUB}
			Let $\nu' \in \mathcal{P}(\R)$ such that $\nu' \neq \nu$. There exists $\epsilon >0$ such that: 
			\[ \limsup_{N \to \infty} \frac{1}{ N k(N)} \ln \Pp[ d(\hat{\nu}_N, \nu') \leq \epsilon ] < 0 \]
		\end{lemma}
		Using the exponential tightness of $\hat{\nu}_N$ under $\Pp^{\theta^N}$, this proves that under $\Pp^{\theta^N}$, $\hat{\nu}_N$ converges in probability toward $\nu$. This finishes the proof of Proposition~\ref{prop:lowerboundLDP}.
		
			\end{proof}
		
		We now prove  Lemma~\ref{tiltUB}.
		
		\begin{proof}[Proof of Lemma \ref{tiltUB}]
			First, if $\nu'$ does not satisfy, $\nu'(] - \infty, -2]) = \nu'([2, +\infty)[) = 1/2$, then the result is a consequence of Lemma \ref{tiltexptight}. So we can assume that $\nu'(] - \infty, -2]) = \nu'([2, +\infty)[) = 1/2$.
			
			For a given $K' > 0$, we define for every $N \in \N$, $i \in[1,k(N)]$
			\[ \theta'^N_i = \max\left(G_{\sigma}\left(Q_{\nu'}\left(1 - \frac{i -1/2}{2 k(N)}\right) \right)^{-1},K' \right), \theta'^N_i = \min\left(G_{\sigma}\left(Q_{\nu'}\left(\frac{i -1/2}{2 k(N)} \right) \right)^{-1},- K' \right) \]
			We also define for $x \in ]0,1[$, 
			\[ \Theta'(x) = \max( \min( G_{\sigma}^{-1}( Q_{\nu'}(x)), K'), - K') \] 
			and for $N \in \N$: 
			\[\Theta'^N = \sum_{i= 1}^{k(N)}\theta'^N_{-i} \mathds{1}_{[ (i -1)/(2 k(N)), i/(2k(N))[} + \sum_{i= 1}^{k(N)}\theta'^N_{i} \mathds{1}_{[ 1 - i/(2 k(N)), 1 - (i-1)/(2k(N))[} \]

			We let $D'_N = \diag( \theta'^N_i)$. Then we can do the same computation as for the upper bound. For $M, \epsilon >0$ let us denote:
\begin{equation}\label{eq:event}
	\mathcal{A}'_{\nu,\epsilon,M}:= \{ \hat{\nu}_N(x^2) \leq M, d(\hat{\mu}_N, \sigma) \leq \epsilon'(N), \hat{\nu}_N \in \mathcal{V}_{\nu,\epsilon} \}
\end{equation}
where we choose $\epsilon'(N)$ as in Lemma \ref{tiltexptight}. 
Here, denoting $\E^{\theta^N}$ the expectation taken according to the probability measure $\Pp^{\theta^N}$, we have			
			\begin{eqnarray*}
				\Pp^{\theta_N}[ \mathcal{A}'_{ \nu', \epsilon,M}] &\leq&  \E^{\theta^N}[ \frac{I(D_N',X_N)}{I(D'_N,X_N)}\mathds{1}_{\mathcal{A}'_{ \nu', \epsilon,M}}] 
				\\
				&\leq & \frac{\E[ I(D_N,X_N) \frac{I(D_N',X_N)}{I(D'_N,X_N)}\mathds{1}_{\mathcal{A}'_{ \nu', \epsilon,M}}]}{\E[ I(D_N,X_N)]} \\
				&\leq&  \E[ \frac{I(D_N',X_N)}{I(D'_N,X_N)}\mathds{1}_{\mathcal{A}'_{ \nu', \epsilon,M}}] \exp \left( \beta k(N) N ( M_N(\nu',\epsilon,M)
				- \frac{1}{2 k(N)} \sum_i  \frac{(\theta_i^N)^2}{2} + o(1)) \right) \\
				& \leq & \exp \Big( \beta k(N) N ( \frac{1}{2 k(N)}\sum_i \frac{(\theta'^N_i)^2}{2} - S'_N(\nu',\epsilon,M)  \\
				& &+ M_N(\nu',\epsilon,M)
				- \frac{1}{2 k(N)} \sum_i\frac{  (\theta^N_i)^2}{2} + o(1)) \Big)
			\end{eqnarray*}
			Where we remind that 
			\[ M_N(\nu',\epsilon,M) = \inf_{ ( \lambda_i) \in \mathcal{E}^N_{\nu',\epsilon,M} } \frac{1}{2 k(N)}\sum_{i= - k(N) \atop i \neq 0}^{k(N)} J(\theta_i^N, \lambda_i, \sigma),  \]
			and 
			\[ S'_N(\nu,\epsilon,M) = \inf_{ ( \lambda_i) \in \mathcal{E}^N_{\nu',\epsilon,M} } \frac{1}{2 k(N)}\sum_{i= - k(N) \atop i \neq 0}^{k(N)} J(\theta'^N_i, \lambda_i, \sigma).  \]
			
			In a similar way as we argued for the upper bound and the lower bound, we have that: 
			
			\begin{multline*}
				\lim_{\epsilon \to 0} \limsup_N \Big( \frac{1}{2 k(N)}\sum_i \frac{(\theta'^N_i)^2}{2} - S'_N(\nu,\epsilon,M) 
				+ M'_N(\nu,\epsilon,M)
				- \frac{1}{2 k(N)} \sum_i\frac{(  \theta^N_i)^2}{2} \Big) \\
				\leq  \int_0^1 \frac{\Theta'^2(x)}{2}  - J(\Theta'(x),Q_{\nu'}(x), \sigma) + J(\Theta(x), Q_{\nu'}(x),\sigma) - \frac{\Theta(x)^2}{2} dx \end{multline*}
			First, we can notice that:
			\[ \int_0^1 \frac{\Theta'^2(x)}{2}  - J(\Theta'(x),Q_{\nu'}(x), \sigma) dx =  - \int I^{K'}(x) d \nu'(x).  \]
			And therefore the limit of this term when $K'$ tends to $+ \infty$ is $- \mathcal{I}(\nu').$
			
			Then, we can notice that:  
			\[ \mathcal{I}(\nu')  > \int_0^1 \left( J(\Theta(x), Q_{\nu'}(x),\sigma) - \frac{\Theta(x)^2}{2} \right)dx. \]
			Indeed, since for $x > 0$, $I(x) = \sup_{\theta \geq 0} (J( \theta, x, \sigma) -  \theta^2 /2)$ and for $x < 0$, $I(x) = \sup_{\theta \leq 0} (J( \theta, x, \sigma) -  \theta^2 /2)$, we have 
			\[ \mathcal{I}(\nu')  \geq  \int_0^1 \left( J(\Theta(x), Q_{\nu'}(x),\sigma) - \frac{\Theta(x)^2}{2} \right)dx. \]
			
			If we had equality, that would mean that for almost all $x \in [0,1]$, $\Theta(x) = G_{\sigma}( Q_{\nu'}(x))^{-1}$. Since  $\Theta(x) = G_{\sigma}( Q_{\nu}(x))^{-1}$, that would mean that $Q_{\nu}(x) = Q_{\nu'}(x)$ for almost all $x$ which implies that $\nu = \nu'$ which is excluded. Therefore we have that 
			\[ \limsup_{N \to \infty} \frac{1}{N k(N)}\ln \Pp^{\theta_N}[ \mathcal{A}'_{\nu',\epsilon,M}] <0 \]
			Then using Lemma \ref{tiltexptight}, we can choose $M$ large enough such that:
			\[ \limsup_{N \to \infty} \frac{1}{N k(N)} \ln \Pp^{\theta^N} [ \hat{\nu}(x^2) \geq M ] \leq -1. \]
			
			For such an $M$, one has:
			\[ \Pp^{\theta_N}[ d( \nu_N, \nu') < \epsilon] \geq \Pp^{\theta_N}[ \mathcal{A}'_{ \nu', \epsilon,M}] - \Pp^{\theta_N}[ \nu_N(x^2) > M ] - \Pp^{\theta_N}[ d(\hat{\mu}_N, \sigma) \geq \epsilon'(N)] \]
			and then using again Lemma \ref{tiltexptight}, one proves that:
			\[ \limsup_{N \to \infty} \frac{1}{N k(N)}\ln \Pp^{\theta_N}[ d( \nu_N, \nu') < \epsilon] <0. \]

		\end{proof}

	\subsection{Strengthening the large deviation principle}
	Using the inverse contraction principle, we can actually strengthen our large deviation principle to the topologies of the associated to the moments of order $p < 2$. More precisely if for $p \in ]0,2 [$, we denote:
	\[ \mathcal{P}_{p}( \R) = \{ \mu \in \mathcal{P}(\R): \int |x|^p d \mu(x)  < + \infty \}\]
	and $d_p$ the disatance on $\mathcal{P}_{p}( \R)$ defined by by:
	\[ d_p(\mu,\nu) = \sup\left\{\Big| \int f d \mu - \int f d \nu \big|: f \in \mathcal{C}(\R) \text{ such that }\forall x, |f(x)|\leq 1 + |x|^p \right\} \]
	
	We denote $\mathcal{T}_p$ the topology induced by $d_p$ on $\mathcal{P}_{p}( \R)$. Then we have the following theorem: 
	\begin{theo}
		For $p \in ]0,2[$, $\mathcal{I}$ is a good rate function on $\mathcal{P}_p(\R)$ withe the topology $\mathcal{T}_p$ and the large deviation principle of Theorem \ref{LDPTheo} extends to $\mathcal{P}_p(\R)$. 
	\end{theo}
	\begin{proof}
		This is an almost direct application of \cite[Corollary 4.2.6]{DZLD}. For this, one has to see that the $\{ \nu \in \mathcal{P}_p( \R): \nu(x^2) \leq M \} $ are compact sets of $\mathcal{P}_p( \R)$. Then Proposition \ref{exptight} gives the exponential tightness also for the topology $\mathcal{T}_p$.   
	\end{proof}

	\section{Large deviations of the extreme eigenvalues of an additive deformation of a Gaussian matrix}\label{sec:extremeeig}
	In this section, we prove a large deviation principle for a random matrix $X_N = Y_N + D_N$ where $Y_N$ is a GOE matrix if $\beta =1$ or a GUE matrix if $\beta =2$ and $D_N$ is a self adjoint constant matrix. First of all, we recall the following large deviation result for what occurs when $D_N$ is of rank $1$. 
	\begin{theo} \cite[Theorem 3.2]{Mai} 
		Let $\theta \geq 0$, $e \in \Ss^{\beta N -1}$, and $Y_N$ a GOE/GUE matrix. The largest eigenvalue of $X_N = Y_N + \theta e e ^*$ satisfy a large deviation principle with rate function $\beta  I_{\theta}/2$ where:
		\[ I_{\theta}(x) = \begin{cases} + \infty \text{ if }x < 2\\
			I(x) - J(\theta,x,\sigma) - \inf_{y \geq 2} (I(y) - J(\theta,y,\sigma)) \text{ if  }x > 2
		\end{cases}
		\]
		and the function $I$ was defined in \eqref{eq:Ix}.
	\end{theo}
	We now consider the growing rank case when $D_N = \diag( \theta^N_{- k(N)}, \dots \theta^N_{-1}, \theta^N_1 ,\dots, \theta^N_{k(N)},0,\dots,0)$ where $k(N)= o(N)$ and 
	\[
	\theta^N_{-k(N)} \leq \dots \leq \theta^N_{-1} \leq 0 \leq \theta^N_1\leq \dots \leq \theta_{k(N)}^N.
	\]
	Then we have the following result:
	\begin{theo}\label{theo:extremeeigGaussian}
		Assume that there is some probability measure $\xi$ such that: 
		\[ \lim_{N \to \infty} \frac{1}{2 k(N)} \sum_{i= -k(N),\dots,k(N) \atop i \neq 0} \delta_{\theta^N_i} = \xi \]
		and that there is $M > 0$ such that $\theta^N_i \leq M$. 
		Then, with the same notations as in Theorem~\ref{LDPTheo}, $(\hat{\nu}_N)_{N \in \N}$ satisfy a large deviation principle in speed $2 N k(N)$  with good rate function $\beta \mathcal{I}_{\xi} /2$ defined by:
		\[ \mathcal{I}_{\xi}(\nu) = \begin{cases}
			 \int_0^1 I_{ Q_{\xi}(t)}( Q_{\nu}(t)) d t &\text{if } \nu(] - \infty, -2]) = \nu([2, + \infty[) = \frac{1}{2} \\  + \infty &\text{otherwise}\end{cases} \]
		and the quantile function $Q_{\mu}$ was defined in \eqref{eq:quantile}.
	\end{theo}
	
	The key observation to prove this result is that, if $U$ is a Haar-distributed matrix independent from $X_N$, the law of $\tilde{X}_N= U X_N U^*$ can be expressed with the following density with regards to the law of the Gaussian invariant matrix $Y_N$:
	\[ d \tilde{X}_N = \frac{ I_N(Y_N,D_N)}{\E[I_N(Y_N,D_N)]} dY_N. \]
	
	Let us notice that since conjugating by $U$ does not impact the eigenvalues, we can study $\tilde{X}_N$ instead of $X_N$. Then, we can notice that:
	\[ I(Y_N,D_N) \approx \exp\left( N \beta k(N) \int_0^1 J(Q_{\xi}(t), Q_{\hat{\nu}_N}(t),\sigma)  dt\right) \]
	and so, heuristically, we end up with the following tilted large deviation principle whose rate function is, up to a constant:
	\begin{eqnarray*}
		 \mathcal{I}_{\xi}(\nu) &=& \mathcal{I}(\nu) - \int_0^1 J(Q_{\xi}(t), Q_{\nu_N}(t),\sigma) dt + C\\
		 &=& \int_0^1  I(Q_{\nu}(t)) - J(Q_{\xi}(t),Q_{\nu}(t), \sigma) dt + C \\
		 &=& \int_0^1 I_{ Q_{\xi}(t)}( Q_{\nu}(t)) d t + C. \end{eqnarray*}
	Recall that $\mathcal{I}$ and $I$ were defined in \eqref{eq:calI} and \eqref{eq:Ix}. 
	
	Using this approach, instead of considering $X_N$ as an additive deformation of $Y_N$, we will consider it as a tilt over the law of a Gaussian random matrix. We will denote $\Pp$ the probability measure such that the law of $X_N$ is GOE/GUE and $\tilde{\Pp}$ the tilted law by $I_N( X_N, D_N)$.  
	Here are the rigorous steps of the proof:
	\begin{itemize}
	
		\item We prove the following proposition:
		\begin{prop}\label{prop:J}
			On events of the form $\mathcal{A}_{\nu, \epsilon,M'}$ defined in \eqref{eq:event},	\[ \sup \Big| \frac{1}{2 N k(N)} \ln  I_N(X_N,D_N) -  \frac{\beta}{2}\int_0^1 J(Q_{\xi}(t), Q_{\hat{\nu}_N}(t),\sigma) \, dt \Big| = o(N) \]
			and 
			\begin{equation}\label{eq:intJ}
			\mathcal{J}_{\xi}:\nu \mapsto \int_0^1 J(Q_{\xi}(t),Q_{\nu}(t), \sigma) dt
			\end{equation}
			is continuous for the weak topology on this event.
		\end{prop} 
		\item Then, using the exponential tightness result of Lemma \ref{tiltexptight} for $\tilde{\Pp}$, we derive the large deviation upper and lower bound. 
	\end{itemize}
	\begin{proof}[Proof of Proposition~\ref{prop:J}]
	First, let us prove a slightly stronger result, that is that $(\nu, \xi) \mapsto \mathcal{J}_{\xi}(\nu)$ is continuous on the set $\{ (\nu, \xi) \in \mathcal{P}(\R)^2: \nu(] - \infty, -2]) = \nu( [2, + \infty[) = 1/2, \nu(x^2) \leq M', \text{supp } \xi \in [-M,M]\}$ where $M' > 0$. If $(\nu_n, \xi_n)$ is a sequence of couple of  probability measures in the aforementioned set converging weakly to $(\nu, \xi)$, then $Q_{\nu_n}$ converges almost everywhere toward $Q_{\nu}$ and $Q_{\xi_n}$ converges almost everywhere toward $Q_{\xi}$. Using the continuity of the function $J$, $t \mapsto J(Q_{\xi_n}(t),Q_{\nu_n}(t), \sigma)$ converges almost everywhere toward $t \mapsto J(Q_{\xi}(t),Q_{\nu}(t), \sigma)$. Furthermore:
	\[\int_{0}^1 |J(Q_{\xi_n}(t),Q_{\nu_n}(t), \sigma)|^2 dt \leq M^2 \int_0^1 |Q_{\nu_n}(t)|^2 dt \leq M^2 M'\]
	and so $t \mapsto J(Q_{\xi_n}(t),Q_{\nu_n}(t), \sigma)$ is bounded in $L^2$ and so it converges in $L^1$, which gives that $\lim_{n}\mathcal{J}_{\xi_n}(\nu_n) = \mathcal{J}_{\xi}(\nu)$.

	Then, using Theorem \ref{generalizedSIconv}, we conclude since 
	
	\[ \frac{1}{2 k(N)} \sum_{i=- k(N), i \neq 0}^{k(N)} J(\theta_i^N,\lambda_i^N,\sigma) = J_{\xi_N}( \hat{\nu}_N) \]
	where $\xi_N =\frac{1}{2 k(N)} \sum \delta_{\theta_i^N}$. 
	\end{proof}
	
	To conclude the proof of the LDP, we prove the lower bound as follows. Given $\nu \in \mathcal{P}(\R)$ such that $\mathcal{I}_{\xi}(\nu) < + \infty$, $\delta > 0 $, using Lemma~\ref{tiltexptight}, we choose $M' \geq 0$ such that if $E_M'$ is the event $\{\hat{\nu}_N(x^2) \geq M' \text{ or } d( \hat{\mu}_N, \sigma) \geq \epsilon(N) \}$, then both $\tilde{\Pp}[E_{M'}]$ is upper bounded by $\exp( - \beta N k(N) ( \mathcal{I}_{\xi}(\nu) + 1 + o(1)))$ and $\Pp[E_{M'}]$ is upper bounded by $\exp( - \beta N k(N) ( \mathcal{I}(\nu) + 1 + o(1)))$. Then for this $M'$, using Proposition \ref{prop:J}, for any $\eta >0$ we can find $\delta>0$ such that if $X_N$ is such that $d(\hat{\nu}_N,\nu) \leq \delta ,d( \hat{\mu}_N, \sigma) \leq \epsilon(N),  \hat{\nu}_N(x^2) \leq M'$ then:
		
\[\mathcal{J}_{\xi}( \nu) - \eta \leq \liminf_{N \to \infty} \frac{1}{\beta N k(N)} \ln  I_N(X_N,D_N) \leq  \limsup_{N \to \infty} \frac{1}{\beta N k(N)} \ln  I_N(X_N,D_N) \leq \mathcal{J}_{\xi}( \nu) + \eta \]
	
	In particular, we have the following lower bound,
	\begin{eqnarray*}
		\tilde{\Pp}[ d( \hat{\nu}_N, \nu) \leq \delta] &\geq & \tilde{\E} \Big[ \frac{ I_N(X_N,D_N)}{I_N(X_N,D_N)} \mathds{1}_{ \{d( \hat{\nu}_N, \nu) \leq \delta\} \cap E_{M'}^c} \Big] \\
		&\geq& \tilde{\Pp}[ \{d( \hat{\nu}_N, \nu) \leq \delta\} \cap E_{M'}^c ] \\
		& \geq& \exp( \beta k(N) N ( \mathcal{J}_{\xi}(\nu) - \xi( x^2/2)- \eta + o(1))) \Pp[  \{d( \hat{\nu}_N, \nu) \leq \delta\} \cap E_{M'}^c] \\
		& \geq &\exp( \beta k(N) N ( \mathcal{J}_{\xi}(\nu) - \xi( x^2/2) -\eta + o(1))) \Big( \Pp[  d( \hat{\nu}_N, \nu) \leq \delta ] - \Pp[ E_{M'}] \Big) \\
		& \geq& \exp( \beta k(N) N ( \mathcal{J}_{\xi}(\nu) - \xi( x^2/2)- \eta  + o(1))) \Big( \exp( -2 k(N) N \mathcal{I}( \nu)) - \Pp[ E_{M'} ] \Big) \\
		& \geq &\exp ( \beta N k(N) (\mathcal{J}_{\xi}(\nu) - \xi( x^2/2) - \mathcal{I}(\nu) + o(1)))
	\end{eqnarray*}
 Where to go from the first to the second line, we used  the fact that since $X_N$ is Gaussian:
	\[\frac{1}{\beta N k(N)} \ln \E[ I_N(X_N,D_N)] = \frac{1}{2k(N)} \sum_{i= - k(N) \atop i \neq 0}^{ k(N)} (\theta_i^N)^2 \approx \frac{1}{2} \xi(x^2) + o(1) \] and 
	where we used the large deviation principle for $\hat{\nu}_N$ under $\Pp$ given by Theorem \ref{LDPTheo}. 
	We conclude here using the fact that
	\[ \mathcal{I}_{\xi}(\nu) = \mathcal{I}(\nu) +\xi(x^2/2) - \mathcal{J}_{\xi}(\nu). \]
For the upper bound, if we denote $\tilde{\E}$ the expectation under $\tilde{\Pp}$, we have:
	\begin{eqnarray*}
\tilde{\Pp}[ d( \hat{\nu}_N, \nu) \leq \delta]&\leq& \tilde{\Pp}[ \{d( \hat{\nu}_N, \nu) \leq \delta\} \cap E_{M'}^c ] + \tilde{\Pp}[E_{M'}] \\
		 &\leq & \tilde{\E} \Big[ \frac{ I_N(X_N,D_N)}{I_N(X_N,D_N)} \mathds{1}_{ \{d( \hat{\nu}_N, \nu) \leq \delta\} \cap E_{M'}^c} \Big]  + \tilde{\Pp}[E_{M'}] \\
		& \leq& \exp( \beta k(N) N ( \mathcal{J}_{\xi}(\nu) - \xi( x^2/2) + \eta  + o(1))) \Pp[  \{d( \hat{\nu}_N, \nu) \leq \delta\} \cap E_{M'}^c] + \tilde{\Pp}[E_{M'}]\\
		& \leq &\exp( \beta k(N) N ( \mathcal{J}_{\xi}(\nu) - \xi( x^2/2) + \eta + o(1)))  \Pp[  d( \hat{\nu}_N, \nu) \leq \delta ]+ \tilde{\Pp}[E_{M'}] \\
		& \leq& \exp( \beta k(N) N ( \mathcal{J}_{\xi}(\nu) - \xi( x^2/2) + o_{\delta}(1)+ \eta + o(1))) \Big( \exp( - \beta k(N) N (\mathcal{I}( \nu)+ o_{\delta}(1)) \Big) + \tilde{\Pp}[E_{M'}] \\
		& \leq & \exp ( \beta N k(N) (\mathcal{J}_{\xi}(\nu) - \xi( x^2/2) - \mathcal{I}(\nu) + \eta + o(1) +o_{\delta}(1)))
	\end{eqnarray*}
	where $o_{\delta}(1)$ denotes a function of $\delta$ that tends to $0$ as $\delta$ tends to $0$.
	 Taking $\eta$ to $0$ then gives us the large deviation principle.
	
	Now we can generalize this large deviation principle to the case of $D_N$ with unbounded entries.
	\begin{theo}
		Let us assume that there is some probability measure $\xi$ such that: 
		\[ \lim_{N \to \infty} \frac{1}{2 k(N)} \sum_{i= -k(N),\dots,k(N) \atop i \neq 0} \delta_{\theta_i} = \xi. \]
		Furthermore, let us assume that $\xi(x^2) < \infty$. 
		Then, with the same notations as in Theorem~\ref{theo:extremeeigGaussian}, $(\hat{\nu}_N)_{N \in \N}$ satisfy a large deviation principle in speed $2 k(N)$  with good rate function $\beta \mathcal{I}_{\xi} /2$ defined by:
		\[ \mathcal{I}_{\xi}(\nu) = \begin{cases}
		\int_0^1 I_{ Q_{\xi}(t)}( Q_{\nu(t)}) d t & \text{if } \nu(] - \infty, -2]) = \nu([2, + \infty[) = \frac{1}{2} \\ 	+ \infty & \text{otherwise.}\end{cases} \]
	\end{theo}
	
	\begin{proof}
		We are going to approximate $\tilde{X}_N$ by $\tilde{X}_N^{(M)}:= U Y_N U ^* + U D_N^{(M)} U^*$ where \newline $D_N^{(M)} = \diag( \theta_1^{N,(M)},\dots,\theta_{k(N)}^{N,(M)}, \theta_{-1}^{N,(M)},\dots,\theta_{- k(N)}^{N,(M)}, 0, \dots,0)$ with $\theta_i^{N,(M)} = \theta_i^N \wedge M$ and $\theta_{-i}^{N,(M)} = \theta_{-i}^N \vee (-M)$. 
		We easily have that: 
		\[ \limsup_{M \to \infty} \limsup_{N \to \infty} \frac{1}{k(N)} \text{rank}( \tilde{X}_N ^{(M)} - \tilde{X}_N ) = 0. \]
		
		Let denote $\mathcal{F}_{BV}$ the following subspace of $\mathcal{C}_c( \R)$:
		\[ \mathcal{F}_{BV}:= \{ f \in \mathcal{ C }_c(\R): f \text{ is 1-Lipshitz and has total variation 1} \} \]
		$d_{BV}$ is defined as: 
		\[ d_{BV}( \mu, \mu') = \sup_{f \in \mathcal{F}_{BV}} \Big| \int f d \mu - \int f d\mu' \Big|. \]
		The distance $d_{BV}$ metrizes the weak topology. We will use the following lemma: 
		
		\begin{lemma}
		  For a universal constant $C >0$, if $A, A' \in \mathcal{H}_{N}^{\beta}$ and if $\nu$ and $\nu'$ are the distribution of the $2 k(N)$ extremal eigenvalues of respectively $A$ and $A'$, we have $d_{BV}( \nu, \nu') \leq \frac{C}{2 k(N)}\text{ rank}( A - A')$.
	\end{lemma}

	\begin{proof}
		For any $f \in \mathcal{F}_{BV}$, there is $g$ and $h$ both $1$-Lipshitz, increasing and uniformly bounded by $1$ such that $f = g-h$. 
		If $A -A'$ is of rank $1$, we have for $k \in [ 2, N-1]$
		\[ \lambda_{k-1}(A') \geq  \lambda_{k}(A) \geq \lambda_{k+1}(A') \]
		so we have 
		\[\sum_{i=1}^{k(N) -1} g(\lambda_i(A')) \geq \sum_{ i=2}^{k(N)} g(\lambda_i(A)) \]
		and so
		\[\sum_{i=1}^{k(N) } g(\lambda_i(A')) \geq \sum_{ i=1}^{k(N)} g(\lambda_i(A)) - 2. \]
		The same steps also imply that:
		\[\sum_{i=1}^{k(N) } g(\lambda_{N -i+1}(A')) \geq \sum_{ i=1}^{k(N)} g(\lambda_{N -i+1}(A)) - 2. \]
		
		Doing the same thing for $h$ and substracting, we have 
		\[ \nu(f) \geq \nu'(f) - \frac{4}{ k(N)} \]
		and symetrically: 
		\[ \nu'(f) \geq \nu(f) - \frac{4}{ k(N)} \]
		so
		\[ d_{BV}( \nu, \nu') \leq \frac{ 4}{ k(N)} .\]
		An obvious recursion on the rank finishes the proof. 
	\end{proof}
		Thus if we denote for every $M >0$, $\hat{\nu}_N^{(M)}$ the extremal eigenvalue distribution of $\tilde{X}_N^{(M)}$, then for $d_{BV}$ the $\hat{\nu}^{(M)}$ are exponential approximations of $\nu^{(M)}$. Furthermore, by construction, it is easy to see that for every $M>0$ the distribution of the $\theta_i^{N,(M)}$ converges toward $\xi^{(M)}$ where $\xi^{(M)}$ is the push forward of $\xi$ by the function $x \mapsto (x \wedge M) \vee -M$. So the $( \hat{\nu}_N^{(M)})_{N \in \N}$ follow a large deviation principle with rate function $\mathcal{I}^{(M)}:= \mathcal{I}_{\xi^{(M)}}$. Therefore using \cite[Theorem 4.2.16]{DZLD}, $(\hat{\nu}_N)_{N \in \N}$ will satisfy a weak large deviation principle with rate function: 
		\[ \mathcal{I}'(\nu) = \limsup_{\delta \to 0} \limsup_{M \to + \infty} \inf_{\nu',d_{BV}(\nu, \nu') \leq \delta} \mathcal{I}^{(M)}( \nu'). \]
		It remains to show the three following things to conclude that we have a (strong) large deviation principle with the wanted rate function:
		\begin{enumerate}
			\item Indeed, we have for every $\nu \in \mathcal{P}(\R)$ such that $\nu(]- \infty,-2]) = \nu([2, + \infty[) = 1/2$, 
			\[ \mathcal{I}'(\nu) = \mathcal{I}_{\xi}(\nu). \]
			\item $\mathcal{I}_{\xi}$ is a good rate function.
			\item For every close set $F$ in the weak topology
			\[ \inf_{\nu \in F} \mathcal{I}(\nu) \leq \limsup_{M \to \infty} \inf_{\nu \in F} \mathcal{I}^{(M)}(\nu). \] 
		\end{enumerate}
		We are first going to prove that $\mathcal{I}_{\xi}$ gives a control over the second moment of $\nu$. For this, we are going to use the following lemma: 
		\begin{lemma}
			For every $\theta \geq 0$ and $x \geq 2$, 
			\[ I_{\theta}(x) \geq \frac{1}{2} \Big( x - \Big( \theta \vee 1	+ \frac{1}{ \theta \vee 1} \Big) \Big)^2 . \]
		\end{lemma}
		This is simply due to the fact that $I_{\theta}( \theta \vee 1	+ \frac{1}{ \theta \vee 1}) =0$ and $I''_{\theta}(x) \geq 1$. From this lemma, one deduces that if $\xi \in \mathcal{P}(\R)$ and if we denote $\tilde{Q}_{\xi}$ the function defined by:
		\[ \tilde{Q}_{\xi}(x) = \begin{cases} (Q_{\xi}(x) \vee 1)  + \frac{1}{ Q_{\xi}(x) \vee 1} \text{ if } x \geq 1/2 \\
			(Q_{\xi}(x) \wedge 1)  + \frac{1}{ Q_{\xi}(x) \wedge 1} \text{ if } x < 1/2	\end{cases} \]
		then with $\nu \in  \mathcal{P}(\R)$ such that $\nu(] - \infty, -2]) = \nu([2, + \infty[) = 1/2$, one has 
		\[ \mathcal{I}_{\xi}(\nu) \geq \frac{1}{2} || Q_{\nu} - \tilde{Q}_{\xi} ||_2^2 .\]
		And therefore, using using the triangle inequality for $|| \cdot ||_2$
		\[ || Q_{\nu}||_2 \leq \sqrt{  2 \mathcal{I}_{\xi}(\nu)} + || \tilde{Q}_{\xi}||_2 .\]
		Moreover $\tilde{Q}_{\xi}^2 \leq ( 2 + Q_{\xi})^2$ and $\xi(x^2) = \int_0^1 Q_{\xi}(t)^2 dt$ so
		\[ || \tilde{Q}_{\xi}||_2 \leq \sqrt{ 4 + 4 \sqrt{ \xi(x^2)} + \xi(x^2)} ,\]
		and so if $\mathcal{I}_{\xi}(\nu)$ and $\xi(x^2)$ are finite, so is $\nu(x^2)$ and 
		\begin{equation} \label{eq2ndmoment}
			\nu(x^2) \leq  \left( \sqrt{2 \mathcal{I}_{\xi}( \nu)} + \sqrt{ 4 + 4 \sqrt{ \xi(x^2)} + \xi(x^2)} \right)^2, \end{equation}
		Note that it also prove the item (2) since our assumption is that $\xi(x^2) < + \infty$ and the sets $\{ \nu \in \mathcal{P}(\R), \nu(x^2) \leq T \}$ are compacts for the weak topology. 
		Conversely it is almost direct from the definition of $I_{\theta}$ that since $I(x) \leq x^2/2$,
		\[ I_{\theta}(x) \leq \frac{x^2 + \theta^2}{2} \]
		implying that if $\nu(x^2)$ and $\xi(x^2)$ are finite, so is $\mathcal{I}_{\xi}(\nu)$ and 
		\[ \mathcal{I}_{\xi}(\nu) \leq \frac{1}{2}( \nu(x^2) + \xi( x^2)). \]
		Let us prove point (1). First the definition implies that 
		\[ \mathcal{I}'(\nu) \leq \liminf_{M \to \infty} \mathcal{I}^{(M)}(\nu). \]
		However, one can notice that $Q^2_{\xi^{(M)}} \leq Q_{\xi}^2$ and so for every $t \in [0,1]$ 
		\[ I_{Q_{\xi^{(M)}}(t)}( Q_{\nu}(t)) \leq \frac{1}{2} ( Q_{\nu}(t)^2 + Q_{\xi}(t)^2). \]
		Therefore, for every $\nu$ such that $\nu(x^2) < + \infty$ the $ t \mapsto I_{Q_{\xi^{(M)}}(t)}( Q_{\nu}(t))$ are equi-integrable. Since these functions converge almost everywhere to $ t \mapsto I_{Q_{\xi}(t)}( Q_{\nu}(t))$, we have that:
		\[ \liminf_{M \to \infty} \mathcal{I}^{(M)}(\nu) = \liminf_{M \to \infty} \int_0^1 I_{Q_{\xi^{(M)}}(t)}( Q_{\nu}(t)) dt = \int_0^1 I_{Q_{\xi}(t)}( Q_{\nu}(t)) dt = \mathcal{I}_{\xi}(\nu) \]
		and therefore 
		\[ \mathcal{I}'(\nu) \leq \mathcal{I}_{\xi}(\nu) \]
		To prove the converse, let us reason by contradiction and assume that 
		\[ \mathcal{I}'(\nu) < \mathcal{I}_{\xi}(\nu) \]
		for some $\nu$. Then using the definition of 
		$\mathcal{I}'$, one can find $\epsilon >0$, a sequence $(\delta_N)$ converging to $0$, a sequence $(M_N)$ diverging to $+ \infty$ a sequence of measure $(\nu_N)$ such that $d_{BV}(\nu, \nu_N) \leq \delta_N$ and so that
		\[ \liminf_{N \to \infty} \mathcal{I}^{(M_N)}( \nu_N) \leq \mathcal{I}_{\xi}(\nu) - \epsilon. \]
		However, since $\nu_N$ converges toward $\nu$ and $\xi^{(M_N)}$ converges toward $\xi$, we have the convergence almost everywhere of $Q_{\nu_N}$, $Q_{\xi^{(M_N)}}$ toward respectively $Q_{\nu}$ and $Q_{\xi}$. Using then the continuity of $(x, \theta) \mapsto I_{\theta}(x)$ and Fatou's lemma as well as the definition of $\mathcal{I}^{(M_N)}$, we get 
		\[ \liminf_{N \to \infty} \mathcal{I}^{(M_N)}( \nu_N) \geq \mathcal{I}_{\xi}(\nu) \]
		which yields the desired contradiction.
		It only remains to show the third point. First we can use equation \eqref{eq2ndmoment} and the fact that $\xi^{(M)}(x^2) \leq \xi(x^2) $ to prove that for every $M > 0$ 
		\[  \mathcal{I}^{(M)}(\nu) \geq \frac{1}{2}\Big( \nu(x^2)^{1/2} - C \Big)^2 \]
		where $C = \sqrt{4 + A \sqrt{\xi(x^2)} + \xi(x^2)}$. Therefore, to prove the third point, we can restrict both the $\inf$ on $F$ to inf on the compact $ F' = F  \cap \{ \nu \in \mathcal{P}(\R): \nu(x^2) \leq K \}$ by choosing $K$ such that $(K^{1/2} - C)^2/2 \geq \inf_{x \in F} \mathcal{I}_{\xi}(x)+1 $. Then, since we have infimums of lower semi-continuous functions, we can replace  $\inf_{\nu \in F'} \mathcal{I}^{(M)}(\nu)$ by $\mathcal{I}( \nu_M)$ with $\nu_M \in F'$. Then the desired inequality becomes 
		\[ \inf_{\nu \in F'} \mathcal{I}_{\xi}(\nu) \leq \limsup_{M \to \infty} \mathcal{I}^{(M)}(\nu_M) . \]
		Now using the compactness of $F'$, there is a sequence $M_N$ increasing to $+ \infty$ and $\nu_{M_N}$ converging to some $\nu_0 \in F$ and such that $\lim_{N \to \infty} \mathcal{I}^{(M_N)}(\nu_{M_N}) =\limsup_{M \to \infty} \mathcal{I}^{(M)}(\nu_M)$. But then, using Fatou's lemma again, we have that
		\[\mathcal{I}_{\xi}(\nu_0) \leq \liminf_{N \to + \infty} I^{(M_N)}(\nu_{M_N}) \]
		and so 
		\[ \inf_{\nu \in F'} \mathcal{I}(\nu) \leq \limsup_{N \to \infty} \mathcal{I}^{(M_N)}(\nu_{M_N}) =  \limsup_{M \to \infty} \mathcal{I}^{(M)}(\nu_M)  \]
		which proves the result. 
	\end{proof}

	\section{Applications to Spin Glasses}\label{sec:spinglass}
	
	In this section, we explain the application of the $o(N)$ spherical integrals to further understand various models appearing in spin glasses. Theorem~\ref{theo:positivetheo} gives an explicit closed form of the free energy of $2$-spin spherical spin glasses with $k(N)$ dimensional spins. This growing rank extension allows us to the take the dimensions of the vector spins to $\infty$ to study the concentration as coupled copies tend to $\infty$.
	
	\subsection{The Spherical SK Model}\label{sec:sphericalSK}
	
	We start by introducing the classical spherical SK model. The Hamiltonian in this model is given by
	\begin{equation}\label{eq:origSKHam}
		\tilde H_N(\bs) = \frac{1}{\sqrt{2N}} \sum_{i,j = 1}^N g_{ij} \sigma_i \sigma_j
	\end{equation}
	where $g_{ij}$ are iid real valued standard Gaussians and the spin variables $\bs \in \R^N$ lie on the sphere of radius $\sqrt{N}$. Given an inverse temperature parameter $\theta > 0$, the free energy associated with this Hamiltonian is
	\[
	\tilde F_N(\theta) = \frac{1}{N} \ln \int_{\| \bs\| = \sqrt{N}} e^{ \theta \tilde H_N(\bs)} \, d\bs
	\]
	where $d \bs$ is the uniform measure on the sphere with radius $\sqrt{N}$. The limit of the free energy is given by the replica symmetric restriction of the Crisanti--Sommers formula \cite{crisanti1992sphericalp,TSPHERE,CASS}
	\begin{equation}\label{eq:cs1dim}
		\lim_{N \to \infty} \E \tilde F_N(\theta) = 	\inf_{q \in [0,1)} \frac{1}{2} \bigg( \frac{\theta^2}{2}(1-q^2)+ \frac{q}{1-q} + \ln(1 - q) \bigg).
	\end{equation}
	Another variational formula for this model using the TAP approach was proved in \cite{BeliusTAP}. The spherical integrals can be used to prove an alternative derivation of the limit of the free energy.

	We first notice that the Hamiltonian \eqref{eq:origSKHam} is equivalent in distribution to a quadratic form of a GOE matrix,
	\[
	H_N(e) = \frac{N}{2} \langle e, G_N e \rangle
	\]
	where $G_N$ is a GOE matrix and $e$ is a vector on the unit sphere.  Indeed, it follows that both $H_N(\bs)$ and $\tilde H_N(e)$ are Gaussian processes with mean $0$ with covariances
	\[
	\E H_N(e^1) H_N(e^2) = \frac{N}{2} (e^1 \cdot e^2)^2 \quad\text{and} \quad \E \tilde H_N(\bs^1) \tilde H_N(\bs^2) = \frac{N}{2} \Big( \frac{\bs^1}{\sqrt{N}} \cdot \frac{\bs^2}{\sqrt{N}}\Big)^2 .
	\]
	Since $\frac{\bs}{\sqrt{N}}$ is on the unit sphere in $\R^N$, both Gaussian processes have the same mean and covariance structure so it suffices to study the Hamiltonian $H_N$ defined on unit vectors.

	A fundamental quantity in the study of spin glasses is the free energy, which in our setting is precisely a spherical integral
	\[
	F_N(\theta) = \frac{1}{N} \E_{G_N} \ln \int_{S_1} e^{ \theta H_N(e)} \, de =  \frac{1}{N} \E_{G_N} \ln  \int \Big[ \exp\Big( \frac{1}{2} N \tr(U^* G_N U D_N) \Big) \Big] \, dU
	\]
	where $D_N = \mathrm{\diag}(\theta,0,\dots,0)$ and the outer expected value is over the randomness in the GOE matrix. Because $D_N$ is rank $1$, in this setting the formula only depends on the first column of $U$ which is uniform on the sphere. 
	
	This is not quite of the form of the spherical integrals $I_N$ we defined earlier because the matrix $G_N$ is random, but we can use the almost sure convergence of its eigenvalues to replace $G_N$ with a deterministic matrix, as we will explain below. If $\lambda_1, \dots, \lambda_N$ denote the random eigenvalues of $G_N$ then the empirical measure of the eigenvalue converges to the semicircle law
	\[
	\frac{1}{N} \sum_{i = 1}^N \delta_{\lambda_i} \to d\sigma(x) = \frac{1}{2\pi} \1_{[-2,2]}(x) \sqrt{4 - x^2} dx. 
	\]
	
	Using the convergence of the empirical distribution, we can diagonalize the random matrix $G_N = U \tilde A_N U^*$ where $\tilde A_N = \diag(\lambda_1, \dots, \lambda_N)$ are the random normalized eigenvalues of $G_N$. We denote the typical values of $\tilde A_N$  by
	\begin{equation}\label{eq:nonrandcoeff}
	a_i = \inf \bigg\{ u ~\bigg\vert~ \int_{-2}^{u}  d \sigma (x) = -2 + (i - 1) \frac{4}{N}  \bigg\}.
	\end{equation}
	If $A_N = \diag(a_1, \dots, a_N)$, then it is clear that the spectral distribution of $A_N$ also converges to $d\sigma(x)$. Because the limit of the spherical integrals only depend on the eigenvalues by rotational symmetry, a simple modification of the argument in Lemma~\ref{lem:deterministicdisorder} implies that
	\begin{equation}\label{eq:FEapprox}
		\lim_{N \to \infty} | \E F_N(\theta) - I_N(A_N, D_N) | = 0.
	\end{equation}
	
	The spherical integral limit Theorem~\ref{theo:positivetheo} when $k(N) = 1$ can be computed explicitly to give the following closed form of the limit of the free energy, which was already proved in \cite{TSPHERE}. 
	\begin{prop}[Spherical SK Model] \label{prop:sphericalSK}
		The limit of the free energy in a spherical $2$-spin model is given by
		\begin{equation}\label{eq:crisantieigenvalue}
				\lim_{N\to \infty} \E F_N(\theta) = \begin{cases}
			\frac{\theta^2}{4} &\quad \theta < 1
			\\\theta - \frac{\ln \theta}{2} - \frac{3}{4}  &\quad \theta \geq 1.
		\end{cases}
		\end{equation}
	\end{prop}
	\begin{proof}
		By our observation \eqref{eq:FEapprox} and  Theorem~\ref{theo:positivetheo}, we have
		\[
		\lim_{N\to \infty} \E F_N(\theta) = \frac{1}{2} J(\theta, 2, \sigma)
		\]
		because $\beta = 1$ in the real case and $2$ is the largest point in the support of $\sigma$. Since $\sigma$ is the semicircle distribution, we can explicitly compute the Stieltjes transform and recover a closed form of the rate function $J$ defined in Definition~\ref{def:J}. 
		
		We have
		\begin{equation}\label{eq:stieljessc}
			G_{ \sigma}(z) = \int_{-2 }^{2} \frac{{\sigma}(x)}{z - x} \, dx = \frac{z - \sqrt{z^2 - 4}}{2} \quad \text{ for } z \geq \sqrt{2},
		\end{equation}
		and integrating this gives the logarithmic potential
		\begin{equation}\label{eq:logpotential}
			h_{\sigma}(z):= \int_{-2}^{2} \ln |z - x| \, d {\sigma}(x) = \frac{z^2}{4} - \frac{z \sqrt{z^2 - 4}}{4} + \ln \frac{z + \sqrt{z^2 - 4}}{2}  - \frac{1}{2} \quad \text{ for } z \geq 2.
		\end{equation}
		Since $G_{\sigma}(2) = 1$, we have
		\[
		v = \begin{cases}
			2 & \text{if }1 \leq \theta,\\
			G^{-1}_{\sigma}(\theta) & \text{if } 1 > \theta.
		\end{cases}
		\]
		We now evaluate $J$ on these two regions. On the high temperature region $\theta < 1$, 
		\begin{equation}\label{eq:invstieltjes}
			G^{-1}_{\sigma}(\theta) = \theta + \frac{1}{\theta}	
		\end{equation}
		so
		\begin{align}
			\frac{1}{2} J(\theta,2,{\sigma})=
			\frac{1}{2} \bigg( 2\theta + \bigg( \theta + \frac{1}{\theta}- 2\bigg) \theta - \ln \theta - h_{\sigma}\bigg( \theta + \frac{1}{\theta} \bigg)  -1 \bigg) = \frac{\theta^2}{4} \label{eq:sphericalspinglasslimit}.
		\end{align}
		On the low temperature region, $\theta > 1$, we see that
		\begin{align}
			\frac{1}{2} J(\theta,2,{ \sigma})
			=  \frac{1}{2} \bigg( 2 \theta  + (2 -2) G_\sigma(2) - \ln \theta -h_\sigma(2) -1 \bigg) \notag
			= \theta - \frac{\ln \theta}{2} - \frac{3}{4} \label{eq:sphericalspinglasslimit2}.
		\end{align}
	\end{proof}
	
	\subsection{The Vector Spin Spherical SK Model}
	We can extend the results for the spherical SK model to study a coupled system of $k(N)$ spherical spin glasses. In the case when $k(N) = k$, this model is called the $k$ dimensional vector spin spherical SK model \cite{PTSPHERE,PVS,kosphere}. These models commonly show up when studying the large deviations for the overlap matrices \cite[Theorem~1.13]{GBAspectralgap}  or computing the probability of sampling $k(N)$ configurations from a Gibbs measure \cite{subaglandcapes}. 
	
	Let $k(N) = o(N)$ and consider a set of $k(N)$ configurations $\Sigma = (\bs^1,\dots, \bs^{k(N) }) \in \R^{k(N) \times N}$. Given a positive definite matrix $Q \in \R^{k(N) \times k(N)}$ with diagonal entries $1$, a central quantity is the constrained free energy defined in terms of the Hamiltonian defined in \eqref{eq:origSKHam},
	\begin{equation}\label{eq:vectorspinFE}
	\tilde F^\epsilon_N(Q) = \frac{1}{Nk(N)} \E \ln \int \1( ||| N^{-1} \Sigma \Sigma^\top - Q ||| \leq \epsilon ) e^{\sum_{\ell = 1}^k(N) \theta_\ell H_N(\bs^\ell)} \, d \bs^1 \cdots d \bs^{k(N)}.
	\end{equation}
	The replica symmetric form of the Crisanti--Sommers formula \cite[Theorem~1]{kocs} in the case when $k(N) = k$ provides an upper bound of the free energy 
	\begin{equation}\label{eq:PARFEnoH}
		\lim_{\epsilon \to 0}  \lim_{N \to \infty} \E \tilde F^\epsilon_N(Q) \leq \inf_{M} \frac{1}{2k} \bigg( \frac{1}{2} \theta^\top (Q^{\odot 2} - M^{\odot 2} ) \theta + \ln | Q - M| + \tr( (Q - M)^{-1} M  ) \bigg)
	\end{equation}
	 where the supremum is over positive semidefinite matrices such that $0 \leq M \leq Q$. Similarly, to the one dimensional case, the asymptotics of the spherical integrals can be applied in this setting to derive closed forms of the limit.
	 
	 \begin{rem}
	 	We use the operator norm in the definition of \eqref{eq:vectorspinFE} instead of the infinity norm on matrices that appears in previous works \cite{PTSPHERE,PVS,kosphere}, because the choice of norm is essential if the rank $k(N) \to \infty$. Of course, in the case that $k(N) = k$ is independent of $N$, norm equivalence in finite dimensions implies that \eqref{eq:vectorspinFE} is equivalent to the free energies appearing in the previous work.
	 \end{rem}
	
	If we take $Q = I_{k(N)} \in \R^{k(N) \times k(N)}$ and restrict the inner products of the configurations to be approximately orthogonal, we are essentially integrating uniformly over unitary matrices in the limit as $\epsilon \to 0$. In this setting, our results follow immediately from the finite rank formulas. The main difficulty is showing that the restriction to an approximate identity is equivalent to integrating over the Haar measure on orthogonal matrices.
	
	\begin{lemma}\label{lem:apporxidentity}  Let $D_N$ satisfy Assumption~\ref{assum:A1} and $k(N) = o(N)$. We have
	\[
	\lim_{\epsilon \to 0} \lim_{N \to +\infty} \E \tilde F^\epsilon_N(I)  = \lim_{N \to \infty} I_N(A_N, D_N),
	\]
	where $D_N = \diag(\theta_1, \dots, \theta_{k(N)})$, $A_N = \diag(a_1, \dots, a_N)$ are the non-random coefficients defined in \eqref{eq:nonrandcoeff}.
	\end{lemma}

	\begin{proof}
		After normalizing and replacing the Gaussian disorder matrix with the deterministic diagonal matrix $A_N$ in the steps leading to \eqref{eq:FEapprox} as explained in the last section, it follows that
		\[
		\lim_{\epsilon \to 0} \lim_{N \to +\infty} \E \tilde F^\epsilon_N(I) = \lim_{\epsilon \to 0} \lim_{N \to +\infty}  F^\epsilon_N(I)
		\]
		 and
		\[
		F^\epsilon_N(I) = \frac{1}{Nk} \ln \int  \1( ||| E_k E_k^\top - I ||| \leq \epsilon ) e^{ N \tr( E_k^\top  A_N E_k D_N ) } \, de^1 \dots d e^{k(N)},
		\]
		where $E_k = (e^1, \dots, e^{k(N)}) \in \R^{k(N) \times N}$ and  $e^1,\dots, e^{k(N)}$ are independent and uniform over the unit sphere. We now have to show that we can express the term on the right as an integral over the Haar measure.  
		
		We consider an extended system of $N$ configurations, and define $E = (e^1, \dots, e^N) \in \R^{N \times N}$ where $e^1, \dots, e^N$ are sampled independently and uniformly on the unit sphere without changing the limit of the free energy. Since $D_N$ is a diagonal matrix of rank $k(N)$, we have
		\[
		\tr( E_k^\top  A_N E_k D_N ) = \tr( E^\top  A_N E D_N ),
		\]
		so we can study this enlarged system without changing the limit of the free energy.  The matrix $T = (E E^\top)^{-\frac{1}{2}} \in \R^{N \times N}$ exists almost surely  and the matrix
		\[
		U:=  T E = (E E^\top)^{-\frac{1}{2}} E \in \R^{N \times N}
		\]
		satisfies $U U^\top = I$. Furthermore, the rotational invariance of the product measure on sphere implies that for every orthogonal matrix $M$, $E \stackrel{d}{=} E M$, so
		\[
		U \stackrel{d}{=} (EM (EM)^\top)^{-\frac{1}{2}} EM = (E E^\top)^{-\frac{1}{2}} EM = U M
		\]
		and therefore $U$ is also rotationally invariant and hence its law under $de^1 \cdots de^N$ is the unique Haar measure. 
		
		We next observe that on the set $\{|||E E^\top - I||| < \epsilon \}$ all eigenvalues of the matrix $E E^\top$ are in an epsilon neighbourhood of $1$, so all eigenvalues of $T$ lie in the interval $] \frac{1}{\sqrt{1 + \epsilon}}, \frac{1}{\sqrt{1 - \epsilon}} [$. Therefore, the von Neumann trace inequality implies that
		\begin{align}
		|N \tr( E^\top A_N E D_N ) - N \tr( U^\top A_N U D_N )| &= |N \tr( U^\top A_N U D_N ) - N \tr( T^{-1} U^\top A_N U T^{-1} D_N  ) )| \notag
		\\&= N M k(N) |||  U^\top A_N U - T^{-1} U^\top A_N U T^{-1} ||| \label{eq:vonneumann}
		\end{align}
		because the matrix $D_N$ is of rank at most $k(N)$ and $|||D_N||| \leq M$. Next, notice that on the set $\{ ||| E E^\top - I ||| < \epsilon \} = \{ ||| T^{-2} - I ||| < \epsilon \} $, all eigenvalues of $T^{-1}$ lie in the interval $]\sqrt{1 + \epsilon},\sqrt{1 - \epsilon} [$, so $||| T^{-1} - I ||| = O(\epsilon)$. The triangle inequality and the fact the operator norm is submultiplicative implies 
		\[
		|||  U^\top A_N U - T^{-1} U^\top A_N U T^{-1} ||| \leq 	|||  U^\top A_N U ||| \cdot ||| T^{-1} - I ||| + ||| T^{-1} - I||| \cdot ||| U^\top A_N U T^{-1} ||| = O(\epsilon)
		\]
		so
		\[
		|N \tr( E^\top A_N E D_N ) - N \tr( U^\top A_N U D_N )|  = O( Nk(N) \epsilon) .
		\]
		We have shown that
		\begin{equation}\label{eq:intermediatehaar}
		  \E \tilde F^\epsilon_N(I) = \frac{1}{Nk(N)} \ln \int  \1( ||| E E^\top - I ||| \leq \epsilon ) e^{ N \tr( U^\top  A_N U D_N ) } \, de^1 \dots d e^N + O(\epsilon). 
		\end{equation}
		
		To decouple the constraint on the approximate indicator, we can add and subtract a normalizing constant to conclude that \eqref{eq:intermediatehaar} is equal to
		\begin{equation}\label{eq:decomposition}
		 \frac{1}{Nk(N)} \ln \E_{I_\epsilon} e^{ N \tr(  U^\top  A_N  U D_N ) }  + \frac{1}{Nk(N)} \ln \Pp( |||  E E^\top - I ||| \leq \epsilon )  + O(\epsilon)
		\end{equation}
		where $\E_{I_\epsilon}$ is the average with respect to the restriction of the probability measure $ de^1 \dots d e^N$ to the set $$I_\epsilon = \{ ||| E E^\top - I ||| \leq \epsilon \}.$$ 	We will show later below the second term of \eqref{eq:decomposition} vanishes. Assuming this, notice that $U$ is Haar distributed under $de^1 \cdots de^n$ by construction, so we can conclude that
		\[
		\frac{1}{Nk(N)} \ln \E_{I_\epsilon} e^{ N \tr(  U^\top  A_N  U D_N ) }  + \frac{1}{Nk(N)} \ln \Pp( |||  E E^\top - I ||| \leq \epsilon )  + O(\epsilon) = I_N(A_N, D_N) + o_N(1) + O(\epsilon)
		\]
		which finishes the proof.
		
		We now prove that the second term of \eqref{eq:decomposition} vanishes in the limit. For each $i \leq k(N)$, let $x^i$ be a standard Gaussian vector on $\R^N$.  Notice that the conditional law of $\frac{x^i}{|x^i|}$ conditionally on the event $\{|\frac{|x^i|}{\sqrt{N}} - 1| < \epsilon\}$ is the same as the law of the uniform vector on the unit sphere $e_i$. Let $X = (x^1, \dots, x^{k(N)})$, $\tilde X = (\frac{x^1}{|x^1|}, \dots, \frac{x^{k(N)}}{|x^{k(N)}|})$  and define $W = \frac{1}{N} X X^\top$. It follows that 
		\begin{align*}
		&\frac{1}{Nk(N)} \ln \Pp( |||  E E^\top - I ||| \leq \epsilon ) 
		\\&= \frac{1}{Nk(N)} \ln \Pp\Big( |||  \tilde X \tilde X^\top - I ||| \leq \epsilon , \sup_{i \leq k(N)} \Big|\frac{|x^i|}{\sqrt{N}} - 1 \Big| < \epsilon\Big) -  \frac{1}{Nk(N)} \ln \Pp\Big( \sup_{i \leq k(N)} \Big|\frac{|x^i|}{\sqrt{N}} - 1\Big| < \epsilon\Big).
		\end{align*}
		The term $ \frac{1}{Nk(N)} \ln \Pp\Big(  \sup_{i \leq k(N)} |\frac{|x^i|}{\sqrt{N}} - 1| < \epsilon\Big) = \frac{1}{N} \ln \Pp\Big(  |\frac{|x^1|}{\sqrt{N}} - 1| < \epsilon\Big) \to 0$ by the law of large numbers. On the first event, there exists absolute constants $c_1$ and $c_2$ such that
		\[
		\{ ||| W - I ||| \leq c_1 \epsilon \} \leq \Big\{ |||  \tilde X \tilde X^\top - I ||| \leq \epsilon , \Big|\frac{|x^i|}{\sqrt{N}} - 1\Big| < \epsilon\Big\} \leq \Big\{ |||  W - I ||| \leq c_2 \epsilon \Big\}
		\]
		The $W$ can be seen here as degenerate versions of Wishart matrices. We have the following lemma
			\begin{lemma}\label{lem:degenerateWishart}
				$|||W - I|||$ converges to 0 in probability.
				\end{lemma}
			\begin{proof}
				There are several ways here one can tackle this problem. Here we use a $\epsilon$-net argument. Let $u \in \R^{ k(N)}$ be a unit vector then $\hat{X}^T u$ is a Gaussian vector of covariance matrix $\frac{1}{N} I_N$ so $\langle u, W u \rangle= ||\hat{X}^T u||_2^2$ is a Gamma random variable of shape parameter $N/2$ and scale parameter $2/N$. For such a random variable, it is easy to see via a Laplace method that 
				\[ \Pp[|  \langle u, W u \rangle - 1 | \geq \epsilon] \leq e^{ -N c(\epsilon)} \]
				for some $c(\epsilon) >0$. Then let $\mathcal{N}_{k(N)}(\epsilon)$ be an $\epsilon$-net of $\mathbb{S}^{k(N)-1}$ of cardinal at most $(3/\epsilon)^{k(N)}$. Using that $k(N) =o(N)$, by a simple union bound we have that 
			\[ \Pp[ \forall u \in \mathcal{N}_{k(N)}(\epsilon), |  \langle u, W u \rangle - 1 | \leq \epsilon] \leq e^{ -N (c(\epsilon) +o(1))} \]
			Using this property, and since $W$ is a positive matrix, one easily deduces that with probability going to $1$, $|||W|||$ is bounded and then that for any $\epsilon >0$ with probability going to $1$,  $ \sup_{u \in \mathbb{S}^N} 	|  \langle u, W u \rangle - 1 | \leq \epsilon$. That easily implies the lemma .
			\end{proof}	
	By Lemma~\ref{lem:degenerateWishart}, it follows that for any $c > 0$,
		\[
		\frac{1}{Nk(N)} \ln \Pp( ||| W - I ||| \leq c \epsilon ) = 	\frac{1}{Nk(N)} \ln \Pp(\lambda_{min}(W) \geq 1 -  c \epsilon , \lambda_{max}(W) \leq 1 +  c \epsilon) \to 0
		\]
	Therefore, the second term in \eqref{eq:decomposition} vanishes as required.
	\end{proof}
	
	The limit in \eqref{lem:apporxidentity} can be explicitly computed using Theorem~\ref{theo:positivetheo} and the computations in Proposition~\ref{prop:sphericalSK} to see
	\[
	\lim_{\epsilon \to 0} \lim_{N \to +\infty} \E \tilde F^\epsilon_N(I) = \frac{1}{k(N)} \sum_{i = 1}^n J(\theta_i, 2 , \sigma) = \frac{1}{k(N)} \sum_{i = 1}^{k(N)} f(\theta_i).
	\]
	
	The challenge is to now is to extend this result to the case when the vectors $e_1, \dots, e_k$ are no longer orthogonal, but constrained to a neighbourhood of $Q$. To have a well defined limit in the growing rank case we need some assumptions on the sequences of constraint matrices $(Q_N)_{N \in \N}$.
	\begin{assum} \label{assum:AQ}
		We assume that $(Q_N)_{N \in \N}$ and diagonal matrices $(D_N)_{N \in \N}$ are two sequences of real valued $k(N) \times k(N)$ matrices such that: 
		\begin{enumerate}
			\item $\forall N \in \N$, $Q_N$ is positive definite and $Q_N$ is $1$ on the diagonals. Furthermore, we assume that there exists a $A > 0$ such that $ |||Q_N^{-1} |||  < A $ for all $N$.
			\item $\forall N \in \N$, $D_N$ is positive definite and diagonal.
			\item There exists $K > 0$ such that for all $N \in \N$, $|||D_N||| \leq K$. 
		\end{enumerate}
			
		To state the limit, we also assume that the eigenvalue distribution of $D_N^{\frac{1}{2}} Q_N D_N^{\frac{1}{2}}$, $\tilde \mu_{k(N)} = \frac{1}{k(N)} \sum_{i=1}^{k(N)} \delta_{\tilde \theta_i}$, where $\tilde \theta_i$ are the eigenvalues of $D_N^{\frac{1}{2}} Q_N D_N^{\frac{1}{2}}$, converges weakly toward a compactly supported measure $\tilde \mu$. Likewise, we also assume that the eigenvalue distribution of $Q_N$ converges to a compactly supported measure $\mu$.
	\end{assum}		

	\begin{rem}
	If $k(N) = k$ is independent of $N$, then any fixed positive definite matrix $k \times k$ matrices $D$ and $Q$ satisfies Assumption~\ref{assum:AQ}. This is the vector spin case. A more interesting case is understanding the replica matrix generated from $k(N)$ samples from the Gibbs measures at constant temperatures. Given $q \in (-1,1)$ the replica symmetric matrix $Q_N = I_{k(N)} + q1_{k(N)} - qI_{k(N)}$ which is $1$ on the diagonal and $q$ on the off diagonals satisfies Assumption~\ref{assum:AQ} for all $D_N$ with constants diagonal entries $\theta$.
	\end{rem}
	
	The $k(N)$ dimensional spherical integral formulas do not immediately apply in this setting, but we can reduce this to the orthogonal case by a change of variables and a modification of the temperature matrix $D_N$. When applied to $k(N) = k$, this gives us an alternative proof that the vector spin Crisanti--Sommers formula for the spherical SK model derived in \cite[Theorem~3]{PTSPHERE} is sharp without relying on the standard tools of spin glasses.	
	\begin{prop}[Vector Spin Spherical SK Model]\label{prop:vectorspinspherical}
		For $k(N) = o(N)$, the limit of the free energy in the vector spin spherical $2$-spin model is given by
		\[
		\lim_{\epsilon \to 0} \lim_{N\to \infty} \bigg| \E \tilde F_N^{\epsilon} (Q_N) - \frac{1}{k(N)} \sum_{i = 1}^{k(N)} f(\tilde \theta_i) + \frac{1}{2k(N)} \ln \det (Q_N) \bigg| = 0
		\]
		where $f(\theta)$  was defined in \eqref{eq:crisantieigenvalue} and $\tilde \theta_i$ are the eigenvalues of the matrix $(\sqrt{\theta_i \theta_j} Q_{ij} )_{i,j \leq k(N)} = D_N^{\frac{1}{2}} Q_N D_N^{\frac{1}{2}}$. 
	\end{prop}
	
	\begin{proof}
		Using the same notation as in Section~\ref{sec:sphericalSK}, it follows that asymptotically almost surely
		\begin{align*}
			F^\epsilon_N(Q) &= \frac{1}{Nk(N)} \ln \int \1( ||| E^\top E - Q ||| \leq \epsilon ) e^{ N \tr( E^\top A_N E D_N ) } \, d E
		\end{align*}
		where $E = (e_1, \dots, e_k) \in \R^{N \times k}$ and $dE = de^1 \dots de^k$ is the uniform measure on the product of $k(N)$ unit spheres. To simplify the notation, we dropped the dependence on $N$ of many terms. We begin by approximating the product of uniform measures on a sphere with a Gaussian measure. We will then do a change of variables to recover the formula for the $Q$ constrained overlaps from the $I$ constrained overlaps.
		
		Let $\gamma_N$ be the Gaussian measure on $\R^N$ with variance $\frac{1}{N} I$. By rotational invariance, we can write $x \in \R^{N}$ in its polar form $x = r e$, where its angular part $e$ is on the unit sphere and its radial part $r \in \R^{+}$. If $x$ has law $\gamma_N$, then the random variables $e$ and $r$ are independent and $e$ is uniform on the unit sphere by rotational invariance of the Gaussian. Let $p_r$ denote the law of $r$. 
		
		Since $Q_N^\epsilon:= \{ (e_i)_{i \leq k(N)} ~:~ ||| E E^\top - Q ||| \leq \epsilon \}$ is a measure $0$ set under $\gamma_N$, we consider the $\epsilon$ enlargement of this constraint,
		\[
		\Omega_N^{\epsilon} = \big\{ (r_ie_i)_{i \leq k(N)}  ~:~ (e_i)_{i \le k(N)} \in Q_N^\epsilon, (r_i)_{i \le k(N)} \in [\sqrt{ 1 - \epsilon}, \sqrt{1 + \epsilon}] \big\}.
		\]
		We have
		\[
		\frac{1}{N k(N)} \ln \int \1_{\Omega_N^\epsilon} e^{ N \tr( X^\top A_N X D_N) }  d\gamma^k_N(x) = 	\frac{1}{N k(N)} \ln \int_{ [ \sqrt{1 - \epsilon}, \sqrt{1+\epsilon}]^k } \int \1_{Q_N^\epsilon} e^{ N \tr( (RE)^\top  A_N (R E) D_N) }  dE dp^k_r(r)
		\]
		where $R = \diag(r_1, \dots, r_{k(N)}) \in \R^{k(N) \times k(N)}$. Since $|||R - I||| \leq \epsilon$ on the region of integration, we have by the computation following \eqref{eq:vonneumann} that
		\[
		| N \tr( E^\top  A_N E D_N) - N \tr( (RE)^\top  A_N (R E) D_N)| \leq M N k(N) \epsilon
		\] 
		where the constant $M$ only depends on the norms of the matrices $A_N$ and $D_N$, which are bounded. Therefore, 
		\begin{align}
			&\frac{1}{Nk(N)} \ln \int \1_{\Omega_N^\epsilon} e^{ N \tr( X^\top A_N X D_N) }  d\gamma^k_N(x) \notag \\&= 	\frac{1}{Nk(N)} \ln \int \1_{Q_N^\epsilon} e^{ N \tr( (RE)^\top  A_N (R E) D_N) }  dE  +  \frac{\ln \gamma_N (E_N^\epsilon) }{ k(N)N} + O(\epsilon^2) \label{eq:Gausstosphere}
		\end{align} 
		where
		\[
		E_N^\epsilon = \big\{ x \in \R^N ~:~ \|x\|^2 \in [1 - \epsilon, 1 + \epsilon ]  \big\} .
		\]
		If $x \sim \gamma_N$, then $\E \| x \|^2 = 1$, so the law of large numbers implies that for every fixed $\epsilon > 0$,
		\[
		\gamma_N (E_N^\epsilon) \to 1,
		\]
		so the error term $\frac{\ln \gamma_N (E_N^\epsilon) }{N} = o_N(1)$ for every fixed $\epsilon$. This implies that it suffices to study the Gaussian model.
		
		We now compute
		\begin{equation}\label{eq:gaussianFE}
			\frac{1}{Nk(N)} \ln \int \1_{\Omega_N^\epsilon} e^{ N \tr( X^\top A_N X D_N) }  d\gamma^{k(N)}_N(x) = \frac{1}{Nk(N)} \ln \frac{1}{(2\pi)^{\frac{k(N)N}{2}}} \int \1_{\Omega_N^\epsilon} e^{ N \tr( X^\top A_N X D_N) } e^{-\frac{\tr(X^\top X)}{2}} dX
		\end{equation}
		where $dX$ is the Lebesgue measure on $\R^{N \times k(N)}$. Let 
		\[
		\tilde \Omega(\delta) = \{ ||| X^\top X - Q_N ||| \leq \delta \} .
		\]
		Clearly we can find a $\delta_1(\epsilon)$ and $\delta_2(\epsilon)$ such that
		\begin{equation}\label{eq:setcontainment2}
		\tilde \Omega(\delta_1) \subseteq \Omega_N^\epsilon \subseteq \tilde \Omega(\delta_2).
		\end{equation}
		
		We start by proving an upper bound for \eqref{eq:gaussianFE}, and the lower bound will be similar. On the set $\tilde \Omega(\delta_2)$ we have
		\[
		|\tr( X^\top X) - \tr(Q_N^{-1} X^\top X )| \leq |\tr(Q_N) - \tr(I)| + O(k(N)\epsilon) \leq O(k(N)\epsilon)
		\]
		because $\tr(Q) = k(N)$ by Assumption~\ref{assum:AQ}. Therefore, \eqref{eq:gaussianFE} is upper bounded by
		\begin{equation}\label{eq:FEgauss1}
		 \frac{1}{N k(N)} \ln \frac{1}{(2\pi)^{\frac{k(N)N}{2}}} \int \1_{\tilde \Omega(\delta_2)} e^{ N \tr( X^\top A_N X D_N) } e^{-\frac{\tr(X^\top Q_N^{-1} X)}{2}} dX + O(\epsilon).
		\end{equation}
		Since the entries of $Q_N$ are bounded, there exists absolute constant $A$ that only depends on the uniform lower bound of the operator norm of $Q_N$ in Assumption~\ref{assum:AQ} such that
		\begin{equation}\label{eq:setcontainment}
			\{ ||| X^\top X - Q_N ||| \leq \epsilon \} \subseteq \{ ||| (XQ_N^{-1/2})^\top (XQ_N^{-1/2}) - I ||| \leq A \epsilon \}.
		\end{equation} 
		Therefore, if we do the linear change of variables $Y =  XQ_N^{-1/2}$ then we have the following upper bound of \eqref{eq:FEgauss1}
		\begin{align*}
		&\frac{1}{Nk(N)} \ln \frac{1}{(2\pi)^{\frac{k(N)N}{2}}} \int \1( ||| Y^\top Y - I ||| \leq A\epsilon ) e^{ N \tr( Q^{1/2} Y^\top  A_N Y Q^{1/2} D_N ) } e^{- \tr( \frac{Y^\top Y}{2} )} \det( Q_N^{N/2} ) \, d Y
		\\&\leq \frac{1}{Nk(N)}  \ln\int \1( ||| Y^\top Y - I ||| \leq A\epsilon ) e^{ N \tr(Y^\top  A_N Y D_N(Q_N) ) } \, d \gamma_N^{k(N)}(y) + \frac{1}{2 k(N)} \ln \det(Q_N)
		\end{align*}
		where 
		  $D_N(Q_N) :=Q_N^{1/2}D_NQ_N^{1/2}$. Since the spectrums are invariant under cyclic permutations, $\tilde \theta$ are also the eigenvalues of the matrix $(\sqrt{\theta_i \theta_j} Q_{ij} )_{i,j \leq k}$. Next, we can apply \eqref{eq:Gausstosphere} to replace the Gaussian integral with one over the uniform samples on a sphere, giving the upper bound
		\[
		\frac{1}{N k(N)}  \ln\int \1( ||| E^\top E - I ||| \leq C\epsilon ) e^{ N \tr(E \tilde A_N E^\top D_N(Q_N) ) } \, d e^1 \cdots de^{k(N)} + \frac{1}{2k(N)} \ln \det(Q_N) + o(1)
		\]
		where the $o(1)$ term tends to $0$ as $N\to \infty$ and $\epsilon \to 0$, and the constant $C$ is possibly different from the one appearing in the previous line. 
		We can then apply the result form Lemma~\ref{lem:apporxidentity} to approximate first term with a spherical integral, to arrive at the upper bound
		\[
		I_N( D_N(Q_N),A_N) + \frac{1}{2k(N)} \ln \det(Q_N) + o(1) + O(\epsilon)
		\]
		This is explicitly computed using Theorem \eqref{theo:positivetheo} and Proposition~\ref{prop:sphericalSK} to prove 
		\[
		\lim_{\epsilon \to 0} \lim_{N\to \infty} \left[ F_N^{\epsilon} (Q_N) - \bigg( \frac{1}{2k(N)} \sum_{i = 1}^k f(\tilde \theta_i) + \frac{1}{2k(N)} \ln \det (Q_N) \bigg) \right] \leq 0.
		\]
		
	The matching lower bound is identical and follows from the lower set containment in \eqref{eq:setcontainment2} and \eqref{eq:setcontainment}.
	\end{proof}
	
	\begin{rem}
		When $k(N) = k$ is independent of $N$, the large deviations proof in Proposition~\ref{prop:vectorspinspherical} implies that the upper bound proved using interpolation in \cite[Theorem~1]{PTSPHERE} is sharp. Indeed, since $\det( D_\theta^{1/2} Q D_\theta^{1/2} ) = \prod_{j = 1}^k \theta_j \det(Q)$, 
		\begin{equation}\label{eq:equivratefunc}
			\sum_{i = 1}^k f(\tilde \theta_i) + \frac{1}{2} \ln \det (Q) = \sum_{i = 1}^k f(\tilde \theta_i) + \frac{1}{2}  \ln (\tilde \theta_i) - \frac{1}{2} \ln(\theta_i) 
		\end{equation}
		and simplifying yields
		\[
		f(\tilde \theta_i) + \frac{1}{2}  \ln (\tilde \theta_i) - \frac{1}{2} \ln(\theta_i) = 
		\begin{cases}
			\frac{1}{4} \tilde \theta_i^2 + \frac{1}{2} \ln \tilde \theta - \frac{1}{2} \ln \theta_i & \tilde \theta_i < 1\\
			\tilde \theta_i - \frac{3}{4} - \frac{1}{2} \ln \theta_i & \tilde \theta_i \geq 1.
		\end{cases}
		\]
	\end{rem}

	\section{Application to Matrix Factorization}\label{sec:matestimation}
	The framework for this application is adapted from a recent articles on extensive rank matrix factorization \cite{barbier2021statistical,maillard2021perturbative}. A model of a simple denoising problem, examines spiked matrices of the form
	\begin{equation}\label{eq:spiked}
		Y_N = G_N + \sqrt{\gamma} U^\top D_N U:=  G_N + \sqrt{\gamma} X_N
	\end{equation}
	where $G_N$ is a $N \times N$ GOE matrix, $U$ are random orthogonal matrices sampled according to the Haar measure on the orthogonal group, and $D_N = \diag(\theta_1, \dots, \theta_{k(N)}, 0,\dots,0)$ is a random rank $k(N)$ diagonal matrix with non-negative entries $\theta_i \geq 0$ and joint eigenvalue distribution $P_D$. In applications, the matrix $G_N$ is the noise matrix, $X_N = U^\top D_N U$ is the signal, and the parameter $\gamma$ controls the signal to noise ratio. In this setting, the hidden matrix $X_N$ is a general random rotationally invariant symmetric matrix with $O(\frac{1}{\sqrt{N}})$ entries. 
	
	We also require an assumption on the joint distribution $P_D$ of the diagonals in the matrix $D_{N}$. 
	\begin{assum}\label{assum:eigdist}
		Suppose the empirical distribution  $\frac{1}{k(N)}\sum_{i=1}^{k(N)}\delta_{\theta_{i}}$  converges under $P_D$ almost surely towards a probability measure $\eta$, and that its law 
		satisfies a large deviations principle with good rate function $\Gamma$ and speed $k(N) N$. We moreover assume that $P_D$ is compactly supported in $[-M,M]^{k(N)}$ for some finite $M$. 
	\end{assum}
	
	\begin{rem}
		For example, if we take a deterministic $D_N = (1,0,\dots, 0)$ then this model is the traditional spiked matrix model from a uniform prior on the sphere with signal to noise ratio $\gamma$. 
	\end{rem}
	
	To estimate the matrix $X_N$ from the signal matrix $Y_N$, we study posterior probability measure
	\begin{align}
		dP(X  | Y) &\propto  e^{ -\frac{1}{4} N \tr ( Y_N - \sqrt{ \gamma }U D_N U^\top )^2 } dU dP_D(\theta) \notag
		\\&\propto  \exp \bigg( \frac{N\sqrt{\gamma}}{2} \tr( U^\top Y_N U D_N )  - \frac{N \gamma}{4} \tr(D_N^2) \bigg) dUdP_D(\theta)\label{eq:posteriorig}.
	\end{align}
	
	The main quantity of interest is the mutual information $I(X,Y)$ between the signal $X$ and the data $Y$, which can be computed via the following entropy decomposition, see for example \cite[Equation~7]{barbier2021statistical}
	\begin{align}
		\frac{1}{Nk(N)}I_N(\gamma):=  \frac{\gamma}{4k} \sum_{i = 1}^k \theta^2_i - \frac{1}{Nk(N)} \E_Y \ln \int e^{- \frac{\gamma N}{4} \tr( D_N^2 ) } \bigg(\int e^{ \frac{\sqrt{\gamma} N}{2} \tr( U^\top Y_N U D_N ) } dU \bigg) d P_D(\theta) .  \label{eq:mutualinformationgrowingrank}
	\end{align}
	Given the mutual information, we can apply the I-MMSE Theorem \cite{IMMSE} to compute the minimal mean square error (MMSE)
	\begin{equation}\label{eq:IMMSE}
		\mathrm{MMSE}_N(\gamma) = \frac{1}{2 N k} \E \|X - \E [X ~|~ Y] \|^2_2 = 4\frac{d}{d\gamma}I_N(\gamma) + O(N^{-1})
	\end{equation}
	in the limit. 
	There is a factor $4$ instead of the usual $2$ that appears in the formula in \cite{IMMSE} because we are considering symmetric matrices, which only requires denoising the lower or upper triangle. By convexity of the mutual information with respect to $\gamma$, the I-MMSE theorem can be extended as $N \to \infty$,
	\[
	\lim_{N \to + \infty} \mathrm{MMSE}_N(\gamma) = 4\frac{d}{d\gamma} \lim_{N \to +\infty} I_N(\gamma) 
	\]
	at all points where the limiting mutual information is differentiable.
	
	By rotational invariance of $G_N$ and $X_N$, the spherical integral only depends on the specturms of $Y_N$ and $D_N$. The behavior of the eigenvalues of spiked matrices have been studied extensively in the past for finite rank perturbations in \cite{BGNSpiked} and sublinear rank perturbations in \cite{JYmeso}. We will  use the following result for the behavior of the extreme eigenvalues of a rank $k(N)$ spiked Gaussian matrix $Y_N$.
	\begin{prop}[{ \cite[Theorem~2.1]{BGNSpiked} and \cite[Theorem~2.8]{JYmeso} }] \label{prop:topeig} Suppose that $G_N$ is a GOE matrix and $$D_N = \diag(\theta_1, \dots, \theta_{k(N)}, 0,\dots,0)$$ is a determnistic rank $k(N)$ diagonal matrix with non-negative entries $\theta_i \geq 0$. Let $\lambda_1 \geq \lambda_2 \geq \dots \geq \lambda_N$ denote the eigenvalues of the perturbed matrix 
		\[
		Y_N = G_N + U_N^\top D_N U_N
		\]
		where $U_N$ is $U$ a random orthogonal matrices sampled according to the Haar measure on the orthogonal group. Let $\sigma$ denote the semicircle distribution. For $1 \leq i \leq k$
		\[
		\lambda_i(\theta) \stackrel{d}{\to} \begin{cases}
			2 & \theta_i < \frac{1}{G_\mu(2)}\\
			G_\sigma^{-1}(\theta_i^{-1}) & \theta_i > \frac{1}{G_\sigma(2)}
			
		\end{cases} = \begin{cases}
			2 & \theta_i\leq 1 \\
			\theta_i + \frac{1}{\theta_i} & \theta_i > 1.
		\end{cases}
		\]

	\end{prop}
	These phase transitions are a special case of a more general phenonmenom called the BBP transition \cite{BBAP05}. The behavior of the spherical integrals in the extensive rank case when $\frac{k(N)}{N} \to \alpha > 1$ was studied in \cite{aliceexpansion} and the behavior in those models are fundamentally different than what happens when  $\frac{k(N)}{N} \to 0$.
	
	This fact will allow us to replace the random $Y_N$ with a deterministic matrix corresponding to the typical eigenvalues. We define the free entropy as
	\begin{equation}\label{eq:freentropyspiked}
		F_N(Y_N) = \frac{1}{Nk(N)} \ln \int e^{- \frac{\gamma N}{4} \tr( D_N^2 ) } \bigg(\int e^{ \frac{\sqrt{\gamma} N}{2} \tr( U^\top Y_N U D_N ) } dU \bigg) d P_D(\theta) .
	\end{equation}
	Let $A_N$ denote a matrix with eigenvalues $\lambda_1 \geq \lambda_2 \geq \dots \geq \lambda_N$. We choose the eigenvalues $\lambda_i$ for $i \geq k(N)$ such that the empirical distribution of the eigenvalues  $\hat \mu_{A_N}$ converges weakly to the semicircle distribution $\sigma$. Furthermore, the outlier eigenvalues $\lambda_1, \dots, \lambda_{k(N)}$ are given by $\lambda_i = f(\theta_i)$ where $f$ is the BBP transition function
	\begin{equation}\label{eq:bbp}
		f(x) = \begin{cases}
			2 & x \sqrt{\gamma} \leq 1 \\
			x \sqrt{\gamma} + \frac{1}{x \sqrt{\gamma}} & x \sqrt{\gamma} > 1.
		\end{cases}
	\end{equation}
	that appears in Propostion~\ref{prop:topeig}. The outlying eigenvalues have have extremal empirical measure $\hat \nu_N = f_{\#}(\mu_{D_N})$, where $\mu_{D_N}$ is the spectral distribution of $D_N$ under $P_D$. The next lemma states that we can replace the $Y_N$ in the free entropy \eqref{eq:freentropyspiked} with its deterministic counterpart $A_N$.
	
	\begin{lemma}\label{lem:deterministicdisorder}
		If $D_N$ satisfies Assumption~\ref{assum:eigdist}, then
		\[
		\lim_{N\to \infty} | \E_Y F_N(Y_N) - F_N(A_N)  | = 0.
		\]
	\end{lemma}
	\begin{proof}
		Let 
		\[
		\hat \mu_Y = \frac{1}{N} \sum_{i = 1}^N \delta_{\lambda_i (Y)}
		\]
		denote the empirical measure of $Y$ and let
		\[
		\hat \eta_Y = \frac{1}{k(N)} \sum_{i = 1}^{k(N)} \delta_{\lambda_i (Y)}
		\]
		denote the extremal empirical measure of $Y$. Recall that Assumption~\ref{assum:eigdist} implies that $\hat \eta_Y \to \eta$. Let $d$ be a distance on $\mathcal{P}(\R)$ metrizing the topology of convergence in law. Consider the event
		\[
		C_N = \{ d(\hat \mu_Y, \sigma) + d(\hat \eta_Y, f_{\#} \eta) \leq \delta \}
		\]
		which denotes the event that both the empirical measure and extremal empirical measure converges to its typical value. By the almost sure convergence of the empirical measures in Lemma~\ref{lem:convergenceempirical}  and \cite[Corollary~2.10]{JYmeso}, have that the probability of $C_N$ tends to $1$ in the limit.
		
		Consider the following decomposition
		\[
		\E_Y F_N(Y_N) = \E_Y \1_{C_N} F_N(Y_N) + \E_Y \1_{C^c_N} F_N(Y_N).
		\]
		By construction,  the empirical measures of $A_N$, $\hat \mu_A$ and $\hat \eta_A$, converge to $\sigma$ and $\eta$ respectively so
		\[
		\lim_{N \to + \infty} | \E_Y \1_{C_N} F_N(Y_N) - F_N(A_N)| = 0
		\]
		by continuity and the fact that the limit only depends on the eigenvalues of $Y_N$ and $A_N$. For the second term,  the von Neumann trace inequality implies the following uniform bound
		\[
		\sup_{D_N} \bigg|\frac{1}{N k(N)} \ln I_N(D_N, Y_N) \bigg| \leq \lambda_1(Y) M.
		\]
		This implies that for any $L > M + \frac{1}{M} + 1$,
		\[
		|\E_Y \1_{C^c_N} F_N(Y_N)| \leq \E_Y (M^2 + \lambda_1(Y) M ) \1_{C^c_N}  \leq (M^2 + ML) \Pp(C^c_N) + M \E [\lambda_1(Y) \1(\lambda_1(Y) > L) ] .
		\]
		The first term is arbitrarily small in the limit because of the almost sure weak convergence of the empirical measures. The second term is arbitrarily small because \cite[Theorem~2.8]{JYmeso} gives exponential control of the top eigenvalue around $\theta_1 + \frac{1}{\theta_1}$. Combining both implies that the upper bound tends to $0$ in the limit. 
	\end{proof}
	
	Our focus now is to compute
	\begin{align}
		\frac{1}{Nk(N)}I_N(\gamma) =  \frac{\gamma}{4 k(N)} \sum_{i = 1}^{ k(N)} \theta^2_i - \frac{1}{Nk(N)}  \ln \int e^{- \frac{\gamma N}{4} \tr( D_N^2 ) } \bigg(\int e^{ \frac{\sqrt{\gamma} N}{2} \tr( U^\top A_N U D_N ) } dU \bigg) d P_D(\theta)  \label{eq:mutualinformationgrowingrankdeterministic}
	\end{align}
	where $A_N$ is a deterministic matrix with limiting empirical measure and extremal empirical measure converging to the same almost sure limit as the empirical measure of $Y_N$ defined in \eqref{eq:bbp}. 
	
	\subsection{Low Rank Matrix Estimation with Deterministic $D_N$}\label{sec:finrank}
	
	We first restrict ourselves to the setting simpler where $k$ is fixed and independent of $N$, and the matrix $D_N:= D_k = \diag(\theta_1, \dots, \theta_k, 0, \dots, 0)$ is deterministic. The random setting with $k = o(N)$ rank will be considered in Subsection~\ref{sec:growingrank}. 
	
	By our simplifying choice of the distribution of $X$, the posterior in \eqref{eq:posteriorig} simplifies to
	\begin{equation}\label{eq:posterior}
		dP(X | Y) \propto e^{-\frac{N\gamma}{4} \sum_{i = 1}^k \theta^2_i} \exp \frac{ \sqrt{\gamma} N }{2} \tr\bigg(   U^\top Y U D_k  \bigg) dU.
	\end{equation}
	Consequently, the mutual information \eqref{eq:mutualinformationgrowingrankdeterministic}, has a simpler structure
	\begin{align}
		\frac{1}{Nk}I_N(\gamma) &=  \frac{\gamma}{4k} \sum_{i = 1}^k \theta^2_i - \frac{1}{Nk} \ln  e^{-\frac{N\gamma}{4} \sum_{i = 1}^k \theta^2_i} \int \exp \frac{ \sqrt{\gamma} N }{2} \tr\bigg(   U^\top A_N U D_k  \bigg) dU \notag
		\\&= \frac{\gamma}{2k} \sum_{i = 1}^k \theta^2_i - \frac{1}{Nk}  \ln  \int \exp \frac{ \sqrt{\gamma} N }{2} \tr\bigg(   U^\top A_N U D_k  \bigg) dU\label{eq:mutualinformation}.
	\end{align}
	
	We can use the spherical integrals to explicitly compute this quantity.	Recall that for the Stieltjes transform is given in \eqref{eq:stieljessc} and in particular $G_\sigma^{-1}(\theta) = \theta + \frac{1}{\theta}$ by \eqref{eq:invstieltjes} and $G_\sigma(2) =1$. By Proposition~\ref{prop:topeig}, it follows that there will be at most $k$ outlying eigenvalues given by
	\begin{equation}\label{eq:spikedtopeigenvalues1}
		\lambda_i(\gamma,D_N) = \lambda_i(\gamma \theta^2_i) = \begin{cases}
			2 & \gamma \leq \frac{1}{\theta_i^2} \\
			\sqrt{\gamma} \theta_i + \frac{1}{\sqrt{\gamma} \theta_i} & \gamma  > \frac{1}{\theta^2_i}
		\end{cases} \qquad \text{for} \qquad  1 \leq i \leq k.
	\end{equation}
	This explicit formula for the eigenvalues from Proposition~\ref{prop:topeig} will allow us explicitly compute $I_N$ and $	\mathrm{MMSE}_N(\gamma)$ with Theorem~\ref{theo:positivetheo}. 
	
	\begin{prop}[Matrix Factorization with Deterministic $D_N$]\label{prop:matfact}
		For fixed $k \geq 1$ and any $\theta_1, \dots, \theta_k > 0$. The mutual information of the spiked matrix model is given by
		\begin{align}
			\lim_{N\to \infty} \frac{1}{Nk}I_N(\gamma) &= \frac{\gamma}{2k} \sum_{i = 1}^k \theta^2_i - \frac{1}{2k} \sum_{i = 1}^k J(\sqrt{\gamma} \theta_i, \lambda_i(\gamma\theta^2_i), \mu) \notag
			\\&=  \frac{1}{k} \sum_{i = 1}^k \frac{\gamma\theta^2_i}{4} \1\Big(\gamma \leq \frac{1}{\theta_i^2}\Big) + \frac{1}{k}  \sum_{i = 1}^k \bigg( \frac{\ln \gamma\theta^2_i}{2} + \frac{1}{4 \gamma\theta^2_i} \bigg) \1\Big(\gamma > \frac{1}{\theta^2_i} \Big) \label{eq:Iformula}
		\end{align}
		and the asymptotic MMSE is given by
		\begin{align}\label{eq:MMSEformula}
			\lim_{N \to + \infty} \mathrm{MMSE}(\theta)	 &= \frac{2}{k} \sum_{i = 1}^k \theta^2_i - \frac{2}{k} \sum_{i = 1}^k \frac{d}{d\gamma} J(\sqrt{\gamma} \theta_i, \lambda_i(\gamma\theta^2_i), \mu) \notag
			\\&= \frac{1}{k} \sum_{i = 1}^k \frac{\theta^2_i}{4} \1\Big(\gamma \leq \frac{1}{\theta_i^2}\Big) + \frac{1}{k}  \sum_{i = 1}^k \bigg( \frac{1}{2\gamma} - \frac{1}{4 \gamma^2\theta^2_i} \bigg) \1\Big(\gamma > \frac{1}{\theta^2_i}  \Big) .
		\end{align}
		where $J$ was defined in Definition~\ref{def:J}.
	\end{prop}
	
	\begin{proof}
		It suffices to only compute the mutual information because the MMSE follows immediately from the relationship between the mutual information and minimal mean squared error by \eqref{eq:IMMSE}. By \eqref{eq:mutualinformation}, we have
		\[
		\frac{1}{Nk} I_N(\gamma) =  \frac{\gamma}{2k} \sum_{i = 1}^k \theta^2_i - \frac{1}{Nk} \E \ln \int \exp \frac{ \sqrt{\gamma} N }{2} \tr\bigg(   U^\top A_N U D_k  \bigg) dU.
		\]
		We can use Theorem~\ref{theo:positivetheo} to compute the limiting free entropy of this model. We have
		\[
		\lim_{N \to + \infty} \frac{1}{N} \E \ln \int \exp  \frac{\gamma N}{2} \tr\bigg( U^\top Y U D_k  \bigg) dU  = \frac{1}{k} \sum_{i = 1}^k J(\sqrt{\gamma} \theta_i,\lambda_i(\gamma\theta^2_i), \sigma)
		\]
		because the limiting spectral distribution of $A_N$ is the semicircle law $\mu$. This proves the first equality in \eqref{eq:Iformula}. 
		
		The functional $J(\sqrt{\gamma} \theta_i,\lambda_i(\gamma\theta^2_i), \sigma)$ is explicit and can be computed similarly using the same chain of computations in the spherical SK model. Notice that $G_\sigma(\lambda(\gamma \theta_i^2)) = 1 \wedge \frac{1}{\sqrt{\gamma} \theta_i}$. We will show that
		\begin{equation}\label{eq:Jspikedmatrix}
			J\bigg(\sqrt{\gamma} \theta_i,\lambda_i(\gamma\theta^2_i), \sigma \bigg) = \begin{cases}
				\frac{\gamma \theta_i^2}{2} & \gamma \leq \frac{1}{\theta_i^2}
				\\\gamma \theta_i^2 - \ln(\gamma \theta_i^2) - \frac{1}{2 \gamma \theta_i^2}  & \gamma > \frac{1}{\theta_i^2}.
			\end{cases}
		\end{equation}
		We do the change of variables and consider $x = \gamma \theta^2_i$ and compute
		\[
		J(\sqrt{x}, \lambda_i(\sqrt{x}), \sigma ).
		\]
		
		We first consider the case that $x < 1$. In this case, the computation follows from the computations with the SK model evaluated at inverse temperature $\sqrt{x}$ (see the proof of Proposition~\ref{prop:sphericalSK}) so
		\[
		J(\sqrt{x}, 2, \sigma) = \frac{x}{2}
		\]
		proving the first case in \eqref{eq:Jspikedmatrix}.
		
		We now consider the complicated case when $x > 1$. Notice that $G_\sigma( \sqrt{x} + \frac{1}{\sqrt{x}}) = \frac{1}{\sqrt{x}}$, so we are in the region where $v = \sqrt{x} + \frac{1}{\sqrt{x}}$ in Definition \ref{def:J}. On this region, we have
		\[
		J\bigg(\sqrt{x}, \sqrt{x} + \frac{1}{\sqrt{x}}, \sigma \bigg) = \sqrt{x} \bigg( \sqrt{x} + \frac{1}{\sqrt{x}} \bigg) - \ln \sqrt{x} - h\bigg( \sqrt{x} + \frac{1}{\sqrt{x}} \bigg) - 1.
		\]
		Using the formula for the logarithmic potential $h$ defined in \eqref{eq:logpotential} we have $h( \sqrt{x} + \frac{1}{\sqrt{x}} ) = \frac{\ln x}{2} + \frac{1}{2x}$ for $x > 1$,	proving the second case in \eqref{eq:Jspikedmatrix}. 
		
		Next, using the relationship for the mutual information and the free energy \eqref{eq:IMMSE}, we see that
		\[
		\lim_{N \to \infty} \mathrm{MMSE}_N(\gamma) = 4 \frac{d}{d\gamma} 	I_N(\gamma) = \frac{2}{k} \sum_{i = 1}^k \theta^2_i - \frac{2}{k} \sum_{i = 1}^k \frac{d}{d\gamma} J(\sqrt{\gamma} \theta_i,\lambda_i(\gamma\theta^2_i), \mu) 
		\]
		proving the first equation in \eqref{eq:MMSEformula}. The result 
		\begin{equation}\label{eq:derivJspikedmatrix}
			\frac{d}{d\gamma} J\bigg(\sqrt{\gamma} \theta_i,\lambda_i(\gamma\theta^2_i), \sigma \bigg) = \begin{cases}
				\frac{\theta_i^2}{2} & \gamma \leq \frac{1}{\theta_i^2}
				\\ \theta_i^2 - \frac{1}{\gamma} + \frac{1}{2 \gamma^2 \theta_i^2}  & \gamma > \frac{1}{\theta_i^2}.
			\end{cases}
		\end{equation}
		can be computed by taking the derivatives of \eqref{eq:Jspikedmatrix}, which proves the second equality in \eqref{eq:MMSEformula}.
	\end{proof}
	
	As a consequence of Proposition~\ref{prop:matfact}, the formula in the rank 1 matrix estimation problem, which was proven earlier in \cite{miolanefundamentallimits},  is simple.
	
	\begin{Ex}[Rank 1 Matrix Estimation]\label{ex:rank1}
		We will now demonstrate how one can derive the fundamental limits for rank 1 matrix estimation using the spherical integral formula. Without loss of generality, we take $\theta_1 = 1$. Applying Proposition~\ref{prop:matfact}, we see that the limiting mutual information in this model is
		\[
		\lim_{N \to + \infty} \frac{1}{N}I_N(\gamma) = \begin{cases}
			\frac{\gamma}{4}	 & \gamma \leq 1\\
			\frac{\ln(\gamma)}{2} + \frac{1}{4\gamma}	 & \gamma > 1
		\end{cases}
		\]
		and the MMSE is
		\[
		\lim_{N \to +\infty} \mathrm{MMSE}(\gamma) = \begin{cases}
			1 &  \gamma \leq 1\\
			\frac{1}{\gamma} \Big(2 - \frac{1}{\gamma} \Big)  &  \gamma > 1
		\end{cases}.
		\]
	\end{Ex}

	\subsection{Growing Rank Matrix Estimation}\label{sec:growingrank}
	
	We now prove a limiting formula for the matrix factorization problem when $k(N)$ increases in $N$ and the perturbation $D_N$ is random.	
	
	Because the asymptotics of the spherical integrals in the growing rank case are given by the sums of the one dimensional sperical integrals, the results for the finite rank case discussed in Section~\ref{sec:finrank} generalizes to the growing rank case. We have the following limit for the mutual information.
	
	\begin{prop}
		Suppose the rank $k(N)$ of the signal satisfies $\lim_{N \to \infty} k(N) = \infty$ and $\lim_{N\to \infty}\frac{k(N)}{N} = 0$. If $D_N$ satisfies Assumption~\ref{assum:eigdist}, then for $f$ defined in \eqref{eq:bbp} and quantile functions $Q_\mu$ defined in \eqref{eq:quantile},
		\begin{align*}
			\lim_{N \to \infty} \frac{1}{Nk(N)} I_N(\gamma)  &= \frac{\gamma}{4} \int_0^1 (Q_\eta (x) )^2 dx - \sup_{\nu} \bigg( - \frac{1}{4} \int_0^1 (Q_\nu(x) )^2 \, d x + \frac{1}{2} \int_0^1 J( \sqrt{\gamma} Q_\nu( x ), f( \sqrt{\gamma} Q_\eta(x) ) , \mu ) \, d x - \Gamma(\nu) \bigg).
		\end{align*}
	\end{prop}
	
	\begin{proof}
		The first term in the mutual information \eqref{eq:mutualinformationgrowingrank} is trivial, and converges to
		\begin{equation}\label{eq:limitfirstterm}
			\lim_{N\to \infty} \frac{\gamma}{4k(N)} \sum_{i = 1}^{k(N)} \theta^2_i =  \frac{\gamma}{4} \int_0^1 x^2 d\eta(x)  = \frac{\gamma}{4} \int_0^1 (Q_\eta (x) )^2 dx.
		\end{equation}
		We focus on computing the limit of the second term of \eqref{eq:mutualinformationgrowingrank}, which we will denote by
		\begin{align*}
			F_N =  \frac{1}{Nk(N)} \ln \int e^{- \frac{\gamma N}{4} \tr( D_N^2 ) } \bigg(\int e^{ \frac{\sqrt{\gamma} N}{2} \tr( U^\top A_N U D_N ) } dU \bigg) d P_D(\theta) .
		\end{align*}
		Using Theorem~\ref{maintheo}, we can compute the spherical integral on the inside to determine that
		\[
		\frac{1}{Nk(N)} \ln \int e^{- \frac{N k(N)}{4}  \frac{1}{k(N)}\sum_{i = 1}^{k(N)} \theta^2_i + N k(N) \frac{1}{k(N)} \sum_{i = 1}^{k(N)} J(\sqrt{\gamma} \theta_i, \lambda_i , \mu ) + o_N(1) }   d P_D(\theta).
		\]
		Recall that the extremal eigenvalues $\lambda_i$ converge weakly to $f^\gamma_{\#} \eta$ where $\eta$ is the limiting eigenvalue distribution of $D_N$ and $f^\gamma$ is the scaled BBP transition map \eqref{eq:bbp}, 
		\[
		f^\gamma(x)  = \begin{cases}
			2	&x \leq \frac{1}{\sqrt{\gamma}}
			\\\sqrt{\gamma} x + \frac{1}{ \sqrt{\gamma} x} &x \geq \frac{1}{\sqrt{\gamma}}.
		\end{cases}
		\]
		Therefore, if the empirical measure of the $\theta$ converges to $\nu$, then
		\[
		\lim_{N \to +\infty} \frac{1}{k(N)} \sum_{i = 1}^{k(N)} J(\sqrt{\gamma} \theta_i, \lambda_i , \mu ) = \frac{1}{2} \int_0^1 J( \sqrt{\gamma} Q_\nu( x ), f^\gamma( Q_\eta(x ) ) , \sigma ) \, d x.
		\]
		We used the quantile functions to couple the ordered eigenvalues $\lambda_i$ in the limit with the ordered eigenvalues $\sqrt{\gamma} \theta_i$. Lastly, if the law $P_D$ of $\theta$ satisfies a large deviations principle with rate function $\Gamma$ and speed $k(N)N$ so that $\inf\{ \Gamma(\mu):\mu(x^{2})\ge L\}$ goes to infinity with $L$, then by Varadhan's lemma,
		\begin{equation}\label{lkj}
			\lim_{N \to \infty} F_N = \sup_{\nu} \bigg( - \frac{\gamma}{4} \int_0^1 x^{2}d\nu(x)  + \frac{1}{2} \int_0^1 J( \sqrt{\gamma} Q_\nu( x ), f^\gamma( Q_\eta(x ) ) , \sigma ) \, d x - \Gamma(\nu) \bigg),
		\end{equation}
	where the supremum is taken over probability measures $\nu$ with finite second moment. 
	Combining \eqref{lkj} and \eqref{eq:limitfirstterm} with the decomposition \eqref{eq:mutualinformationgrowingrank} finishes the proof.
	\end{proof}
	
	\begin{rem}
		The rate function $J$ has an explicit form. The terms in $J$ were computed in the proof of Proposition~\ref{prop:sphericalSK} and Proposition~\ref{prop:matfact}. By the definition of $G_\sigma$, we see that
		\[
		G_{\sigma}( f^\gamma( Q_\eta(x ) ) ) = \begin{cases}
			1 &  \sqrt{\gamma} Q_\eta(x )  \leq 1 \\
			\frac{1}{ \sqrt{\gamma} Q_\eta(x ) } & \sqrt{\gamma} Q_\eta(x )  \geq 1.
		\end{cases} = 1 \wedge 	\frac{1}{ \sqrt{\gamma} Q_\eta(x ) } 
		\]
		Since $G_{\sigma}( f^\gamma( Q_\eta(x ) ) )$ is decreasing in $x$ and $\sqrt{\gamma} Q_\nu( x )$ is increasing in $x$, so we can define $x^*$  to be the smallest number (which may be infinite) such that $G_{\sigma}( f^\gamma( Q_\eta(x^* ) ) ) = \sqrt{\gamma} Q_\nu( x^* )$. It follows that $\sqrt{\gamma} Q_\nu( x ) \leq G_{\sigma}( f^\gamma( Q_\eta(x ) ) )$ for $x < x^*$ and $\sqrt{\gamma} Q_\nu( x ) \geq G_{\sigma}( f^\gamma( Q_\eta(x ) ) )$ for $x > x^*$, so
		\[
		v(  f^\gamma( Q_\eta(x ) ), \sqrt{\gamma} Q_\nu( x )) = \begin{cases} 	
			\sqrt{\gamma} Q_\nu( x ) + \frac{1}{\sqrt{\gamma} Q_\nu( x )} &x < x^*\\
			f^\gamma( Q_\eta(x ) )  &x > x^*.
		\end{cases}
		\]
		Substituting this into Definition~\ref{def:J} and using formulas \eqref{eq:stieljessc} and \eqref{eq:logpotential} implies
		\[
		J( \sqrt{\gamma} Q_\nu( x ), f^\gamma( Q_\eta(x ) ) , \sigma ) = \begin{cases} 	
			\frac{\gamma Q^2_\nu( x )}{2}	 &x < x^*\\
			\sqrt{\gamma} Q_\nu( x ) f^\gamma( Q_\eta(x ) )  - \ln \sqrt{\gamma} Q_\nu( x )  - h( f^\gamma(Q_\eta(x)) ) - 1   &x > x^*,
		\end{cases}
		\]
		where
		\[
		h(z) = \frac{z^2}{4} - \frac{z \sqrt{z^2 - 4}}{4} + \ln \frac{z + \sqrt{z^2 - 4}}{2}  - \frac{1}{2} \quad \text{ for } z \geq 2.
		\]
	\end{rem}

	\appendix 
		\section{Annealed spherical integral lower bound: Proof of lemma \ref{excessmass} }\label{app:A}
	To prove this lemma, we are going to prove first to separate the terms in the sum that define $A_N^{(\epsilon)}$ into two terms, a term covering the diagonal entries of $(U^* D_N U)$ and another one covering the off-diagonal entries. First, let us deal with the diagonal entries: 
	\begin{lemma} \label{excessmass1}
		The random variable $\max_{i =1,\dots,N} |(U^* D_N U)_{i,i}|$ converges to $0$ in probability.  \end{lemma}
	\begin{proof}
		By unitary invariance all the $|(U^* D_N U)_{i,i}|$ have the same distribution. If $(e_i)_{1 \leq i \leq N}$ is the first column of $U$, we have that: 
		\[ |(U^* D_N U)_{1,1}| = \Big| \sum_{i=1}^{k(N)} \theta_{-i}^N e_i ^2 + \sum_{i=1}^{k(N)}\theta_{i}^N e_{i+k(N)} ^2 \Big| \leq K \sum_{i=1}^{2k(N)} e_i^2 \]
		and $\sum_{i=1}^{2k(N)} e_i^2$ is a beta variable of parameter $( \beta k(N) , \frac{\beta N - 2 \beta k(N)}{2})$. Using a classical Laplace method, one has for every $\epsilon > 0$ the existence of some $c(\epsilon) > 0$ such that: 
		\[ \Pp[ |(U^* D_N U)_{1,1}| \geq \epsilon] = O(e^{-N c(\epsilon)}) \]
		The lemma then come from a union bound on the $i$. 
	\end{proof}
	
	With the following lemma, we deal with the off-diagonal entries. 
	\begin{lemma}\label{excessmass2}
		If $k(N) = o(N / \ln N) $ then for all $\epsilon > 0$: 
		\[ \Pp[ \exists i,j \in[1,N], \text{ such that }i \neq j \text{ and }\sqrt{N} |(U D_N U^*)_{i,j}| / 2 > \epsilon] = o(1) \]
	\end{lemma}
	\begin{proof}
		There again, for all $i \neq j$, $(U^* D_N U)_{i,j}$ has the same distribution as $(U^* D_N U)_{1,2}$. Let us denote $(e_i)_{1 \leq i \leq  N}$ and $(f_i)_{1 \leq i \leq  N}$ respectively the first and second columns of $U$. For $u \in \Ss^{\beta N-1}$, we denote $\Pi^{(u)}$ the orthogonal projection on the orthogonal of $Vect(u)$, then if we condition on $f$, $e$ is distributed uniformly on the sphere $\Pi^{(f)} (\Ss^{\beta N -1})$ of dimension $N-2$. Therefore, since: 
		\[ (U^* D_N U)_{1,2} = \langle e, D_N f \rangle = \langle e,\Pi^{(f)} D_N f \rangle \]
		conditionally on $f$, $ (U^* D_N U)_{1,2}^2$ has the law of a beta variable of parameters $\frac{\beta}{2}( 1, N-2)$ multiplied by $||\Pi^{(f)} D_N f||^2$. First, we have that 
		\[ ||\Pi^{(f)} D_N f||^2 \leq ||D_N f||^2 \leq K \sum_{i=1}^{2 k(N)} f^2_i \]
		Let $l(N) = \max(\ln N, 2 k(N))$. Let us prove that there is $C >0$ such that $\Pp[ ||\Pi^{(f)} D_N f||^2 > C l(N)/N] \leq N^{-2}$. 
		One can write that 
		\[ \Pp[ ||\Pi^{(f)} D_N f||^2 > C l(N)/N] \leq \Pp \Big[ \sum_{i=1}^{2 k(N)} f^2_i > \frac{ C l(N)}{KN} \Big] .\]
		Since the distribution of $\sum_{i=1}^{2 k(N)} f^2_i$ is a Beta law of parameter $\frac{\beta}{2} ( 2 k(N), N - 2 k(N))$, whose density on $[0,1]$ is given by: 
		\[ \frac{\Gamma(\frac{\beta N}{2})}{\Gamma(\beta k(N)) \Gamma( \frac{\beta}{2}(N - 2 k(N)))}x^{\beta k(N) -1} (1- x)^{\frac{\beta}{2}(N - 2 k(N)) -1} \]
		there, let us deal with the term in $\Gamma$. Using Stirling's equivalent, we have that 
		
		\begin{eqnarray*}  \frac{\Gamma(\frac{\beta N}{2})}{\Gamma(\beta k(N)) \Gamma( \frac{\beta}{2}(N - 2 k(N)))} &=&  \Big( \frac{ 2 k(N)( N -2k(N))}{N} \Big)^2 \Big( \frac{N}{e} \Big)^{\frac{\beta N}{2}} \Big( \frac{(N- 2 k(N))}{e} \Big)^{- \frac{\beta (N- 2 k(N))}{2}} \Big( \frac{2k(N)}{e} \Big)^{-\beta k(N)} +o(1) \end{eqnarray*}
		And therefore
		\begin{eqnarray*}
			\ln \frac{\Gamma(\frac{\beta N}{2})}{\Gamma(\beta k(N)) \Gamma( \frac{\beta}{2}(N - 2 k(N)))} &=& \frac{1}{2} ( \ln(2k(N)) + \ln (N -2 k(N)) - \ln N ) \\
			& & \quad+ \frac{\beta}{2} ( N \ln N  - 2 k(N) \ln(2k(N)) - (N - 2k(N)) \ln (N - 2k(N))) + O(1) \\
			&=& \frac{1}{2} ( \ln(2k(N)) - \frac{2k(N)}{N} + o( \frac{k(N)}{N} ))  \\
			& & \quad+ \frac{\beta}{2} ( 2 k(N) \ln N  - 2 k(N) \ln(2k(N)) - \frac{\beta}{2}(N - 2k(N)) \ln (1 - \frac{2k(N)}{N})) +O(1) \\
			&=& \frac{\beta}{2} ( 2k(N) \ln N  - \beta k(N) \ln(2k(N)) + 2 k(N) + o(k(N)) 
		\end{eqnarray*}

	  Since for $x^* = \frac{ \beta k(N)  -1}{ \beta N -2}$, $x \mapsto x^{\beta k(N) -1} (1- x)^{\frac{\beta}{2}(N - 2 k(N)) -1}$ is increasing on $[0, x^*]$ and decreasing on $[x^*,1]$, for $C > 1$, we have for $N$ large enough:
		\[ \Pp[ ||\Pi^{(f)} D_N f||^2 > C l(N)/N] \leq \frac{\Gamma(\frac{\beta N}{2})}{\Gamma(\beta k(N)) \Gamma( \frac{\beta}{2}(N - 2 k(N)))} \Big( \frac{C l(N)}{N} \Big)^{\beta k(N) -1}\Big( 1 - \frac{C l(N)}{N} \Big)^{\frac{\beta}{2}(N - 2 k(N)) -1} \]
		and therefore 
		\begin{eqnarray*} \ln \Pp[ ||\Pi^{(f)} D_N f||^2 > C l(N)/N] &\leq& \frac{\beta}{2} ( 2k(N)( \ln( C l(N)) - \ln( 2 k(N)) )+ \beta k(N)  - \frac{\beta}{2} Cl(N) + o(\max(k(N),l(N)))) 
			\\
		\end{eqnarray*}  
		
		Then, using that $2 k(N) \leq \l(N)$, we have that $2 k(N)\ln \frac{l(N)}{2k(N)} \leq e^{-1} l(N)$ and therefore:
		\begin{eqnarray*} 
			\ln \Pp[ ||\Pi^{(f)} D_N f||^2 > C l(N)/N] &\leq& \frac{\beta}{2} ( 2k(N) \ln(\frac{C l(N)}{2 k(N)})+2k(N)- C l(N)) + o(l(N)) \\
			& \leq & \frac{\beta}{2}( e^{-1}   + \ln C + (1 - C)) l(N) + o(l(N) )
		\end{eqnarray*}
		Choosing $C$ large enough such that 
		\[  \frac{\beta}{2}( e^{-1}   + \ln C + (1 - C)) <  -2 \]
		we have using $l(N) \geq \ln N$ that 
		\[ \Pp[ ||\Pi^{(f)} D_N f||^2 > C l(N)/N] = o(N ^{-2}) \]
		
		Going back to $(U^* D_N U)_{1,2}$, we have that 
		\[\Pp[ |\sqrt{N}(U^* D_N U)_{1,2}| \geq \epsilon] \leq \Pp[ B_N \geq \frac{\epsilon^2}{ C l(N) }] + o(N^{-2})\]
		where $B_N$ is some Beta variable of parameters $\frac{\beta}{2}(1,N-1)$. With the same estimation as for $||\Pi^{(f)} D_N f||^2$, one gets
		
		\[  \ln \Pp[ B_N \geq \frac{\epsilon^2}{ C l(N) }] \leq \frac{\beta}{2} ( \ln N - \frac{N \epsilon^2}{ C l(N) }) + o(\max(\frac{N }{l(N) }, \ln N))\]
		Using that $k(N) = o \Big( \frac{N}{\ln N} \Big)$, we have $ l(N) =  o \Big( \frac{N}{\ln N} \Big)$ and therefore $\ln N = o( \frac{N}{l( N)})$ whice gives then $\Pp[ B_N \geq \frac{\epsilon^2}{ C l(N) }] = o(N^{-2})$ and therefore 
		$\Pp[ |\sqrt{N}(U^* D_N U)_{1,2}| \geq \epsilon] = o(N^{-2})$. The lemma then follows by a simple union bound. 
	\end{proof}
	We now have all the ingredients to prove Lemma \ref{excessmass}. 
	\begin{proof}[Proof of Lemma \ref{excessmass}]
		We split $A_N^{(\epsilon)}$ into two terms: 
		
		\[ A^{(\epsilon)}_N:= \frac{1}{k(N)}\Big[ \sum_{i \neq j}  \mathds{1}_{ \beta\sqrt{N}| ( U^* D_N U)_{i,j}|/2 \geq \epsilon} |( U^* D_N U)_{i,j}|^2 + \sum_{i }  \mathds{1}_{ \beta\sqrt{N}| ( U^* D_N U)_{i,i}|/2 \geq \epsilon} |( U^* D_N U)_{i,i}|^2 \Big]\]
		Following Lemma \ref{excessmass2}, the first term is equal to $0$ with probability $1 - o(1)$. The second term can be bounded as follows: 
		
		\begin{eqnarray*}
			\sum_{i }  \mathds{1}_{ \beta\sqrt{N}| ( U^* D_N U)_{i,i}|/2 \geq \epsilon} |( U^* D_N U)_{i,i}|^2 &\leq & \sum_{i }   |( U^* D_N U)_{i,i}|^2 \\
			& \leq& \max_{j=1}^N |( U^* D_N U)_{j,j}| \sum_{i }  |( U^* D_N U)_{i,i}| \\
			& \leq & \max_{j=1}^N |( U^* D_N U)_{j,j}| \sum_{i }  ( U^* |D_N| U)_{i,i} \\
			& \leq& \max_{j=1}^N |( U^* D_N U)_{j,j}| \tr(|D_N|) \\
			& \leq&2 K k(N) \max_{j=1}^N |( U^* D_N U)_{j,j}| 
		\end{eqnarray*}
		where $|D_N|$ is the diagonal matrix whose entries are the $|\theta_{\pm i}^N|$. We used here that $|(U D_N U)_{i,i}| \leq (U |D_N| U)_{i,i}$ and that $|\theta^N_{\pm i} | \leq K$. From this bound and Lemma \ref{excessmass1}, this second term divided by $k(N)$ converges in probability toward $0$. Therefore the Lemma is proved. 
	\end{proof}
	
	\section{Proof of lemma \ref{tiltexptight}}
	To prove this exponential tightness lemma, we will first need the following result:
	\begin{lemma}
		For $C >0$, 
		
		\[\E[ \mathds{ 1}_{ I_N(X_N,D_N) \geq \exp(Nk(N)C)} I_N(X_N,D_N) ] \leq \exp( \frac{N k(N)}{2} ( (5\beta/2) K^2 - C)) \]
		
	\end{lemma}
	\begin{proof}
		In this proof, we will the sharp sub-Gaussian character of $X_N$ to do the following bound for any $A \in \mathcal{H}_N^{\beta}$:
		\[ \E[ \exp(Tr(AX_N)) ] \leq \exp( \frac{\beta}{4} Tr(A^2)) \]
		
		First we use Cauchy Schwartz inequality:
		\[\E[ \mathds{ 1}_{ I_N(X_N,D_N) \geq \exp(Nk(N)C)} I_N(X_N,D_N) ]	\leq \sqrt{ \Pp[ \mathds{ 1}_{ I_N(X_N,D_N) \geq \exp(Nk(N)C)}] \E[ I_N(X_N,D_N)^2]} \]
		By Markov inequality, we have: 
		
		\begin{eqnarray*} \E[ \mathds{ 1}_{ I_N(X_N,D_N) \geq \exp(Nk(N)C)}] &\leq& \E[ I_N(X_N,D_N)] \exp( - N k(N) C ) \\
			& \leq & \E_U [ \E_X [ \exp \Big( \frac{\beta N}{2} Tr( X_N U D_N U^* \Big) ]]\exp( - N k(N) C ) \\
			&\leq& \exp( N \frac{\beta}{2} Tr(D_N^2))\exp( - N k(N) C ) \\
			&\leq& \exp( N k(N) ( \frac{\beta}{2} K^2 - C)) \end{eqnarray*}
		
		If $U,U'$ are two independant random matrices Haar-distributed in $\mathcal{U}_N^{\beta}$:
		
		\begin{eqnarray*} \E[ I_N(X_N,D_N)^2] &=& \E_{U,U'}[ \E_X[ \exp \Big( \frac{\beta N }{2} \Big(\tr( X_N U D_N U^*) + \tr(X_N U' D_N U'^*)\Big) \Big)]] \\
			& =&\E_{U,U'}[ \E_X[ \exp \Big( \frac{\beta N }{2} \tr( X_N (U D_N U^* + U'D_N U'^*) )\Big)]] \\
			&\leq & \E_{U,U'}[ \exp( N \frac{\beta}{4} \tr(  (U D_N U^* + U'D_N U'^*)^2))] \\
			&\leq &  \exp( N \beta 2 k(N) K^2) .
		\end{eqnarray*}
		
		Here we used that $	 \tr(  (U D_N U^* + U'D_N U'^*)^2) \leq 4 \tr(D_N^2)\leq 8 K^2 k(N)$.
		
	\end{proof}
	Therefore, if $E$ is some event, we have that 
	\begin{eqnarray*} 
		\Pp^{\theta_N}[ E] &\leq& \E[ \exp( N I_N(X_N,D_N)) \mathds{1}_{E}] \exp( - \frac{\beta}{2} N k(N)( K^2 + o(1))) \\
		& \leq & (\E[ \mathds{1}_E \mathds{1}_{ I_N(D_N,X_N) \leq \exp(N k(N) C)}  I_N(X_N,D_N) ] + \exp( \frac{N k(N)}{2} ( (5\beta/2) K^2 - C)))
		\\
		& & \quad \exp( - \frac{\beta}{2} N k(N)( K^2 + o(1))) \\
		&\leq& \Pp[E] \exp ( N k(N)( C - \frac{\beta}{2} K^2 + o(1) )) + \exp(\frac{N k(N)}{2} ( (3\beta/2) K^2 - C +o(1))) .
	\end{eqnarray*}
	Then using Proposition \ref{exptight}  as well as Assumption \ref{assum2} for the measure $\Pp$, we prove Lemma \ref{tiltexptight}.

\end{document}